\documentclass[12pt,a4paper]{article}
\usepackage{euscript,amsfonts,amssymb,amsmath,amscd,amsthm, ytableau}

\usepackage{graphicx}

\usepackage{color}

\usepackage[T2A]{fontenc}

\usepackage{hyperref}

\usepackage[affil-it]{authblk}

\newtheorem{theorem}{Theorem}[section]
\newtheorem*{theorem*}{Theorem}

\newtheorem{definition}[theorem]{Definition}
\newtheorem{proposition}[theorem]{Proposition}
\newtheorem{conjecture}[theorem]{Conjecture}
\newtheorem{corollary}[theorem]{Corollary}
\newtheorem{exercise}[theorem]{Exercise}

\numberwithin{equation}{section}

\renewcommand{\P}{\mathbb P}

\newcommand{\E}{\mathbb E}

\newcommand{\Z}{\mathbb Z}

\newcommand{\R}{\mathbb R}

\newcommand{\Conf}{{\rm Conf}}

\newcommand\sgn{{\operatorname{sgn}}}

\newcommand{\X}{\mathfrak X}

\newcommand{\Y}{\mathbb Y}

\renewcommand{\i}{{\mathbf i}}

\title{Lectures on integrable probability}
\author{Alexei Borodin\thanks{Department of Mathematics, Massachusetts Institute of Technology,
USA, and
 Institute for Information Transmission Problems of Russian Academy of Sciences,  Russia}
\footnote{e-mail: borodin@math.mit.edu} \quad Vadim Gorin$^*$\footnote{e-mail:
vadicgor@math.mit.edu}}
\date{}

\begin{document}
\maketitle
\begin{abstract}
 These are lecture notes for a mini-course given at the St.\ Petersburg School in Probability and Statistical Physics
 in June 2012. Topics include integrable models of random growth,
 determinantal point processes, Schur processes and Markov dynamics on them,
 Macdonald processes and their application to asymptotics of directed polymers in random media.
\end{abstract}

\section*{Preface}

These lectures are about probabilistic systems that can be analyzed by essentially algebraic
methods.

The historically first example of such a system goes back to De Moivre (1738) and Laplace (1812)
who considered the problem of finding the asymptotic distribution of a sum of i.i.d. random
variables for Bernoulli trials, when the pre-limit distribution is explicit, and took the limit of
the resulting expression. While this computation may look like a simple exercise when viewed from
the heights of modern probability, in its time it likely served the role of a key stepping stone
--- first rigorous proofs of central limit theorems appeared only in the beginning of the XXth
century.

At the moment we are arguably in a ``De Moivre-Laplace stage'' for a certain class of stochastic
systems which is often referred to as the \emph{KPZ universality class}, after an influential work
of Kardar-Parisi-Zhang in mid-80's. We will be mostly interested in the case of one space
dimension. The class includes models of random growth that have built-in mechanisms of smoothing
and lateral growth, as well as directed polymers in space-time uncorrelated random media and
driven diffusive lattice gases.

While the class and some of its members have been identified by physicists, the first examples of
convincing (actually, rigorous) analysis were provided by mathematicians, who were also able to
identify the distributions that play the role of the Gaussian law. Nowadays, they are often
referred to as the \emph{Tracy-Widom type distributions} as they had previously appeared in
Tracy-Widom's work on spectra of large random matrices.

The reason for mathematicians' success was that there is an unusually extensive amount of algebra
and combinatorics required to gain access to suitable pre-limit formulas that admit large time
limit transitions. As we will argue below, the ``solvable'' or \emph{integrable} members of the
class should be viewed as projections of much more powerful objects whose origins lie in
representation theory. In a way, this is similar to integrable systems that can also be viewed as
projections of representation theoretic objects; and this is one reason we use the words
\emph{integrable probability} to describe the phenomenon. There are also much more direct links
between integrable systems and integrable probability some of which we mention below.

The goal of these notes is not to give a survey of a variety of integrable probabilistic models
that are known by now (there are quite a few, and \emph{not all} of them are members of the KPZ
universality class), but rather to give a taste of them by outlining some of the algebraic
mechanisms that lead to their solution.

The notes are organized as follows.

In Section \ref{Section_Intro} we give a brief and non-exhaustive overview of the integrable
members of the KPZ universality class in (1+1) dimensions.

 In Section \ref{Section_symmetric} we
provide the basics of the theory of symmetric functions that may be seen as a language of the
classical representation theory.

In Section \ref{Section_determinantal} we discuss determinantal random point processes --- a
fairly recent class of point processes that proved very effective for the analysis of growth
models and also for the analysis of integrable probabilistic models of random matrix type.

Section \ref{Section_RSK} explains a link between a particular random growth model (known as the
polynuclear growth process or PNG) and the so-called Plancherel measures on partitions that
originates from representation theory of symmetric groups.

In Section \ref{Section_Schur_measures} we introduce a general class of the Schur measures that
includes the Plancherel measures; members of these class can be viewed as determinantal point
processes, which provides a key to their analysis. We also perform such an analysis in the case of
the Plancherel measure, thus providing a proof of the celebrated Baik-Deift-Johansson theorem on
asymptotics of longest increasing subsequences of random permutations.

Section \ref{Section_Schur_process} explains how integrable models of stochastic growth can be
constructed with representation theoretic tools, using the theory of symmetric functions developed
earlier.

In Section \ref{Section_Macdonald} we show how one can use celebrated Macdonald symmetric
functions to access the problem of asymptotic behavior of certain directed random polymers in
(1+1) dimensions (known as the O'Connell-Yor polymers). The key feature here is that the formalism
of determinantal processes does not apply, and one needs to use other tools that here boil down to
employing Macdonald-Ruijsenaars difference operators.

\section{Introduction}

\label{Section_Intro}

Suppose that you are building a tower out of unit blocks. Blocks are falling from the sky, as
shown at Figure \ref{Fig_falling} (left picture) and the tower slowly grows. If you introduce
randomness here by declaring the times between arrivals of blocks to be independent identically
distributed (i.i.d.) random variables, then you get the simplest 1d random growth model. The kind
of question we would like to answer here is what the height $h(t)$ of tower at time $t$ is?

The classical central limit theorem (see e.g.\ \cite[Chapter 5]{Bi} or \cite[Chapter 4]{Kal})
provides the answer:
$$
 h(t)\approx  c_1^{-1} t  + \xi c_2 c_1^{-\frac32} t^{\frac12},
$$
where $c_1$ and $c_2$ are the mean and standard deviation of the times between arrivals of the
blocks, respectively, and $\xi$ is a standard normal random variable $N(0,1)$.

\begin{figure}[h]
\begin{center}
\hfill{\scalebox{0.9}{\includegraphics{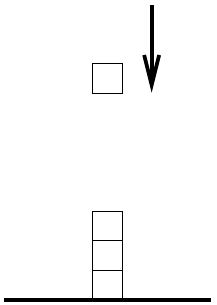}}}\hfill
{\scalebox{0.9}{\includegraphics{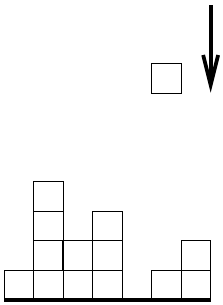}}}\hfill
{\scalebox{0.9}{\includegraphics{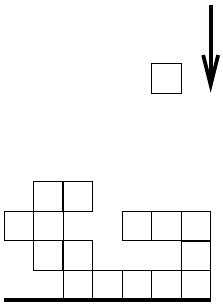}}}\hfill\, \caption{Growth models based on the
falling boxes: 1d model, 2d model without interactions, 2d model with sticky boxes (ballistic
deposition)  \label{Fig_falling} }
\end{center}

\end{figure}

If blocks fall independently in different columns, then we get a 2d growth model, as shown at
Figure \ref{Fig_falling} (middle picture). When there are no interactions between blocks and the
blocks are aligned, the columns grow independently and fluctuations remain of order $t^{1/2}$. But
what happens if we make blocks sticky so that they get glued to the boxes of adjacent columns, as
shown at Figure \ref{Fig_falling} (right picture)? This model is known as \emph{ballistic
deposition} and, in general, the answer for it is unknown. However, computer simulations (see
e.g.\ \cite{BaSt}) show that the height fluctuations in this model are of order $t^{1/3}$,
 and the same happens when the interaction is introduced in various other ways. Perhaps, there is also some form of a
Central Limit Theorem for this model, but nobody knows how to prove it.

\smallskip

Coming back to the 1d case, one attempt to guess the central limit theorem would be through
choosing certain very special random variables. If the times between arrivals are geometrically
distributed random variables, then $h(t)$ becomes the sum of independent Bernoulli random
variables and the application of the Stirling's formula proves the convergence of rescaled $h(t)$
to the standard Gaussian. (This is the famous De Moivre--Laplace theorem.)

In the 2d case it is also possible to introduce particular models for which we can prove
something. Consider the interface which is a broken line with slopes $\pm 1$, as shown at Figure
\ref{Fig_broken} (left picture) and suppose that a new unit box is added at each local minimum
independently after an exponential waiting time.

\begin{figure}[h]
\begin{center}
\hfill{\scalebox{1.0}{\includegraphics{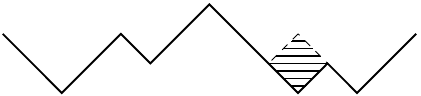}}}\hfill
{\scalebox{1.0}{\includegraphics{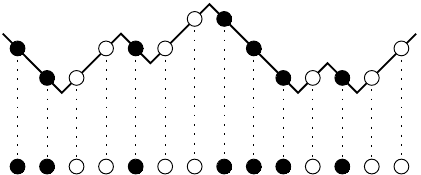}}}\hfill\,
 \caption{Broken line with slopes $\pm 1$, local minimum where a box can be added, and correspondence with particle
 configurations on $\mathbb Z$.
   \label{Fig_broken} }
\end{center}

\end{figure}

There is also an equivalent formulation of this growth model. Project the interface to a straight
line and put ``particles'' at projections of unit segments of slope $-1$ and ``holes'' at
projections of segments of slope $+1$, see Figure \ref{Fig_broken} (right picture). Now each
particle independently jumps to the right after an exponential waiting time (put it otherwise,
each particle jumps with probability $dt$ in each very small time interval $[t,t+dt]$) except for
the exclusion constraint: Jumps to the already occupied spots are prohibited. This is a simplified
model of a one-lane highway which is known under the name of Totally Asymmetric Simple Exclusion
Process (TASEP), cf.\ \cite{Spitzer}, \cite{L1}, \cite{L2}.

\begin{figure}[h]
\begin{center}
\hfill{\scalebox{1.0}{\includegraphics{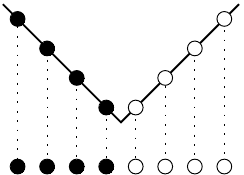}}}\hfill
{\scalebox{1.0}{\includegraphics{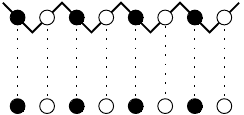}}}\hfill\,
 \caption{Wedge and flat initial conditions: broken lines and corresponding particle configurations.
   \label{Fig_broken_IC} }
\end{center}

\end{figure}

\begin{theorem}[\cite{J-TASEP}]
\label{Theorem_TW_in_TASEP_step}
 Suppose that at time $0$ the interface $h(x;t)$ is a wedge ($h(x,0)=|x|$) as shown at Figure \ref{Fig_broken_IC} (left
 picture).
 Then for every $x\in(-1,1)$
$$
 \lim_{t\to\infty}\P\left(\frac{h(t,tx)-c_1(x)t}{c_2(x) t^{1/3}}\ge -s\right) = F_2(s),
$$
where $c_1(x)$, $c_2(x)$ are certain (explicit) functions of $x$.
\end{theorem}

\begin{theorem}[\cite{S-TASEP}, \cite{BFPS-TASEP}]
\label{Theorem_TW_in_TASEP_flat}
 Suppose that at time $0$ the interface $h(x;t)$ is \emph{flat} as shown at Figure \ref{Fig_broken_IC} (right
 picture).
 Then for every $x\in\mathbb R$
$$
 \lim_{t\to\infty}\P\left(\frac{h(t,x)-c_3 t}{c_4 t^{1/3}}\ge -s\right) = F_1(s),
$$
where $c_3$, $c_4$ are certain (explicit) positive constants.
\end{theorem}

Here $F_1(s)$ and $F_2(s)$ are distributions from \emph{random matrix theory}, known under the
name of Tracy-Widom distributions. They are the limiting distributions for the largest eigenvalues
in Gaussian Orthogonal Ensemble and Gaussian Unitary Ensemble of random matrices (which are the
probability measures with density proportional to $\exp(-Trace(X^2))$ on real symmetric and
Hermitian matrices, respectively), see \cite{TW-U}, \cite{TW-O}.

These two theorems give the conjectural answer for the whole ``universality class'' of 2d random
growth models, which is usually referred to as the KPZ (Kardar-Parisi-Zhang) universality class.
Comparing to the answer in the 1d case we see that the asymptotic behavior becomes more delicate
--- while scaling by $t^{1/3}$ is always the same, the resulting distribution may also depend on
the ``subclass'' of our model. Also, conjecturally, the only two
generic subclasses are
the ones we have seen. They are distinguished by whether the global surface profile is locally
curved or flat near the observation location.

\bigskip

 Let us concentrate on the wedge initial
condition. In this case there is yet another reformulation of the model. Write in each box $(i,j)$
of the positive quadrant a random ``waiting time'' $w_{(i,j)}$. Once our random interface (of type
pictured in Figure \ref{Fig_broken}) reaches the box $(i,j)$ it takes time $w_{(i,j)}$ for it to
absorb the box. Now the whole quadrant is filled with nonnegative i.i.d. random variables. How to
reconstruct the growth of the interface from these numbers? More precisely, at what time $T(i,j)$
a given box $(i,j)$ is absorbed by the growing interface? A simple argument shows that
\begin{equation}
\label{eq_LPP}
 T(i,j) =\max\limits_{(1,1)=b[1]\to b[2]\to\dots b[i+j-1]=(i,j)} \sum_{k=1}^{i+j-1} w_{b[k]},
\end{equation}
where the sum is taken over all directed (leading away from the origin) paths joining $(1,1)$ and
$(i,j)$, see Figure \ref{Fig_directed_paths}. The quantity \eqref{eq_LPP} is known as the
(directed) Last Passage Percolation time. Indeed, if you think about numbers $w_{(i,j)}$ as of
times needed to percolate into a given box, then \eqref{eq_LPP} gives the time to percolate from
$(1,1)$ in $(i,j)$ in the worst case scenario.

\begin{figure}[h]
\begin{center}
{\scalebox{1.2}{\includegraphics{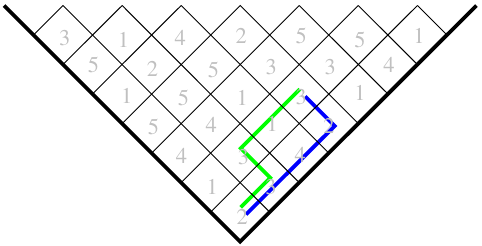}}}
 \caption{The quadrant filled with waiting times and two (out of $ {4\choose 1} = 4$ possibilities) directed paths joining $(1,1)$ and $(4,1)$.
   \label{Fig_directed_paths} }
\end{center}
\end{figure}

Universality considerations make one believe that the limit behavior of the Last Passage
Percolation time should not depend on the distribution of $w_{(i,j)}$ (if this distribution is not
too wild), but we are very far from proving this at the moment. However, again, the Last Passage
Percolation time asymptotics has been computed for certain distributions, e.g.\ for the
exponential distribution in the context of Theorem \ref{Theorem_TW_in_TASEP_step}.

Let us present another example, where the (conjecturally, universal) result can be rigorously
proven. Consider the homogeneous, density 1 Poisson point process in the first quadrant, and let
$L(\theta)$ be the maximal number of points one can collect along a North-East path from $(0,0)$
to $(\theta,\theta)$, as shown at Figure \ref{Fig_directed_Path_in_poisson}.

\begin{figure}[h]
\begin{center}
{\scalebox{1.0}{\includegraphics{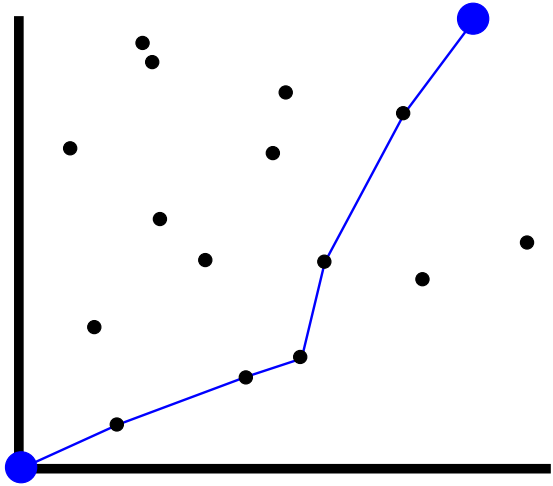}}}
 \caption{The Poisson point process in the first quadrant and a North--East path joining $(0,0)$ and $(\theta,\theta)$ and
 collecting maximal number of points, which is $5$ here.
   \label{Fig_directed_Path_in_poisson} }
\end{center}

\end{figure}

 This quantity can be
seen as a limit of the LPP times when $w_{(i,j)}$ takes only two values $0$ and $1$, and the
probability of $1$ is very small. Such considerations explain that $L(\theta)$ should be also in
the KPZ universality class. And, indeed, this is true.

\begin{theorem}[\cite{BDJ}]\label{Theorem_first_row_plancherel}
$$
\lim_{\theta\to\infty} \P\left(\frac{L(\theta)-2\theta}{\theta^{1/3}}\le s \right) = F_2(s).
$$
\end{theorem}

It is not hard to show that Theorem \ref{Theorem_first_row_plancherel} is equivalent to

\begin{theorem}[\cite{BDJ}]
\label{Theorem_longest_increasing}
 Let $\sigma$ be a uniformly distributed permutation of the set $\{1,\dots,n\}$, and let
 $\ell_n(\sigma)$ be the length of the longest increasing subsequence of $\sigma$. Then
$$
\lim_{n\to\infty} \P\left(\frac{\ell_n-2\sqrt{n}}{n^{1/6}}\le s \right) = F_2(s).
$$
\end{theorem}
The problem of understanding the limit behavior of $\ell_n$ has a long history and goes back to
the book of Ulam of 1961 \cite{Ulam}. Ulam conjectured that $\E \ell_n\approx c\sqrt{n}$ but was
not able to identify the constant; he also conjectured Gaussian fluctuations. In
1974 Hammersley
\cite{Ham} proved, via a sub--additivity argument, that there exists a constant such that
$\ell_n\approx c\sqrt{n}$ and this constant was identified in 1977  by Kerov and Vershik
\cite{VK}.

The random variable $\ell_n$ has an interesting interpretation in terms of an airplane boarding
problem. Imagine a simplified airplane with one seat in each of $n$ rows,
large distances between
rows, and one entrance in front. Each entering passenger has a ticket with a
seat number, but the
order of passengers in the initial queue is random (this is our random
permutation). Suppose that
each passenger has a carry-on, and it takes one minute for that person to load
it into the overhead
bin as soon as (s)he reaches her/his seat. The aisle is narrow, and nobody can
pass
 the passenger who is loading the carry-on.  It turns out that the total time to board the
airplane is precisely $\ell_n$. Let us demonstrate this with an example.

Consider
the permutation
$\sigma=2413$ with $\ell_4(\sigma)=2$. The airplane boarding looks as follows:
The first passenger
enters the airplane and proceeds to the seat number $2$. While (s)he loads a
carry-on, the other passengers stay
behind and the one with the ticket for the seat number $1$ ((s)he was the third
person in the original queue) starts loading her/his
carry-on. After one minute, the passenger with the ticket for the seat number
$4$ proceeds to his seat and also starts loading,
as well as the one aiming for the seat number $3$. In two minutes the boarding
is complete.

Interestingly enough, if
 the queue is divided into groups, as often happens in reality, then the boarding time (for
 long queues)
will only increase by the factor $\sqrt{k}$, where $k$ is the number of the groups.

\medskip

Let us now proceed to more recent developments. In the Last Passage Percolation problem we were
maximizing a functional $H(x)$ over a set $\mathcal X$. A general statistical mechanics principle
says that such a maximization can be seen as zero-temperature limit of the \emph{Gibbs ensemble}
on $\mathcal X$ with Hamiltonian $-H(x)$. More formally, we have the following essentially obvious
statement
$$
 \max_{x\in\mathcal X} H(x) =\lim_{\beta\to\infty} \frac{1}{\beta} \ln \sum_{x\in\mathcal X}
 e^{\beta H(x)}.
$$
The parameter $\beta$ is usually referred to as \emph{the inverse temperature} in the statistical
mechanics literature.

In the Last Passage Percolation model, $\mathcal X$ is the set of all directed paths joining
$(1,1)$ with a point $(a,b)$, and the value of $H$ on path $x$ is the sum of $w_{(i,j)}$ along the
path $x$. The Gibbs ensemble in this case is known under the name of a ``Directed Polymer in
Random Media''. The study of such objects with various path sets and various choices of noise
(i.e.\ $w_{(i,j)}$) is a very rich subject.

Directed Polymers in Random Media appeared for the first time close to thirty years ago in an
investigation of low temperature expansion of the partition function of the Ising model with
domain wall boundary conditions, see \cite{HH}, \cite{IS}, but nowadays there are many other
physical applications. Let us give one concrete model where such polymers arise.

Consider a set of massive particles in $\mathbb Z$ that evolve in discrete time as follows. At
each time moment the mass of each particle is multiplied by a random variable $d_{t,x}$, where $t$
is the time moment and $x$ is the particle's position. Random variables $d_{t,x}$ are typically
assumed to be i.i.d. Then each particle gives birth to a twin
 of the same mass and the twin moves to $x+1$. If we now start at time $0$ with a single particle
 of mass $1$ at $x=1$, then the mass $Z(T,x)$ of all particles at $x$ at time $T$ can be computed as a sum
 over all directed paths $(1,1)=b[1]\to b[2]\to\dots
b[x+T-1]=(T,x)$ joining $(1,1)$ and $(T,x)$:
\begin{equation}
\label{eq_parition_function_gamma_polymer} Z(T,x)=\sum\limits_{(1,1)=b[1]\to b[2]\to\dots
b[x+T-1]=(T,x)} \prod_{k=1}^{x+T-1} d_{b[k]},
\end{equation}
This model can be used as a simplified description for the migration of plankton with $d_{t,x}$
representing the state of the ocean at location $x$ and time $t$ which affects the speed of growth
of the population. Independent $d_{t,x}$ model quickly changing media, e.g.\ due
to the turbulent
flows in the ocean.

Random Polymers in Random Media exhibit a very interesting phenomenon called intermittency which
is the existence of large peeks happening with small probability, that are high enough to dominate
the asymptotics of the moments.  Physicists believe that intermittency is widespread in nature
and, for instance, the mass distribution in the universe or a magnetogram of the sun show
intermittent behavior. To see this phenomenon in our model, suppose for a moment that $d_{t,x}$
does not depend on $t$. Then there would be locations where the amount of plankton exponentially
grows, while in other places all the plankton quickly dies, so we see very high peaks. Now it is
reasonable to expect that such peaks would still be present when $d_{t,x}$ are independent both of
$t$ and $x$ and this will cause intermittency. Proving and quantifying intermittency is, however,
rather difficult.

Regarding the distribution of $Z(T,x)$, it was long believed in the physics literature that it
should belong to the same KPZ universality class as the Last Passage Percolation. Now, at least in
certain cases, we can prove it. The following integrable random polymer was introduced and studied
by Sepp\"{a}l\"{a}inen \cite{Sep} who proved the $t^{\frac{1}{3}}$ exponent for the fluctuations.
The next theorem is a refinement of this result.

\begin{theorem}[\cite{BCR}] \label{Theorem_polymer_intro}
Assume $d_{t,x}$ are independent positive random variables with density
$$
 \frac{1}{\Gamma(\theta)} x^{-\theta-1} \exp\left(-\frac1x\right).
$$
Then there exist $\theta^*>0$ and (explicit) $c_1, c_2>0$ such that for
$0<\theta<\theta^*$,
$$
 \lim_{n\to\infty} \P\left( \frac{Z(n,n)-c_1 n}{c_2 n^{1/3}}\le s \right)= F_2(s).
$$
\end{theorem}
The upper bound on the parameter $\theta>0$ in this theorem is technical
and it will probably be removed in future works.

In a similar way to our transition from Last Passage Percolation to monotone paths in a Poisson
field and longest increasing subsequences, we can do a limit transition here, so that discrete
paths in \eqref{eq_parition_function_gamma_polymer} turn into Brownian bridges,
while $d_{t,x}$
turn into the space--time white noise. Let us explain in more detail how this works as this will
provide a direct link to the Kardar--Parisi--Zhang equation that gave the name to the KPZ
universality class.

For a Brownian bridge $B=B(s)$ we obtain a functional
\begin{equation}
\label{eq_partition_function_cont_polymer}
 H(B) =\int {\beta \dot W (s,B(s))} ds,
\end{equation}
where $\dot W$ is the 2d white noise. Thus, the partition function $Z(t,x)$ has the form
\begin{equation}
 Z(t,x)=\frac{1}{\sqrt{2\pi t}}
\exp\left(-\frac{x^2}{2t}\right)\E\big(:\exp:(H(B)\big),
\end{equation}
where $\E$ is the expectation with respect to the law of the Brownian bridge which starts at $0$
at time $0$ and ends at $x$ at time $t$, and $:\exp:$ is the Wick ordered exponential, see
\cite{AKQ} and references therein for more details. Note that the randomness coming from the white
noise is still there, and $Z(t,x)$ is a random variable.

Another way of defining $Z(t,x)$ is through the stochastic PDE it satisfies:
\begin{equation}
\label{eq_stochastic_heat}
 \frac{\partial}{\partial t} Z(t,x) =\frac{1}{2}\left(\frac{\partial}{\partial x}\right)^2 Z(t,x) + \dot W Z.
\end{equation}
 This is known as the \emph{stochastic heat equation}. Indeed, if we remove the part with the white noise
 in \eqref{eq_stochastic_heat}, then we end up with the usual heat equation.

If the space (corresponding to the variable $x$) is discrete, then an equation similar to
\eqref{eq_stochastic_heat} is known as the \emph{parabolic Anderson model}; it has been
extensively studied for many years.

Note that through our approach the solution of \eqref{eq_stochastic_heat} with $\delta$--initial
condition at time $0$ is the limit of discrete $Z(t,x)$ of
\eqref{eq_parition_function_gamma_polymer} and, thus, we know something about it.

If we now define $U$ through the so--called Hopf--Cole transformation
$$
 Z(x,t)=\exp(U(x,t))
$$
then, as a corollary of \eqref{eq_stochastic_heat}, $U$ formally satisfies
\begin{equation}
\label{eq_KPZ}
 \frac{\partial}{\partial t} U(t,x) =\frac{1}{2}\left(\frac{\partial}{\partial x}\right)^2 U(t,x) + \left(\frac{\partial}{\partial x} U(t,x)\right)^2 +\dot
 W,
\end{equation}
which is the \emph{non-linear} Kardar--Parisi--Zhang (KPZ) equation introduced in \cite{KPZ}
 as a way of understanding the growth of surfaces we started with (i.e.\ ballistic
deposition), see \cite{C-KPZ} for a nice recent survey.

Due to non-linearity of \eqref{eq_KPZ} it is tricky even to give a  meaning to this equation (see,
however, \cite{Hai} for a recent progress), but physicists still dealt with it and that's one way
how the  exponent $1/3$ of $t^{\frac13}$ was predicted. (An earlier way was through dynamical
renormalization group techniques, see \cite{FNS}.)

\bigskip

If we were to characterize the aforementioned results in one phrase, we would use ``integrable
probability''. ``Integrable'' here refers to explicit formulas that can be derived, and also hints
at parallels with integrable systems. There are direct connections, e.g.\ $y(s)$ defined via
$$
 y^2(s)=-(\ln F_2(s))''
$$
solves the (nonlinear) Painleve II differential equation (see \cite{TW-U})
$$
 y''(s)=sy(s)+2y(s)^3.
$$
Also if we define
$$
 F(x_1,\dots,x_n;t)=\E\left( Z(x_1,t)\cdots Z(x_n,t)\right),
$$
where $Z(t,x)$ is the solution of Stochastic Heat Equation \eqref{eq_stochastic_heat}, then
\begin{equation}
\label{eq_q_Bose}
 \frac{\partial }{\partial t} F=\frac{1}{2} \left( \sum_{i=1}^n \left(\frac{\partial}{\partial
 x_i}\right)^2 +\sum_{i\ne j} \delta(x_i-x_j)\right) F,
\end{equation}
where $\delta$ is the Dirac delta--function. \eqref{eq_q_Bose} is known as the evolution equation
of the quantum delta-Bose gas. It was the second quantum many body system solved
via Bethe ansatz,
see \cite{LL}, \cite{Mcg}.

There is also a deeper analogy: Both integrable systems and integrable probability models can be
viewed as shadows of representation theory of infinite--dimensional Lie groups and algebras.
However, while integrable PDEs often represent rather exotic behavior from the point of view of
general PDEs, integrable probability delivers universal behavior for the whole universality class
of similar models. Moreover, in the rare occasions when the universality can be proved (e.g.\ in
random matrices, see recent reviews \cite{EY}, \cite{TV} and references therein, or in $(1+1)$d
polymers in the so-called intermediate disorder regime, see \cite{AKQ_disorder}), one shows that
the generic behavior is the same as in the integrable case. Then the integrable case provides the
only known route to an explicit description of the answer.

While we will not explain in these notes the representation theoretic undercurrent in any detail,
we cannot and do not want to get rid of it completely. In what follows we will rely on the theory
of symmetric functions which is the algebraic--combinatorial apparatus of the representation
theory.

\subsection*{Acknowledgments} We are very grateful to Ivan Corwin, Grigori Olshanski, and Leonid
Petrov for very valuable comments on
an earlier version
of this text. A.~B. was partially supported by the NSF grant DMS-1056390. V.~G.
was partially supported by  RFBR-CNRS grants 10-01-93114 and 11-01-93105.

\section{Symmetric functions}
\label{Section_symmetric}

In this section we briefly review certain parts of the theory of symmetric functions. If the
reader is familiar with the basics of this theory, (s)he might want to skip this section returning
to it, if necessary, in the future. There are several excellent treatments of symmetric functions
in the literature, see e.g.\ \cite{M}, \cite{Sagan}, \cite{St_book}. We will mostly follow the
notations of \cite{M} and recommend the same book for the proofs.

\medskip

Our first aim is to define the algebra $\Lambda$ of symmetric functions in \emph{infinitely many}
variables. Let $\Lambda_N = \mathbb C[x_1, \dots, x_N]^{S_N}$ be the space of polynomials in $x_1,
\dots, x_N$ which are symmetric with respect to permutations of the $x_j$. $\Lambda_N$ has a
natural grading by the \emph{total degree} of a polynomial.

Let $\pi_{N+1} \colon \mathbb C[x_1, \dots, x_{N+1}] \to \mathbb C[x_1, \dots, x_N]$ be the map
defined by setting $x_{N+1} = 0$. It preserves the ring of symmetric polynomials and gradings.
Thus we obtain a tower of graded algebras
\[
 \mathbb C \xleftarrow{\pi_1} \Lambda_1 \xleftarrow{\pi_2} \Lambda_2 \xleftarrow{\pi_3}
\dots.
\]
We define $\Lambda$ as the \emph{projective limit} of the above tower
\[
\Lambda = \varprojlim_N \Lambda_N = \{(f_1, f_2, f_3, \dots) \mid f_j \in \Lambda_j, \ \pi_j f_j =
f_{j-1},\, \deg(f_j) \text{ are bounded }\}.
\]
An equivalent definition $\Lambda$ is as follows: Elements of $\Lambda$ are formal power series
$f(x_1, x_2,\dots)$ in infinitely many indeterminates $x_1, x_2,\dots$ of bounded degree that are
invariant under the permutations of the $x_i$'s. In particular,
$$
 x_1+x_2+x_3+\dots
$$
is an element of $\Lambda$, while
$$
 (1+x_1)(1+x_2)(1+x_3)\cdots
$$
is not, because here the degrees are unbounded.

\emph{Elementary symmetric functions} $e_k$, $k=1,2,\dots$ are defined by
$$
 e_k=\sum\limits_{i_1<i_2<\dots<i_k} x_{i_1}\cdots x_{i_k}.
$$

\emph{Complete homogeneous functions} $h_k$, $k=1,2,\dots$ are defined by
$$
 h_k=\sum\limits_{i_1\le i_2\le \dots\le i_k} x_{i_1}\cdots x_{i_k}.
$$

\emph{Power sums} $p_k$, $k=1,2,\dots$ are defined by
$$
 p_k=\sum\limits_i x_i^k.
$$

\begin{theorem}
\label{Theorem_fundamental_sym}
 The systems $\{e_k\}$, $\{h_k\}$, $\{p_k\}$ are algebraically independent generators of $\Lambda$. In other words,
 $\Lambda$ can be seen as the algebra of polynomials in $h_k$, or the algebra of polynomials in
 $e_k$, or the algebra of polynomials in $p_k$:
 $$
 \Lambda  =\mathbb C [e_1,e_2,\dots]=\mathbb C [h_1,h_2,\dots] =\mathbb C [p_1,p_2,\dots].
$$
\end{theorem}
The proof of this statement can be found in \cite[Chapter I, Section 2]{M}. Theorem
\ref{Theorem_fundamental_sym} for the polynomials in finitely many variables is known as \emph{the
fundamental theorem of symmetric polynomials}.

\smallskip

It is convenient to introduce generating functions for the above generators:
$$
 H(z):=\sum_{z=0}^{\infty} h_k z^k, \quad E(z):=\sum_{z=0}^{\infty} e_k z^k,\quad P(z):=\sum_{z=1}^{\infty} p_k
 z^{k-1},
$$
where we agree that $h_0=e_0=1$.

\begin{proposition} We have
\label{Prop_gen_functions}
\begin{equation}
\label{eq_H_gen}
 H(z)=\prod_{i} \frac{1}{1-x_i z},
\end{equation}
\begin{equation}
\label{eq_E_gen}
 E(z)=\prod_{i} (1+x_i z),
\end{equation}
\begin{equation}
\label{eq_P_gen}
 P(z)=\frac{d}{dz} \sum_{i} \ln\frac{1}{1-x_i z}.
\end{equation}
In particular,
\begin{equation}
\label{eq_gen_functions_relation}
 H(z)=\frac{1}{E(-z)}=\exp\left(\sum_{k=1}^{\infty} \frac{z^k p_k}{k}\right).
\end{equation}
\end{proposition}
\begin{proof}
 In order to prove \eqref{eq_E_gen} open the parentheses and compare with the definition of
 $e_k$. To prove \eqref{eq_H_gen} note that
 $$
  \prod_{i} \frac{1}{1-x_i z}=\prod_i(1+x_i z+ x_i^2 z^2+ x_i^3 z^3+\dots)
 $$
and open the parentheses again. Finally, using the power series expansion of the logarithm we get
\begin{multline*}
\frac{d}{dz} \sum_{i} \ln\frac{1}{1-x_i z}=\frac{d}{dz} \sum_i \left(x_iz+ \frac{x_i^2
z^2}{2}+\frac{x_i^3 z^3}{3}+\dots\right)\\= \sum_i \left(x_i+ x_i^2 z+ x_i^3
z^2+\dots\right)=\sum_{k=1}^{\infty} p_k z^{k-1}. \end{multline*}
\end{proof}

\begin{figure}[h]
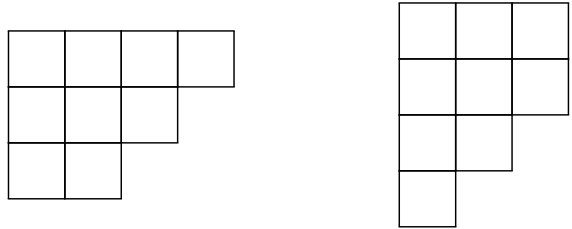

\centerline{ \large \ytableausetup{centertableaux}
\begin{ytableau}
 \,&\,  & \, &\,  \\
 \,&\,  &\,  \\
\, &\,
\end{ytableau}
\hskip 2cm
\begin{ytableau}
\, & \, & \, \\
\, & \, & \, \\
\, & \, \\
\,
\end{ytableau}
} \caption{Left panel: Young diagram $\lambda$ of size $9$ with row lengths $(4,3,2)$ and column
lengths $(3,3,2,1)$. Right panel: Transposed diagram $\lambda'$. \label{Figure_young_diag}}
\end{figure}

 Let $\lambda$ be a
\emph{Young diagram} of size $n$ or, equivalently, a \emph{partition} of $n$. In other words,
$\lambda$ is a sequence $\lambda_1\ge\lambda_2\ge\dots$ of non-negative integers (which are
identified with \emph{row lengths} of the Young diagram), such that $\sum_i
\lambda_i=|\lambda|=n$. The diagram whose row lengths are column lengths of $\lambda$ is called
\emph{transposed diagram} and denoted $\lambda'$. In other words, for each $i$, $\lambda'_i$ is
equal to the number of $j$ such that $\lambda_j\ge i$. We draw Young diagrams as collections of
unit boxes and Figure \ref{Figure_young_diag} gives an example of a Young diagram and its
transpose.

The \emph{length} $\ell(\lambda)$ of $\lambda$ is defined as the number of non-zero numbers
$\lambda_i$ (equivalently, the number of rows in $\lambda$). Clearly, $\ell(\lambda)=\lambda'_1$.

We denote the set of all Young diagrams by $\Y$. By definition $\Y$ includes the empty partition
$\varnothing=(0,0,\dots)$.

\begin{definition} The \emph{Schur polynomial} $s_\lambda(x_1,\dots,x_N)$ is a symmetric polynomial in $N$
variables parameterized by Young diagram $\lambda$ with $\ell(\lambda)\le N$ and given by
\begin{equation}
\label{eq_Schur} s_\lambda(x_1,\dots,x_N)= \dfrac{\det
\left[x_i^{\lambda_j+N-j}\right]_{i,j=1}^N}{\prod_{i<j}
 (x_i-x_j)}.
\end{equation}
\end{definition}
One proves that when $\ell(\lambda)\le N$
$$
 \pi_{N+1} s_\lambda(x_1,\dots,x_N,x_{N+1})= s_\lambda(x_1,\dots,x_N,0)=s_\lambda(x_1,\dots,x_N).
$$
In addition,
$$
 \pi_{\ell(\lambda)} s_\lambda(x_1,\dots,x_{\ell(\lambda)})=0.
$$
Therefore, the sequence of symmetric polynomials $s_\lambda(x_1,\dots,x_N)$ with fixed $\lambda$
and varying number of variables $N\ge\ell(\lambda)$, complemented by zeros for $N<\ell(\lambda)$,
defines an element of $\Lambda$ that one calls the \emph{Schur symmetric function} $s_\lambda$. By
definition $s_\varnothing(x)\equiv 1$.

\begin{proposition}
\label{proposition_prop_Schur_basic} The Schur functions $s_\lambda$, with $\lambda$ ranging over
the set of all Young diagrams, form a linear basis of $\Lambda$. They are related to generators
$e_k$ and $h_k$ through the \emph{Jacobi--Trudi} formulas:
$$
s_\lambda=\det\Bigl[h_{\lambda_i-i+j}\Bigr]_{i,j=1,\dots,\ell(\lambda)}
 =\det\Bigl[e_{\lambda'_i-i+j}\Bigr]_{i,j=1,\dots, \ell(\lambda')},
$$
where we agree that $h_k=e_k=0$ for $k<0$.
\end{proposition}
The proof of this statement can be found in \cite[Chapter I, Section 3]{M}. The expression of
$s_\lambda$ through generators $p_k$ is more involved and is related to the table of
\emph{characters of irreducible representations} of the symmetric groups, see \cite[Chapter I,
Section 7]{M}.

Now suppose that we have two copies of the algebra $\Lambda$ or, in other words, two sets of
variables $x=(x_1,x_2,\dots)$ and $y=(y_1,y_2,\dots)$. We can consider functions of the form
$s_\lambda(x)s_\mu(y)$, which will be symmetric functions in variables $x$ and $y$ separately, but
not jointly; formally such function can be viewed as an element of the tensor product $\Lambda
\otimes \Lambda$. More generally, we can consider an infinite sum
\begin{equation}
\label{eq_Cauchy_sum}
 \sum_{\lambda} s_\lambda(x)s_\lambda(y),
\end{equation}
 as an infinite series symmetric in variables $x_1,x_2,\dots$ and in variables $y_1,y_2,\dots$.
 The following theorem gives a neat formula for the sum \eqref{eq_Cauchy_sum}.

 \begin{theorem}[The Cauchy identity]
\label{theorem_Cauchy}
 We have
\begin{equation}
\label{eq_Schur_Cauchy}
 \sum_{\lambda\in\Y} s_\lambda(x_1,x_2,\dots) s_\lambda(y_1,y_2,\dots)=\prod_{i,j} \frac{1}{1-x_i y_j},\quad
\end{equation}
and also
 \begin{multline}
 \label{eq_p_Cauchy}
 \sum_{\lambda\in\Y} \frac{p_\lambda(x_1,x_2,\dots) p_\lambda(y_1,y_2\dots)}{z_\lambda}
\\ =
 \exp\left(\sum\limits_{k=1}^\infty \frac{p_k(x_1,x_2,\dots)p_k(y_1,y_2,\dots)}{k}\right)=
 \prod_{i,j} \frac{1}{1-x_i y_j},
\end{multline}
where
$$
 p_\lambda=p_{\lambda_1}p_{\lambda_2}\cdots p_{\lambda_{\ell(\lambda)}}
$$
and
 $ z_\lambda=\prod_{i\ge 1} i^{m_i}  m_i!, $ where $m_i(\lambda)$ is the number of rows of length $i$
in $\lambda$.
 \end{theorem}
\noindent {\bf Remark.} The right--hand sides of \eqref{eq_Schur_Cauchy} and \eqref{eq_p_Cauchy}
should be viewed as formal power series via
$$
\frac{1}{1-x_i y_j}=1+x_i y_j +(x_i y_j)^2+ (x_i y_j)^3+\dots.
$$
The proof of Theorem \ref{theorem_Cauchy} can be found in \cite[Chapter I, Section 4]{M}. In fact,
\eqref{eq_Schur_Cauchy} is a particular case of the more general \emph{skew Cauchy identity}. We
need to introduce further notations in order to state it.

Take two sets of variables $x=(x_1,x_2,\dots)$ and $y=(y_1,y_2,\dots)$ and a symmetric function
$f\in\Lambda$. Let $(x,y)$ be the union of sets of variables $x$ and $y$. Then we can view
$f(x,y)$ as a function in $x_i, y_j$ symmetric with respect to all possible permutations of
variables. In particular, $f(x,y)$ is a symmetric function in $x_i$ and also a symmetric function
in $y_i$, more precisely, $f(x,y)$ is a sum of products of symmetric functions of $x_i$ and
symmetric functions of $y_i$. What does this decomposition look like? The answer, of course,
depends on $f$. For instance,
$$
 p_k(x,y)=\sum_i x_i^k+\sum_i y_i^k =p_k(x)+p_k(y).
$$
\begin{definition}
\label{Definition_skew_Schur} Let $\lambda$ be any Young diagram. Expand $s_\lambda(x,y)$ as a
linear combination of Schur symmetric functions in variables $y_i$; the coefficients of this
expansion are called \emph{skew Schur functions} and denoted $s_{\lambda/\mu}$:
$$
 s_{\lambda}(x,y)=\sum_{\mu} s_{\lambda/\mu}(x) s_\mu(y).
$$
In particular, $s_{\lambda/\mu}(x)$ is a symmetric function in variables $x_i$.
\end{definition}

\begin{proposition}[The skew Cauchy identity] For any Young diagrams $\lambda$, $\nu$ we have
\label{proposition_skew_Cauchy}
\begin{multline}
\label{eq_skew_Cauchy}
 \sum_{\mu\in\Y} s_{\mu/\lambda}(x_1,x_2,\dots) s_{\mu/\nu}(y_1,y_2,\dots) \\= \prod_{i,j}\frac{1}{1-x_i y_j}
  \sum_{\kappa\in\Y} s_{\lambda/\kappa}(y_1,y_2,\dots) s_{\nu/\kappa}(x_1,x_2,\dots).
\end{multline}
\end{proposition}
For the proof of this statement see \cite[Chapter I, Section 5, Example 26]{M}. In order to see
that \eqref{eq_Schur_Cauchy} is indeed a particular case of \eqref{eq_skew_Cauchy} we need the
following generalization of Jacobi--Trudi identity (its proof can be found in the same section of
\cite{M}).

\begin{proposition} Assuming that $h_k=0$ for $k<0$, we have \label{Prop_J-T_skew}
$$
 s_{\lambda/\mu} = \det
 \Bigl[h_{\lambda_i-\mu_j-i+j}\Bigr]_{i,j=1,\dots,\max(\ell(\lambda),\ell(\mu))}.
$$
In particular, $s_{\lambda/\mu}=0$ unless $\mu\subset\lambda$, i.e.\ $\mu_i\le\lambda_i$ for all
$i$.
\end{proposition}
Comparing Proposition \ref{Prop_J-T_skew} with Proposition \ref{proposition_prop_Schur_basic} we
conclude that if $\mu=\varnothing$ is the empty Young diagram, then (one can also see this
independently from definitions)
$$
 s_{\lambda/\mu}=s_\lambda,
$$
and also
$$
 s_{\mu/\nu}=s_{\varnothing/\nu}=\begin{cases} 1, & \nu=\varnothing,\\ 0,&\text{otherwise.} \end{cases}
$$
 Now if we set $\lambda=\nu=\varnothing$ in \eqref{eq_skew_Cauchy} we get
\eqref{eq_Schur_Cauchy}.

Another property of the skew Schur functions is summarized in the following proposition (its proof
can be found in \cite[Chapter I, Section 5]{M}).
\begin{proposition} \label{proposition_fundamental_skew_Schur}Let $x$ and $y$ be two sets of variables.
For any $\lambda,\mu\in\Y$ we have
$$
 s_{\lambda/\mu}(x,y)=\sum_{\nu\in\Y} s_{\lambda/\nu}(x) s_{\nu/\mu}(y).
$$
\end{proposition}

Another important notion is that of a \emph{specialization}.
\begin{definition}
Any algebra homomorphism $\rho :\Lambda\to \mathbb C,\quad f\mapsto f(\rho)$,  is called a
 \emph{specialization}. In other words, $\rho$ should satisfy the following properties:
 $$
 (f+g)(\rho)=f(\rho)+g(\rho),\quad
  (fg)(\rho)=f(\rho)g(\rho),\quad   (\theta f)(\rho)=\theta f(\rho), \quad \theta\in \mathbb C.
 $$
\end{definition}

Take any sequence of complex numbers $u_1,u_2,\dots$ satisfying $\sum |u_i|<\infty$. Then the
\emph{substitution} map $\Lambda\to \mathbb C,\,\,  x_i\mapsto u_i$ is a specialization. More
generally, any specialization is uniquely determined by its values on any set of generators of
$\Lambda$. Furthermore, if the generators are algebraically independent, then these values can be
any numbers. What this means is that defining $\rho$ is equivalent to specifying the set of
numbers $p_1(\rho)$, $p_2(\rho)$,\dots, \emph{or} the set of numbers $e_1(\rho)$,
$e_2(\rho)$,\dots, \emph{or} the set of numbers $h_1(\rho)$, $h_2(\rho)$,\dots. In particular, if
$\rho$ is the substitution of complex numbers $u_i$, then
\begin{equation}
\label{eq_substitution}
 p_k\mapsto p_k(\rho)=\sum_{i} (u_i)^k.
\end{equation}
Note that the condition $\sum_i |u_i|<\infty$ implies that the series in \eqref{eq_substitution}
converges for any $k\ge 1$.

Sometimes it is important to know which specializations are \emph{positive} in a certain sense. We
call a specialization $\rho$ \emph{Schur--positive} if for every Young diagram $\lambda$ we have
$$
 s_\lambda(\rho)\ge 0.
$$
There is an explicit classification for Schur--positive specializations.
\begin{theorem}
\label{Theorem_Thoma} The Schur--positive specializations are parameterized by pairs of sequences
of non-negative reals $\alpha=(\alpha_1\ge\alpha_2\ge\dots\ge 0)$ and
$\beta=(\beta_1\ge\beta_2\ge\dots\ge 0)$ satisfying $\sum_i(\alpha_i+\beta_i)<\infty$ and an
additional parameter $\gamma\ge 0$. The specialization with parameters $(\alpha;\beta;\gamma)$ can
be described by its values on power sums
$$
p_1\mapsto p_1(\alpha;\beta;\gamma)=\gamma+\sum_i (\alpha_i+\beta_i),
$$
$$
 p_k\mapsto
p_k(\alpha;\beta;\gamma)=\sum\limits_{i}
 (\alpha_i^k + (-1)^{k-1} \beta_i^{k}), \, \quad k\ge 2
$$
or, equivalently, via generating functions
$$
\sum\limits_{k=0}^\infty h_k(\alpha;\beta;\gamma) z^k = e^{\gamma z} \prod\limits_{i\ge 1}
\dfrac{1+\beta_i z} {1-\alpha_i
 z}.
$$
\end{theorem}
\noindent{\bf Remark.} One can show that if $\rho$ is a Schur--positive specialization, then also
$s_{\lambda/\mu}(\rho)\ge 0$ for any $\lambda,\mu\in\Y$.

\smallskip

Theorem \ref{Theorem_Thoma} has a number of equivalent reformulations, in particular, it is
equivalent to the description of all characters of the infinite symmetric group $S(\infty)$ and to
the classification of \emph{totally positive} triangular Toeplitz matrices. The first proofs of
Theorem \ref{Theorem_Thoma} were obtained (independently) by Thoma \cite{Thoma1} and Edrei
\cite{Ed}, a proof by a different method can be found in \cite{VK_S}, \cite{Kerov_book},
\cite{KOO}, and yet another proof is given in \cite{Ok}.

\medskip

Our next goal is to study the simplest Schur--positive specializations more thoroughly. Given two
Young diagrams $\lambda$ and $\mu$ we say that $\lambda/\mu$ is a \emph{horizontal strip} if $0\le
\lambda'_i-\mu'_i \le 1$ for all $i$; $\lambda/\mu$ is a \emph{vertical strip} if $0\le
\lambda_i-\mu_i \le 1$ for all $i$.

 A \emph{semistandard Young tableau} of shape $\lambda$ and
rank $N$ is a filling of boxes of $\lambda$ with numbers from $1$ to $N$ in such
a way that the
numbers strictly increase along the columns and weakly increase along the rows. A \emph{standard
Young tableau} of shape $\lambda$ is a filling of boxes of $\lambda$ with numbers from $1$ to
$|\lambda|$ in such a way that the numbers strictly increase both along the columns and along the
rows (in particular, this implies that each number appears exactly once). Examples of Young
tableaux are given in Figure \ref{Figure_young_tabl}.

\begin{figure}[h]
\centerline{ \large \ytableausetup{centertableaux}
\begin{ytableau}
1 & 1 & 2 &5 & 5 \\
2 & 3 & 4 \\
3 & 4
\end{ytableau}
\hskip 2cm
\begin{ytableau}
1 & 2 & 3& 5 \\
4 & 6 & 8 \\
7
\end{ytableau}
} \caption{Left panel: a semistandard Young tableau of shape $(5,3,2)$ and rank $5$. Right panel:
a standard Young tableau of shape $(4,3,1)$. \label{Figure_young_tabl}}
\end{figure}

The number of all semistandard Young tableaux of shape $\lambda$ and rank $N$ is denoted as ${\rm
Dim}_N(\lambda)$. The number of all standard Young tableau of shape $\lambda$ is denoted as
$\dim(\lambda)$. These quantities have representation--theoretic interpretations, namely, ${\rm
Dim}_N(\lambda)$ is the \emph{dimension} of the irreducible representation of unitary group $U(N)$
indexed by (the highest weight) $\lambda$ and ${\rm dim}(\lambda)$ is the dimension of the
irreducible representation of symmetric group $S(|\lambda|)$ indexed by $\lambda$.

The proofs of the following statements are a combination of Theorem \ref{Theorem_Thoma} and
results of \cite[Chapter I]{M}.

\begin{proposition}
\label{Proposition_spec_alpha} Suppose that $\alpha_1=c$ and all other
$\alpha$-, $\beta$-,
$\gamma$-parameters are zeros. Then for Schur--positive specialization
$(\alpha;\beta;\gamma)$,
$s_\lambda(\alpha;\beta;\gamma)=0$ unless $\lambda$ is a one--row Young diagram (i.e.\
$\ell(\lambda)=1$), and, more generally, $s_{\lambda/\mu}(\alpha;\beta;\gamma)=0$ unless
$\lambda/\mu$ is a horizontal strip. In the latter case
$$
 s_{\lambda/\mu}(\alpha;\beta;\gamma)=c^{|\lambda|-|\mu|}.
$$
\end{proposition}
\begin{proposition}
\label{Proposition_spec_beta} Suppose that $\beta_1=c$ and all other $\alpha$-,
$\beta$-,
$\gamma$-parameters are zeros. Then for the Schur--positive specialization
$(\alpha;\beta;\gamma)$, $s_\lambda(\alpha;\beta;\gamma)=0$ unless $\lambda$ is a one--column
Young diagram (i.e.\ $\lambda_1\le 1$), and, more generally,
$s_{\lambda/\mu}(\alpha;\beta;\gamma)=0$ unless $\lambda/\mu$ is a vertical strip. In the latter
case
$$
 s_{\lambda/\mu}(\alpha;\beta;\gamma)=c^{|\lambda|-|\mu|}.
$$
\end{proposition}

\begin{proposition}
\label{Proposition_spec_unitary_dim} Suppose that $\alpha_1=\alpha_2=\dots=\alpha_N=1$ and all
other $\alpha$-, $\beta$-,
$\gamma$-parameters  are zeros. Then  for the Schur--positive specialization
$(\alpha;\beta;\gamma)$
$$
s_\lambda(\alpha;\beta;\gamma)={\rm Dim}_N(\lambda), \quad \lambda\in\Y.
$$
\end{proposition}

\begin{proposition}
\label{Prop_plancherel} Suppose that $\gamma=c$ and all $\alpha$- and
$\beta$-parameters are
zeros. Then for the Schur--positive specialization $(\alpha;\beta;\gamma)$
$$
s_\lambda(\alpha;\beta;\gamma)=\frac{c^{|\lambda|}}{|\lambda|!}\, {\rm dim}(\lambda),\quad
\lambda\in\Y.
$$
\end{proposition}

\section{Determinantal point processes}

\label{Section_determinantal}

In this section we introduce \emph{determinantal point processes} which are an important tool in
the study of growth models and in the KPZ universality class.

Consider a reasonable ``state space'' or ``one particle space'' $\X$, say the real line $\R$, or
the Euclidean space $\R^d$, or a discrete space such as the set $\Z$ of integers or its subset.  A
\emph{point configuration} $X$ in $\X$ is a locally finite (i.e.\ without accumulation points)
collection of points of the space $\X$. For our purposes it suffices to assume that the points of
$X$ are always pairwise distinct. The set of all point configurations in $\X$ will be denoted as
$\Conf(\X)$.

A compact subset $A\subset \X$ is called \emph{a window}. For a window $A$ and $X\in\Conf(\X)$,
set $N_A(X)=|A\cap X|$ (number of points of $X$ in the window). Thus, $N_A$ is a function on
$\Conf(\X)$. We equip $\Conf(\X)$ with the Borel structure (i.e.\ $\sigma$--algebra) generated by
functions $N_A$ for all windows $A$.

A \emph{random point process} on $\X$ is a probability measure on $\Conf(\X)$. We will often use
the term \emph{particles} for the elements of a random point configuration. Thus, we will speak
about particle configurations.

The most known example of a random point process is the homogeneous (rate 1) Poisson process on
$\R$. For any finite interval $A\subset \R$ (or, more generally, for any compact set $A$), the
number $N_A$ of particles falling in $A$ is finite because, by the very assumption, $X$ has no
accumulation points. Since $X$ is random, $N_A$ is random, too. Here are the key properties of the
Poisson process (see e.g.\ \cite[Section 23]{Bi} or \cite[Chapter 10]{Kal}):

\begin{itemize}
\item  $N_A$ has the Poisson distribution with parameter $|A|$, the length of $A$. That is
$$
\P(N_A=n)=e^{-|A|}\frac{|A|^n}{n!}, \qquad n=0,1,2,\dots
$$
\item If $A_1$,\dots,$A_k$ are pairwise disjoint intervals, then the corresponding random
variables $N_{A_1},\dots,N_{A_k}$ are independent. This means that the particles do not interact.
\end{itemize}

The Poisson process can be constructed as follows. Let $M=1,2,3,\dots$ be a natural number. Take
the interval $[-M/2,M/2]$ and place $M$ particles in it, uniformly and independently of each
other. Observe that the mean density of the particles is equal to 1 for any $M$ because the number
of particles and the length of the interval are the same. Now pass to the limit as $M\to\infty$.
As $M$ gets large, the interval approximates the whole real line, and in the limit one obtains the
Poisson random configuration.

\begin{exercise} Assuming that the limit process exists, show that it satisfies the above two properties
concerning the random variables $N_A$.
\end{exercise}

The above simple construction contains two important ideas: First, the idea of limit transition.
Starting from $M$--particle random configurations one can get infinite particle configurations by
taking a limit. Second, the observation that the limit transition may lead to a simplification.
Indeed, the structure of the joint distribution of $N_{A_1},\dots,N_{A_k}$ simplifies in the
limit.

\bigskip

Let us construct a discrete analog of the Poisson process. Replace the real line $\R$ by the
lattice $\Z$ of integers. This will be our new state space. A particle configuration on $\Z$ is
simply an arbitrary subset $X\subset\Z$, the assumption of absence of accumulation points holds
automatically.

Fix a real number $p\in(0,1)$. The stationary Bernoulli process with parameter $p$ on $\Z$ is
constructed as follows: For each integer $n\in\Z$ we put a particle at the node $n$ with
probability $p$, independently of other nodes. This procedure leads to a random particle
configuration.

Equivalently, the Bernoulli process is a doubly infinite sequence $\xi_n$, $n\in\Z$, of binary
random variables, such that each $\xi_n$ takes value 1 or 0 with probability $p$ or $1-p$,
respectively, and, moreover, these variables are independent. Then the random configuration $X$
consists of those $n$'s for which $\xi_n=1$.

The following construction is a simple example of a \emph{scaling limit transition}. Shrink our
lattice by the factor of $p$. That is, instead of $\Z$ consider the isomorphic lattice
$p\,\Z\subset\R$ with mesh $p$, and transfer the Bernoulli process to $p\,\Z$. The resulted scaled
Bernoulli process can be regarded as a process on $\R$ because $p\,\Z$ is contained in $\R$, and
each configuration on $p\,\Z$ is simultaneously a configuration on $\R$. As $p$ goes to 0, the
scaled Bernoulli process will approximate the Poisson process on the line. This is intuitively
clear, because for Bernoulli, like Poisson, there is no interaction, and the mean density of
particles for the scaled Bernoulli process is equal to 1.

\bigskip

How to describe a point process? The problem here comes from the fact that the space of particle
configurations $\Conf(\X)$ is, typically, infinite--dimensional and, thus, there is no natural
``Lebesgue measure'' which could be used for writing densities. One solution is to use
\emph{correlation functions}.

Let us temporarily restrict ourselves to point processes on a finite or countable discrete space
$\X$ (for instance, the reader may assume $\X=\Z$). Such a process is the same as a collection
$\{\xi_x\}$ of binary random variables, indexed by elements $x\in\X$, which indicate the presence
of a particle at $x$. (They are often called \emph{occupancy variables}.) The Bernoulli process
was a simple example. Now we no longer assume that these variables are independent and identically
distributed; they may have an arbitrary law. Their law is simply an arbitrary probability measure
$\P$ on the space of all configurations. This is a large space, it can be described as the
infinite product space $\{0,1\}^\X$. Thus, defining a point process on $\X$ amounts to specifying
a probability measure $\P$ on $\{0,1\}^\X$.

\begin{definition}
 Let $A$ range over \emph{finite} subsets of $\X$. The
correlation function of a point process $X$ on $\X$ is the function $\rho(A)$ defined by
$$
\rho(A)=\P(A \subset X).
$$
\end{definition}

If $A$ has $n$ points, $A=\{x_1,\dots,x_n\}$, then we also employ the alternative notation
$$
\rho(A)=\rho_n(x_1,\dots,x_n).
$$
In this notation, the single function $\rho$ splits into a sequence of functions
$\rho_1,\rho_2,\dots$, where $\rho_n$ is a symmetric function in $n$ distinct arguments from $\X$,
called the \emph{$n$--point correlation function}.

Equivalently, in terms of occupation random variables $\{\xi_x\}$
$$
\rho_n(x_1,\dots,x_n)=\P(\xi_{x_1}=\dots=\xi_{x_n}=1).
$$

For example, for the Bernoulli process one easily sees that $\rho(A)=p^{|A|}$.

\begin{exercise}
Show that a random point process on a discrete set $\X$ is uniquely determined by its correlation
function.
\end{exercise}


Let us now extend the definition of the correlation functions to arbitrary, not necessarily
discrete state spaces $\X$.

Given a random point process on $\X$, one can usually define a sequence $\{\rho_n\}_{n=1}^\infty$,
where $\rho_n$ is a symmetric measure on $\X^n$ called the $n$th \emph{correlation measure}. Under
mild conditions on the point process, the correlation measures exist and determine the process
uniquely.

The correlation measures are characterized by the following property: For any $n\ge1$ and a
compactly supported bounded Borel function $f$ on $\X^n$, one has
$$
\int_{\X^n}f\rho_n=\E\left( \sum_{x_{i_1},\dots,x_{i_n}\in X}f(x_{i_1},\dots,x_{i_n}) \right)
$$
where $\E$ denotes averaging with respect to our point process, and the sum on the right is taken
over all $n$-tuples of pairwise distinct points of the random point configuration $X$.

In particular, for $n=1$ we have
$$
\int_{\X}f(x)\rho_1(dx)=\E\left( \sum_{x\in X} f(x)\right)
$$
and $\rho_1$ is often called the \emph{density measure} of a point process.

Often one has a natural measure $\mu$ on $\X$ (called \emph{reference measure}) such that the
correlation measures have densities with respect to $\mu^{\otimes n}$, $n=1,2,\dots\,$. Then the
density of $\rho_n$ is called the $n$th \emph{correlation function} and it is usually denoted by
the same symbol $\rho_n$.

\begin{exercise} Show that if $\X$ is discrete, then the previous definition of the $n$th correlation function
coincides with this one provided that $\mu$ is the counting measure (i.e. $\mu$ assigns weight 1
to every point of $\X$).
\end{exercise}

\begin{exercise}
 Show that the $n$th correlation function of homogeneous rate 1 Poisson process (with respect to Lebesgue measure)
 is identically equal to $1$:
 $$
 \rho_n\equiv 1,\quad n\ge 1.
 $$
\end{exercise}

If $\X\subset \R$ and $\mu$ is absolutely continuous with respect to the Lebesgue measure, then
the probabilistic meaning of the $n$th correlation function is that of the density of probability
to find a particle in each of the infinitesimal intervals around points $ x_1, x_2, \ldots x_n $:
\begin{multline*}
\rho_n(x_1, x_2, \ldots x_n) \mu(dx_1) \cdots \mu(dx_n)\\= \P \,( \text{there is a particle in
each interval}\ \, (x_i, x_i +dx_i)).
\end{multline*}

For a random point process with fixed number of particles, say $N$, described by a joint
probability distribution $P(dx_1,\dots,dx_N)$ (it is natural to assume that $P$ is symmetric with
respect to permutations of the arguments), the correlation measures for $n\le N$ are given by
\begin{equation}
\label{eq_cor_func_eval} \rho_n(dx_1,\dots,dx_n)=\frac{N!}{(N-n)!}\int_{x_{n+1},\dots,x_N\in\X}
P(dx_1,\dots,dx_N).
\end{equation}
Indeed, we have
\begin{multline*}
\E\left( \sum_{\substack{ x_{i_1},\dots,x_{i_n}\\
\text{pairwise distinct}}} f(x_{i_1},\dots,x_{i_n})\right)=\int_{\X^N}\sum_{\substack{
x_{i_1},\dots,x_{i_n}\\ \text{pairwise distinct}}} f(x_{i_1},\dots,x_{i_n})P(dx)\\=
\frac{N!}{(N-n)!}\int_{\X^N} f(x_{1},\dots,x_{n})P(dx).
\end{multline*}

 For $n>N$ the correlation measures $\rho_n$ vanish
identically.

Another important property of the correlation measures is obtained by considering the test
functions $f$ given by products of characteristic functions of a window $A\subset \X$. Then one
immediately obtains
$$
\E \bigl( N_A(N_A-1)\cdots(N_A-n+1)\bigr)=\int_{A^n}\rho_n(dx_1,\dots,dx_n).
$$

\begin{definition}
A point process on $\X$ is said to be \emph{determinantal} if there exists a function $K(x,y)$ on
$\X\times\X$ such that the correlation functions (with respect to some reference measure) are
given by the determinantal formula
$$
\rho_n(x_1,\dots,x_n)={\det[K(x_i,x_j)]}_{i,j=1}^n
$$
for all $n=1,2,\dots$. The function $K$ is called the \emph{correlation kernel}.
\end{definition}

That is,
\begin{gather*}
\rho_1(x_1)=K(x_1,x_1), \quad \rho_2(x_1,x_2)=\det\bmatrix K(x_1,x_1) &
K(x_1,x_2)\\ K(x_2,x_1) & K(x_2,x_2)\endbmatrix\\
\rho_3(x_1,x_2,x_3)=\det\bmatrix K(x_1,x_1) & K(x_1,x_2) &K(x_1,x_3),\\
K(x_2,x_1) & K(x_2,x_2) & K(x_2,x_3)\\ K(x_3,x_1) & K(x_3,x_2) & K(x_3,x_3)\endbmatrix, \quad
\text{etc}.
\end{gather*}

Note that the determinants in the right--hand side do not depend on the ordering of the arguments
$x_i$. Also the correlation kernel is not unique --- gauge transformations of the form
$$
 K(x,y)\mapsto \frac{f(x)}{f(y)} K(x,y)
$$
with a non-vanishing $f:\X\to\mathbb C$ do not affect the correlation functions.

 The correlation kernel is a single function of two variables while the correlation functions
form an infinite sequence of functions of growing number of variables. Thus, if a point process
happens to be determinantal, it can be described by a substantially reduced amount of data. This
is somehow similar to Gaussian processes, for which all the information about process is encoded
in a single covariance function.

\begin{exercise}
 Prove that the stationary Poisson process on $\R$ and the Bernoulli process are
determinantal. What are their correlation kernels?
\end{exercise}

Determinantal processes appeared in the 60s in the context of random matrix theory with first
example going back to the work of Dyson \cite{Dy}. As a class such processes were first
distinguished in 1975 by Macchi \cite{Mac}, who considered the case when correlation kernel is
Hermitian ($K(x,y)=\overline K(y,x)$) and called them ``Fermion processes''. In this case
$$
\rho_2(x,y)=\rho_1(x)\rho_1(y)-|K(x,y)|^2\le \rho_1(x)\rho_1(y)
$$
which, from a probabilistic point of view, means that particles repel.

The term ``determinantal point process'' first appeared in \cite{BO98} together with first natural
examples of such processes with \emph{non--Hermitian} kernels. Now this term is widespread and
there is even a wikipedia article with the same name.

Nowadays lots of sources of determinantal point processes are known including (in addition to  the
random matrix theory) dimers on bipartite graphs \cite{Ken08}, uniform spanning trees
\cite{Lyo03}, ensembles of non-intersecting paths on planar acyclic graphs \cite{J-paths} and
zeros of random analytic functions \cite{Peres_book}. A recent review of determinantal point
processes can be found in \cite{B-det}.

An important class of determinantal random point processes is formed by \emph{bi\-or\-thogonal
ensembles}.

\begin{definition}
\label{Definition_biorth}
 Consider a state space $\X$ with a reference measure $\mu$. An $N$--point
biorthogonal ensemble on $\X$ is a probability measure on $N$--point subsets $\{x_1,\dots,x_N\}$
of $\X$ of the form
$$
P_N(dx_1,\dots,dx_N)=c_N\cdot \det\left[\phi_i(x_j)\right]_{i,j=1}^N
\det\left[\psi_i(x_j)\right]_{i,j=1}^N \cdot \mu(dx_1)\cdots \mu(dx_N)
$$
for a normalization constant $c_N>0$ and functions $\phi_1,\psi_1,\dots,\phi_N,\psi_N$ on $\X$
such that all integrals of the form $\int_\X\phi_i(x)\psi_j(x)\mu(dx)$ are finite.
\end{definition}

Important examples of biorthogonal ensembles come from the random matrix theory. One case is the
measure on $N$--particle configurations on the unit circle with density proportional to
$$
 \prod_{i<j} |x_i-x_j|^2,
$$
which can be identified with the distribution of eigenvalues of the $N\times N$ random unitary
matrices; here we equip the unitary group $U(N)$ with the Haar measure of total mass 1.

Another case is the measure on $N$--particle configurations on $\mathbb R$ with density
proportional to
$$
 \prod_{i<j} (x_i-x_j)^2 \prod_i e^{-x_i^2},
$$
which is the distribution of eigenvalues of a random Hermitian matrix from Gaussian Unitary
Ensemble (GUE), see \cite{Me}, \cite{For}, \cite{AGZ}, \cite{ABF}.

\begin{theorem}\label{Theorem_biorth}
Any biorthogonal ensemble is a determinantal point process. Its correlation kernel has the form
$$
K(x,y)=\sum_{i,j=1}^N \phi_i(x)\psi_j(y) [G^{-t}]_{ij},
$$
where $G=[G_{ij}]_{i,j=1}^N$ is the Gram matrix:
$$
G_{ij}=\int_{\X}\phi_i(x)\psi_j(x)\mu(dx).
$$
\end{theorem}
\noindent {\bf Remark.} When the functions $\phi_i(x)$ and $\psi_j(y)$ are biorthogonal,  the Gram
matrix is diagonal and it can be easily inverted. This is the origin of the term ``biorthogonal
ensemble''.

\begin{proof}[Proof of Theorem \ref{Theorem_biorth}]
Observe that
\begin{multline*}
\int_{\X^N}\det\left[\phi_i(x_j)\right]_{i,j=1}^N \det\left[\psi_i(x_j)\right]_{i,j=1}^N \cdot
\mu(dx_1)\cdots \mu(dx_N)\\=\int_{\X^N}\left(\sum_{\sigma,\tau\in S(N)}\sgn(\sigma\tau)
\prod_{i=1}^N\phi_{\sigma(i)}(x_j)\psi_{\tau(i)}(x_j)\right)\mu(dx_1)\dots\mu(dx_N)\\=
\sum_{\sigma,\tau\in S(N)}\sgn(\sigma\tau)\prod_{i=1}^N G_{\sigma(i)\tau(j)}=N!\det G.
\end{multline*}
This implies that the normalization constant $c_N$ in the definition of the biorthogonal ensemble
above is equal to $(N!\det G)^{-1}$, and that the matrix $G$ is invertible.

We need to prove that $\rho_n(x_1,\dots,x_n)=\det[K(x_i,x_j)]_{i,j=1}^n$ for any $n\ge 1$. For
$n>N$ the statement is trivial as both sides vanish (the right-hand side vanishes because the
matrix under determinant has rank no more than $N$ due to the explicit formula for $K$).

Assume $n\le N$. By formula \eqref{eq_cor_func_eval} for the correlation functions,
\begin{multline*}
\rho_n(x_1,\dots,x_n)\\=\frac {N!c_N}{(N-n)!}\int_{\X^{N-n}}
\det\left[\phi_i(x_j)\right]_{i,j=1}^N \det\left[\psi_i(x_j)\right]_{i,j=1}^N \mu(dx_{n+1})\cdots
\mu(dx_N).
\end{multline*}
Let $A$ and $B$ be two $N\times N$ matrices such that $AGB^t=Id$. Set
$$
\Phi_k=\sum_{k=1}^N A_{kl}\phi_l,\qquad \Psi_k=\sum_{k=1}^N B_{kl}\psi_l,\qquad k=1,\dots,N.
$$
Then
$$
{\langle \Phi_i,\Psi_j\rangle}_{L^2(\X,\mu)}=[AGB^t]_{ij}=\delta_{ij}.
$$
On the other hand, for any $x_1,\dots,x_N$,
\begin{gather*}
\det\left[\Phi_i(x_j)\right]_{i,j=1}^N=\det
A\det\left[\phi_i(x_j)\right]_{i,j=1}^N,\\
\det\left[\Psi_i(x_j)\right]_{i,j=1}^N=\det B\det\left[\psi_i(x_j)\right]_{i,j=1}^N.
\end{gather*}
Also $AGB^t=Id$ implies $\det A \det B \det G=1$. Hence,  the formula for the correlation
functions can be rewritten as
\begin{multline*}
\rho_n(x_1,\dots,x_n)\\=\frac {1}{(N-n)!}\int_{\X^{N-n}} \det\left[\Phi_i(x_j)\right]_{i,j=1}^N
\det\left[\Psi_i(x_j)\right]_{i,j=1}^N \mu(dx_{n+1})\cdots \mu(dx_N).
\end{multline*}
Opening up the determinants and using the fact that $\Phi_i$'s and $\Psi_j$'s are biorthogonal, we
obtain
\begin{multline*}
\rho_n(x_1,\dots,x_n)=\frac {1}{(N-n)!}\sum_{\substack{ \sigma,\tau\in S(N)\\
\sigma(k)=\tau(k),\, k>n}}\sgn(\sigma\tau)
\prod_{i=1}^n\Phi_{\sigma(i)}(x_i)\Psi_{\tau(i)}(x_i)\\= \sum_{ 1\le j_1<\dots<j_n\le
N}\det\Phi^{j_1,\dots,j_n} \det\Psi^{j_1,\dots,j_n},
\end{multline*}
where $\Phi^{j_1,\dots,j_n}$ is the submatrix of $\Phi:=[\Phi_j(x_i)]_{\substack{ i=1,\dots,n\\
j=1,\dots,n}}$ formed by columns $j_1,\dots,j_n$, and similarly for $\Psi^{j_1,\dots,j_n}$. The
Cauchy-Binet formula now yields $\rho_n(x_1,\dots,x_n)=\det \Phi\Psi^t$, and
$$
[\Phi\Psi^t]_{ij}=\sum_{k,l,m=1}^N A_{kl}B_{km}\phi_l(x_i)\psi_m(x_j)=\sum_{l,m=1}^N
[A^tB]_{lm}\phi_l(x_i)\psi_m(x_j).
$$
The right-hand side is equal to $K(x_i,x_j)$ because $AGB^t=Id$ is equivalent to $A^tB=G^{-t}$.
\end{proof}

The appearance of biorthogonal ensembles in applications is often explained by the combinatorial
statement known as the Lindstr\"om-Gessel-Viennot (LGV) theorem, see \cite{Ste90} and references
therein, that we now describe.

Consider a finite\footnote{The assumption of finiteness is not necessary as long as the sums in
(\ref{eq_paths}) converge.} directed acyclic graph and denote by $V$ and $E$ the sets of its
vertices and edges. Let $w:E\to{\mathbb C}$ be an arbitrary weight function. For any path $\pi$
denote by $w(\pi)$ the product of weights over the edges in the path: $w(\pi)=\prod_{e\in \pi}
w(e)$. Define the weight of a collection of paths as the product of weights of the paths in the
collection (we will use the same letter $w$ to denote it). We say that two paths $\pi_1$ and
$\pi_2$ do not intersect (notation $\pi_1\cap\pi_2=\varnothing$) if they have no common vertices.

For any $u,v\in V$, let $\Pi(u,v)$ be the set of all (directed) paths from $u$ to $v$. Set
\begin{equation}\label{eq_paths}
\mathcal{T}(u,v)=\sum_{\pi\in\Pi(u,v)} w(\pi).
\end{equation}

\begin{theorem}\label{Theorem_LGV} Let $(u_1,\dots,u_n)$ and $(v_1,\dots,v_n)$ be two
$n$-tuples of vertices of our graph, and assume that for any nonidentical permutation $\sigma\in
S(n)$,
$$
\left\{(\pi_1,\dots,\pi_n)\mid \pi_i\in \Pi\left(u_i,v_{\sigma(i)}\right),\
\pi_i\cap\pi_j=\varnothing,\ i,j=1,\dots,n\right\}=\varnothing.
$$
Then
$$
\sum_{\substack{\pi_1\in \Pi(u_1,v_1),\dots, \pi_n\in \Pi(u_n,v_n)\\
\pi_i\cap \pi_j=\varnothing,\
i,j=1,\dots,n}}w(\pi_1,\dots,\pi_n)=\det\left[\mathcal{T}(u_i,v_j)\right]_{i,j=1}^n.
$$
\end{theorem}

Theorem \ref{Theorem_LGV} means that if, in a suitable weighted oriented graph, we have
nonintersecting paths with fixed starting and ending vertices, distributed according to their
weights, then the distribution of the intersection points of these paths with any chosen
``section'' has the same structure as in Definition \ref{Definition_biorth}, and thus by Theorem
\ref{Theorem_biorth} we obtain a determinantal point process. More generally, the distribution of
the intersection points of paths with finitely many distinct ``sections'' also form a
determinantal point process. The latter statement is known as the \emph{Eynard--Metha theorem},
see e.g.\ \cite{BR05} and references therein.

A continuous time analog of Theorem \ref{Theorem_LGV} goes back to \cite{Kar59}, who in particular
proved the following statement (the next paragraph is essentially a quotation).

Consider a stationary stochastic process whose state space is an interval on the extended real
line. Assume that the process has strong Markov property and that its paths are continuous
everywhere. Take $n$ points $x_1<\dots<x_n$ and $n$ Borel sets $E_1<\dots<E_n$, and suppose $n$
labeled particles start at $x_1,\dots,x_n$ and execute the process simultaneously and
independently. Then the determinant $\det\left[P_t(x_i,E_j)\right]_{i,j=1}^n$, with $P_t(x,E)$
being the transition probability of the process, is equal to the probability that at time $t$ the
particles will be found in sets $E_1,\dots,E_n$, respectively, without any of them ever having
been coincident in the intervening time.

Similarly to Theorem \ref{Theorem_LGV}, this statement coupled with Theorem \ref{Theorem_biorth}
(or more generally, with Eynard--Metha theorem) leads to determinantal processes.

\section{From Last Passage Percolation to Plancherel measure}

\label{Section_RSK}

The aim of this section is to connect the Last Passage Percolation with certain probability
measures on the set of Young diagrams often referred to as \emph{Plancherel measures for the
symmetric groups.}

\begin{figure}[h]
\begin{center}
{\scalebox{1.0}{\includegraphics{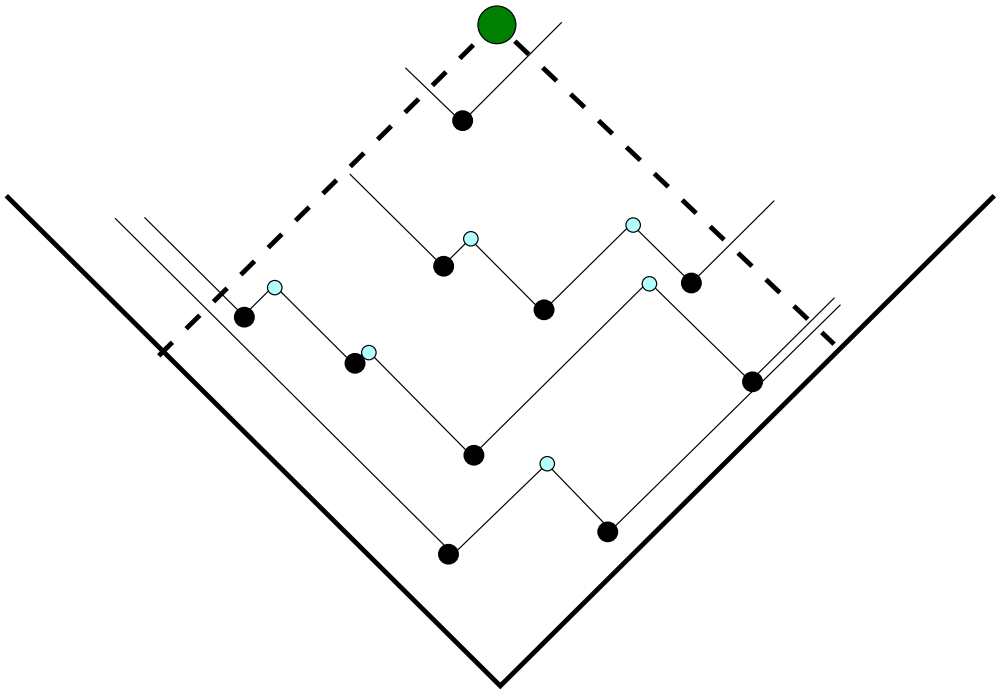}}} \caption{Points of Poisson process (in black),
broken lines joining them and points of the second generation (in light blue). Maximum number of
points collected along monotonous paths joining points $(\theta,\theta)$ (the big green point) and
$(0,0)$ coincides with number of broken lines separating them, which is $4$ in our case. Only the
points inside dashed square matter. \label{Fig_geo_RSK} }
\end{center}

\end{figure}

We start from the Poisson process in the first quadrant, as in Theorem
\ref{Theorem_first_row_plancherel} and Figure \ref{Fig_directed_Path_in_poisson}, but now we
rotate the quadrant by 45 degrees, like in Figure \ref{Fig_directed_paths}. There is a graphical
way to find the value of $L(\theta)$. Namely, for each point of the process draw two rays starting
from it and parallel to the axes. Extend each ray till the first intersection with another ray. In
this way, we get a collection of broken lines, as shown in Figure \ref{Fig_geo_RSK}. At the first
intersection points of the rays we put new points that form the second generation. Note now that
$L(\theta)$ is equal to the number of broken lines separating $(\theta,\theta)$ and the origin. As
it turns out, it is beneficial to iterate this process. We erase all the points of the original
Poisson process, but keep the points of the second generation and draw broken lines joining them;
we repeat this until no points inside the square with vertices $(0,0)$, $(0,\theta)$,
$(\theta,0)$, and $(\theta,\theta)$ are left, as shown in Figure \ref{Fig_geo_RSK2}. Compute the
number of broken lines separating $(\theta,\theta)$ and $(0,0)$ at each step and record these
numbers to form a Young diagram $\lambda(\theta)=(\lambda_1(\theta),\lambda_2(\theta),\dots)$, so
that, in particular, $\lambda_1(\theta)=L(\theta)$. Observe that $|\lambda(\theta)|$ equals the
number of points of the original Poisson process inside the square with vertices $(0,0)$,
$(0,\theta)$, $(\theta,0)$,  and $(\theta,\theta)$. The procedure we just described is known as
Viennot's geometric construction of the Robinson--Schensted correspondence, see e.g.\
\cite{Sagan}.

\begin{figure}[h]
\begin{center}
{\scalebox{0.58}{\includegraphics{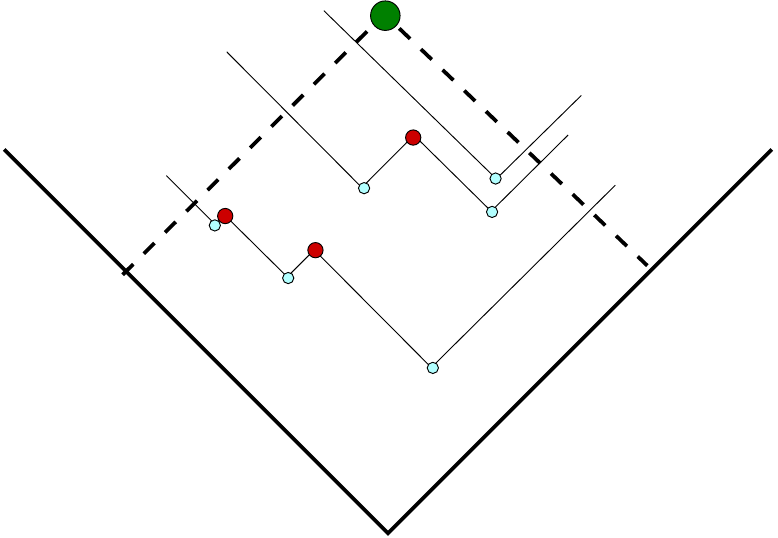}}} \hfill
{\scalebox{0.58}{\includegraphics{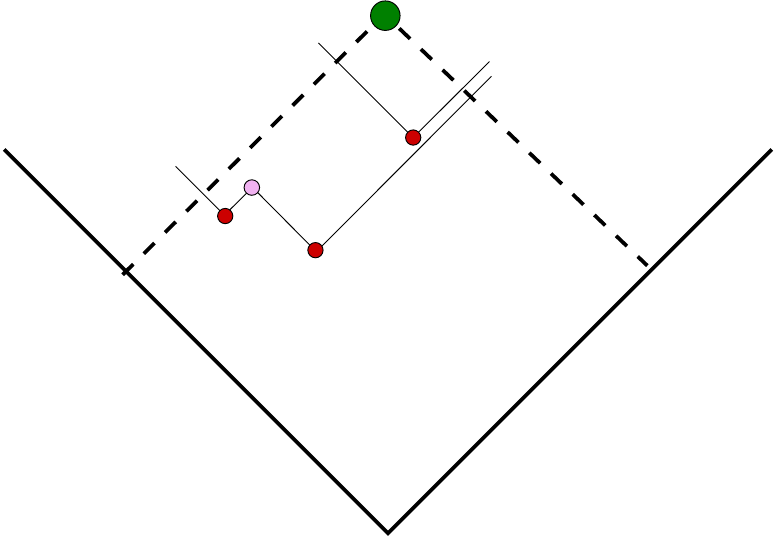}}} \caption{Left panel: Points of second generation
(in light blue), broken lines joining them and points of the third generation (in dark red). Right
panel: Points of the third generation (in dark red), broken lines joining them and points of the
fourth generation (in purple).
 \label{Fig_geo_RSK2} }
\end{center}
\end{figure}

Our interest in this construction is based on the fact that the distribution of $\lambda(\theta)$
can be fairly easily computed, as opposed to that of $L(\theta)=\lambda_1(\theta)$.

\begin{theorem} The distribution of $\lambda(\theta)$ is given by the \emph{Poissonized Plancherel
measure}
\begin{equation}
\label{eq_Poi_Planch}
 \P(\lambda(\theta)=\mu)=e^{-\theta^2}\left( \frac{\theta^{|\mu|} {\rm dim}(\mu)}{|\mu|!} \right)^2,\quad \mu\in\Y.
\end{equation}
\end{theorem}
\begin{proof}[Sketch of the proof]
Note that $\lambda(\theta)$ depends only on the relative order of coordinates of the points inside
the square with vertices $(0,0)$, $(\theta,0)$, $(0,\theta)$, $(\theta,\theta)$, but not on their
positions. This order can be encoded by a permutation $\sigma(\theta)$ which is the permutation
between the points ordered by $x$ coordinate and by $y$ coordinate. Observe that the size
$n(\theta)$ of $\sigma(\theta)$, i.e.\ the number of points inside the square, has Poisson
distribution with parameter $\theta^2$, and given this number, the distribution of
$\sigma(\theta)$ is uniform.

Next, we need to use the Robinson-Schensted-Knuth (RSK) correspondence (or rather its earlier
Robinson--Schensted version), which is an explicit bijection between permutations of size $n$ and
pairs of standard Young tableaux of same shape. We refer the reader to \cite{Sagan},
\cite{Fulton}, \cite{Kn} for the definition and properties of this correspondence. What is
important for us is that $\lambda(\theta)$ is precisely the common shape of two obtained tableau.
Therefore, given that the size  $n(\theta)$ of $\sigma(\theta)$ is $m$, the number of boxes in
$\lambda(\theta)$ is also $m$, and the conditional distribution of $\lambda(\theta)$ is
$$
 \P(\lambda(\theta)=\mu\mid n(\theta)=m)=\frac{{\rm dim}^2(\mu)}{m!}.
$$
Taking into account the Poisson distribution on $m$, we arrive at \eqref{eq_Poi_Planch}.
\end{proof}
\noindent {\bf Remark.} The fact that the RSK correspondence is a bijection implies that
\begin{equation}
\label{eq_Burnside}
 \sum_{|\lambda|=n} {\rm dim}^2(\lambda)=n!.
\end{equation}
On the other hand, if we recall the definition of ${\rm dim}(\lambda)$ as the dimension of
irreducible representation of the symmetric group $S(n)$, then, taking into the account that
$|S(n)|=n!$, the equality \eqref{eq_Burnside} is nothing else but the celebrated \emph{Burnside
identity} which says that squares of the dimensions of irreducible complex representations of any
finite group sum up to the number of the elements in the group.

\smallskip

Let us now suggest some intuition on why the asymptotic behavior of $L(\theta)$ should be related
to those of growth models and, in particular, to the ballistic deposition that we started with in
Section \ref{Section_Intro}. Introduce coordinates $(z,t)$ in Figure \ref{Fig_geo_RSK} so that $t$
is the vertical coordinate, and consider the following growth model. At time $t$ the \emph{height
profile} is given by an integer--valued (random) function $h(x,t)$, at time zero $h(x,0)\equiv 0$.
At any given time $t$ and any point $x$, the left and right $x$--limits of the function $h(x,t)$
differ at most by $1$, in other words, $h(x,t)$ is almost surely a step function with steps $(+1)$
(``up step'', when we read the values of the function from left to right) and $(-1)$ (``down
step''). If there is a point of the Poisson process (of Figure \ref{Fig_geo_RSK}) at $(z,t)$, then
at time $t$ a \emph{seed} is born at position $x=z$, which is combination of up and down steps,
i.e.\ $h(z,t)$ increases by $1$. After that the down step starts moving with speed $1$ to the
right, while the up step starts moving to the left with the same speed. When the next seed is
born, another up and down steps appear and also start moving. When up and down steps (born by
different seeds) meet each other, they disappear. This model is known as \emph{Polynuclear Growth}
(PNG), see \cite{Meakin}, \cite{PS_review}, a very nice computer simulation for it is available at
Ferrari's website \cite{Fe}, and at Figure \ref{Fig_PNG} we show one possible height function.

\begin{figure}[h]
\begin{center}
{\scalebox{1.3}{\includegraphics{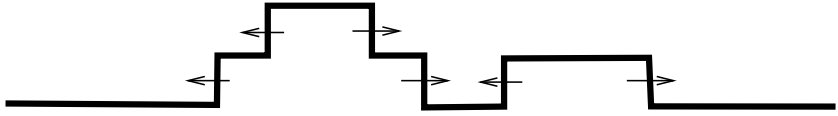}}} \caption{Height function of PNG model after the birth
of 3 seeds.
 \label{Fig_PNG} }
\end{center}
\end{figure}

Coming back to Figure \ref{Fig_geo_RSK}, note that its broken lines symbolize the space--time
trajectories of up/down steps, while second generation points are identified with collisions of up
and down steps. In particular, the positions of up and down steps at time $t=t_0$ are the points
of intersection of the line $t=t_0$ at Figure \ref{Fig_geo_RSK} with parts of broken lines of
slope $(-1)$ and $(+1)$, respectively. Now it is easy to prove that the PNG-height $h(0,t)$ at
time $t$ and point $0$ is precisely the Last Passage Percolation Time $L(t)$. In order to observe
the full Young diagram $\lambda(\theta)$ one should introduce \emph{multi-layer} PNG model, where
a seed on level $k$, $k\ge 2$, is born when the up and down steps collide at level $k-1$, and the
position of seed coincides with the position of collision, see \cite{PS} for details.

The PNG model is in the KPZ universality class, and obtaining information on its asymptotic
behavior (roughness of interface, fluctuation distributions) would give us similar (although
conjectural) statements for other members of the same universality class.

\section{The Schur measures and their asymptotic behavior}

\label{Section_Schur_measures}

The aim of this section is to show how the asymptotic behavior of the Poissonized Plancherel
measure and certain more general distributions on Young diagrams can be analyzed.

Take any two Schur--positive specializations $\rho_1$, $\rho_2$ of the algebra of symmetric
functions $\Lambda$ (those were classified in Theorem \ref{Theorem_Thoma}). The following
definition first appeared in \cite{Ok-wedge}.

\begin{definition}
\label{Def_Schur_meas}
 \emph{The Schur measure} $\mathbb S_{\rho_1;\rho_2}$ is a probability measure on the set of all Young diagrams defined
 through
 $$
 \P_{\rho_1,\rho_2}(\lambda)=\dfrac{s_\lambda(\rho_1) s_\lambda(\rho_2)}{H(\rho_1;\rho_2)},
 $$
 where the normalizing constant $H(\rho_1;\rho_2)$ is given by
$$
 H(\rho_1;\rho_2)=\exp\left(\sum_{k=1}^{\infty} \frac{p_k(\rho_1)p_k(\rho_2)}{k}\right).
$$
\end{definition}
\noindent{\bf Remark.}
 The above definition makes sense only if $\rho_1$, $\rho_2$ are such that
 \begin{equation}
 \label{eq_sum_schur}
  \sum_\lambda s_\lambda(\rho_1) s_\lambda(\rho_2)<\infty,
 \end{equation}
and in the latter case this sum equals $H(\rho_1;\rho_2)$, as follows from Theorem
\ref{theorem_Cauchy}. The convergence of \eqref{eq_sum_schur} is guaranteed, for instance, if
$|p_k(\rho_1)|<C r^k$ and $|p_k(\rho_2)|< C r^k$ with some constants $C>0$ and $0<r<1$. In what
follows we assume that this (or a similar) condition is always satisfied.

\begin{proposition}
 Let $\rho_\theta$ be the (Schur--positive) specialization with single non-zero parameter $\gamma=\theta$, i.e.\
 $$
  p_1(\rho_\theta)=\theta,\quad p_k(\rho_\theta)=0, \quad k>1.
 $$
 Then $\P_{\rho_\theta,\rho_\theta}$ is the Poissonized Plancherel measure \eqref{eq_Poi_Planch}.
\end{proposition}
\begin{proof}
 This is an immediate corollary of Proposition \ref{Prop_plancherel}.
\end{proof}

Our next goal is to show that any Schur measure is a determinantal point process. Given a Young
diagram $\lambda$, we associate to it a point configuration $X(\lambda)=\{\lambda_i-i+1/2\}\subset
\mathbb Z+1/2$. This is similar to the correspondence shown in Figures \ref{Fig_broken},
\ref{Fig_broken_IC}. Note that $X(\lambda)$ is semi--infinite, i.e.\ there are finitely many
points to the right of the origin, but almost all points to the left of the origin belong to
$X(\lambda)$.

\begin{theorem}[\cite{Ok-wedge}]
\label{theorem_corr_Schur}
 Suppose that the $\lambda\in\Y$ is distributed according to the Schur measure $\mathbb S_{\rho_1;\rho_2}$.
  Then $X(\lambda)$ is a determinantal point process on $\mathbb Z+1/2$ with
 correlation kernel $K(i,j)$ defined by the generating series
 \begin{equation}
 \label{eq_x2}
  \sum_{i,j\in\mathbb Z+\frac12}  K(i,j) v^i w^{-j}=\frac{ H(\rho_1;v) H(\rho_2;w^{-1})}{ H(\rho_2;v^{-1}) H(\rho_1; w)}
  \sum_{k=\frac12,\frac32,\frac52,\dots} \left(\frac{w}{v}\right)^k,
 \end{equation}
 where
 $$
  H(\rho;z)=\sum_{k=0}^{\infty} h_k(\rho) z^k=\exp\left(\sum_{k=1}^{\infty}
  p_k(\rho)\frac{z^k}{k}\right).
 $$
\end{theorem}
\noindent {\bf Remark 1.} If we expand $H$--functions in the right--hand side of \eqref{eq_x2}
into power series and multiply the resulting expressions, then \eqref{eq_x2} can be viewed as a
\emph{formal} identity of power series.

\noindent {\bf Remark 2.} There is also an \emph{analytical} point of view on \eqref{eq_x2}.
 Using the fact that (under suitable convergence conditions) the contour
integral around zero
$$
 \frac{1}{2\pi \i} \oint \left( \sum_{k=-\infty}^{\infty} a_k z^k \right) \frac{dz}{z^{n+1}}
$$
is equal to $a_n$ and also that when $|w|<|v|$ we have
$$
\sum_{k=\frac12,\frac32,\frac52,\dots} \left(\frac{w}{v}\right)^k = \frac{\sqrt{vw}}{v-w},
$$
we can rewrite \eqref{eq_x2} as
\begin{equation}
\label{eq_x4}
 K(i,j)=\frac{1}{(2\pi \i)^2} \oint\oint \frac{ H(\rho_1;v) H(\rho_2;w^{-1})}{ H(\rho_2;v^{-1}) H(\rho_1;
 w)} \frac{\sqrt{vw}}{v-w} \frac{dv dw}{v^{i+1} w^{-j+1}},
\end{equation}
with integration going over two circles around the origin $|w|=R_1$, $|v|=R_2$ such that $R_1<R_2$
and the functions $H(\rho_1;u)$,  $H(\rho_2;u^{-1})$  are holomorphic in the annulus
$R_1-\varepsilon<|u|<R_2+\varepsilon$. In particular,
 if $|p_k(\rho_1)|<C r^k$ and $|p_k(\rho_2)|< C r^k$ with some constants $C>0$ and $0<r<1$, then
 any $r<R_1<R_2<r^{-1}$ are suitable.

\medskip

We now present a proof of Theorem \ref{theorem_corr_Schur}  which is due to Johansson
\cite{Joh-EuroCongr}, see also \cite{Ok-wedge} for the original proof.
\begin{proof}[Proof of Theorem \ref{theorem_corr_Schur}]
We have to prove that for any finite set $A=\{a_1,\dots,a_m\}\subset \mathbb Z+1/2$ we have
$$
 \sum_{\lambda: A\subset X(\lambda)} \dfrac{s_\lambda(\rho_1)
 s_\lambda(\rho_2)}{H(\rho_1;\rho_2)} =\det \left[K(a_i,a_j)\right]_{i,j=1}^m.
$$
For this it suffices to prove the following formal identity of power series. Let $x=(x_1,\dots)$
and $y=(y_1,\dots)$ be two sets of variables; then
\begin{equation}
\label{eq_Schur_cor_formal}
 \sum_{\lambda : A\subset X(\lambda)} \dfrac{s_\lambda(x)
 s_\lambda(y)}{\prod_{i<j}(1-x_iy_j)^{-1}} =\det \left[\widehat K(a_i,a_j)\right]_{i,j=1}^m,
\end{equation}
where the generating function of $\widehat K(i,j)$ is similar to that of $K(i,j)$ but with
$\rho_1$ and $\rho_2$ replaced by $x$ and $y$, respectively. One shows (we omit a justification
here) that it is enough to prove \eqref{eq_Schur_cor_formal} for arbitrary \emph{finite} sets of
variables $x$ and $y$, so let us prove it for $x=(x_1,\dots,x_N)$, $y=(y_1,\dots,y_N)$. In the
latter case the Schur functions $s_\lambda(x)$ are non-zero only if $\ell(\lambda)\le N$. Because
of that it is more convenient to work with a finite point configuration
$X_N(\lambda)=\{\lambda_j+N-j\}_{j=1}^N\subset {\mathbb Z}_{\ge 0}$ which differs from
$X(\lambda)$ in two ways. First there is a deterministic shift by $N-1/2$, this has an evident
effect on the correlation functions. Second, $X(\lambda)$ is infinite, while $X_N(\lambda)$ is
finite. However,  the additional points of $X(\lambda)$ (as compared to those of $X_N(\lambda)$)
are deterministically located and move away to $-\infty$ and, therefore, they do not affect the
correlation functions in the end.

The definition of Schur functions \eqref{eq_Schur} implies that
\begin{equation}
\label{eq_x1}
 \dfrac{s_\lambda(x_1,\dots,x_N)
 s_\lambda(y_1,\dots,y_N)}{\prod_{i<j}(1-x_iy_j)^{-1}}=\frac{1}{Z} \det
 \left[x_i^{\lambda_j+N-j}\right]_{i,j=1}^N \det \left[y_i^{\lambda_j+N-j}\right]_{i,j=1}^{N}.
\end{equation}
Now we recognize a biorthogonal ensemble in  the right--hand side of \eqref{eq_x1}. Therefore, we
can use Theorem \ref{Theorem_biorth} which yields that $X_N(\lambda)$ is a determinantal point
process with correlation kernel
$$
 \widetilde K(\ell_1,\ell_2)=\sum_{i,j=1}^N x_i^{\ell_1} y_j^{\ell_2} G_{ij}^{-t},
$$
where $G_{ij}^{-t}$ is the inverse--transpose matrix of the $N\times N$ Gram matrix
$$
 G_{ij}=\sum_{\ell\ge 0} x_i^{\ell} y_j^{\ell} =\frac{1}{1-x_i y_j}.
$$
We can compute the determinant of $G$, which is
\begin{equation}
\label{eq_Cauchy} \det[G_{ij}]_{i,j=1}^N  = \det \left[\frac{1}{1-x_i y_j}\right]_{i,j=1}^N
=\frac{\prod_{i<j} (x_i-x_j)\prod_{i<j} (y_i-y_j)}{\prod_{i,j=1}^N (1-x_i y_j)}.
\end{equation}
This is known as the \emph{Cauchy determinant} evaluation. One can prove \eqref{eq_Cauchy}
directly, see e.g.\ \cite{Kr}. Another way is to recall that in the proof of Theorem
\ref{Theorem_biorth} we showed that the determinant of $G$ is the normalization constant of the
measure, and we know the normalization constant from the very definition of the Schur measure,
i.e.\ by the Cauchy identity.

By the Cramer's rule, we have
$$
 (G^{-t})_{k,\ell}= \frac{ (-1)^{k+l} \det[G_{ij}]_{i\ne k, j\ne \ell}}{\det[G_{ij}]_{i,j=1}^N }.
$$
Using the fact that submatrices of $G_{ij}$ are matrices of the same type and their determinants
can be evaluated using \eqref{eq_Cauchy}, we get
$$
 (G^{-t})_{k,\ell}=\dfrac{\prod_{j=1}^N (1-x_j y_\ell)(1-x_k y_j)}{(1-x_k y_\ell)\prod_{j\ne k}(x_k-x_j) \prod_{j\ne
 \ell}(y_j-y_\ell)}.
$$
We claim that
\begin{equation}
\label{eq_x3}
 \widetilde K(\ell_1,\ell_2)=\sum_{i,j=1}^N x_i^{\ell_1} y_j^{\ell_1} G_{ij}^{-t} =
 \frac{1}{(2\pi \i)^2}\oint\oint \prod_{k=1}^N \frac{(1-zy_k)(1-v x_k)}{(z-x_k)(v-y_k)} \frac{z^{\ell_1} v^{\ell_2}}{1-zv} dzdv,
\end{equation}
with contours chosen so that they enclose the singularities at $z=x_k$ and $v=y_k$, but do not
enclose the singularity at $zv=1$. Indeed, \eqref{eq_x3} is just the evaluation of the double
integral as the sum of the residues. Changing the variables $z=1/w$ and shifting $\ell_1,\ell_2$
by $N-1/2$ in \eqref{eq_x3} we arrive at \eqref{eq_x4}.
\end{proof}

Applying Theorem \ref{theorem_corr_Schur} to the Poissonized Plancherel measure we obtain
\begin{corollary}[\cite{BOO},\cite{J-Annals}]
\label{Cor_Planch_cor_function}
 Suppose that $\lambda$ is a random Young diagram distributed by the Poissonized Plancherel measure.
 Then the points of $X(\lambda)$ form a determinantal point process on $\mathbb Z +\frac12$ with correlation kernel
 $$
 K_\theta(i,j)=\frac{1}{(2\pi \i)^2} \oint\oint  \exp\bigg(\theta(v -v^{-1}-w+w^{-1})\bigg) \frac{\sqrt{vw}}{v-w} \frac{dv dw}{v^{i+1} w^{-j+1}},
 $$
 with integration over positively oriented  simple contours enclosing zero  and such that $|w|<|v|$.
\end{corollary}

Our next aim is to study the behavior of $K_\theta(i,j)$ as $\theta\to\infty$. The argument below
is due to Okounkov \cite{Ok-lectures}, but the results were obtained earlier in \cite{BOO},
\cite{J-Annals} by different tools. Let us start from the case $i=j$. Then $K(i,i)$ is the
\emph{density} of particles of our point process or, looking at Figure \ref{Fig_broken_IC}, the
average local slope of the (rotated) Young diagram. Intuitively, one expects to see some
non-trivial behavior when $i$ is of order $\theta$. To see that  set $i=u\theta$. Then $K_\theta$
transforms into
\begin{equation}
\label{eq_x5}
 K_\theta(u\theta,u\theta)=\frac{1}{(2\pi \i)^2} \oint\oint  \exp\bigg(\theta(S(v)-S(w))\bigg) \frac{\sqrt{vw}}{v-w} \frac{dv dw}{vw},
\end{equation}
with
$$
 S(z)=z-z^{-1}-u\ln z.
$$
Our next aim is to deform the contours of integration so that $\Re (S(v)-S(w))<0$ on them. (It is
ok if $\Re (S(v)-S(w))=0$ at finitely many points.)  If we manage to do that,
then  \eqref{eq_x5}
would decay as $\theta\to\infty$. Let us try to do this. First, compute the critical points of
$S(z)$, i.e.\ roots of its derivative
$$
 S'(z)=1+z^{-2}-u z^{-1}.
$$
When $|u|<2$ the equation $S'(z)=0$ has two complex conjugated roots of absolute value $1$ which
we denote $e^{\pm i\phi}$. Here $\phi$ satisfies $2\cos(\phi)=u$. Let us deform the contours so
that both of them pass through the critical points and look as shown at Figure
\ref{Figure_contours}.
\begin{figure}[h]
\begin{center}
{\scalebox{1.0}{\includegraphics{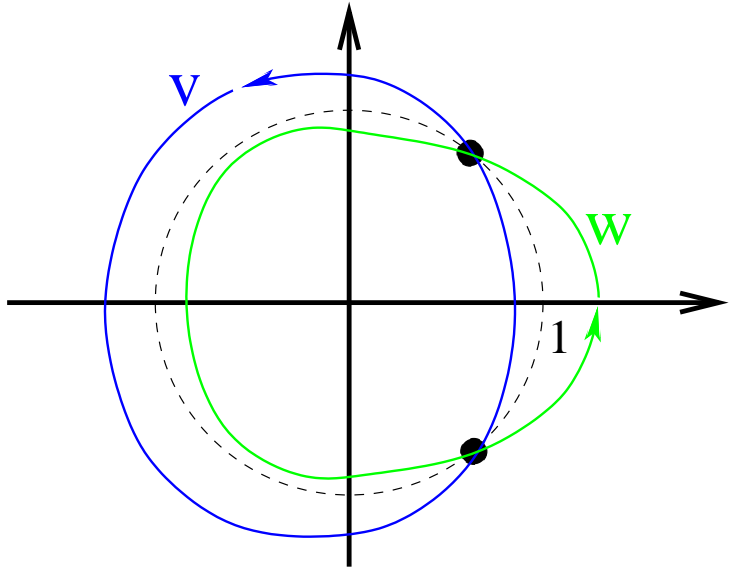}}} \caption{Deformed contours: $v$--contour in blue
and $w$--contour in green. The dashed contour is the unit circle and the black dots indicate the
critical points of $S(z)$.
 \label{Figure_contours}}
\end{center}
\end{figure}
We claim that now $\Re S(v)<0$ everywhere on its contour except at critical points $e^{\pm
i\phi}$, and $\Re S(w)>0$ everywhere on its contour except at critical points $e^{\pm i\phi}$
($\Re S(v)=\Re S(w)=0$ at $e^{\pm\phi}$.) To prove that observe that $\Re S(z)=0$ for $z$ on the
unit circle $|z|=1$ and compute the gradient of $\Re S(z)=\Re S(a+bi)$ on the unit circle (i.e.\
when $a^2+b^2=1$):
$$
 \Re S(a+bi)= a -\frac{a}{a^2+b^2} -\frac{u}{2}\ln(a^2+b^2),
$$
\begin{multline}
\label{eq_x7}
 \nabla \Re S(a+bi)=\left(1 - \frac{b^2-a^2}{(a^2+b^2)^2}-\frac{au}{a^2+b^2}\, ,\, \frac{2ab}{(a^2+b^2)^2}
 -\frac{bu}{a^2+b^2}\right)
 \\=\left(1 - b^2+a^2-au\, , 2ab -bu\right)=\left(2a^2-au\, , 2ab-bu\right)=(2a-u)(a,b).
\end{multline}
Identity \eqref{eq_x7} implies that the gradient vanishes at points $e^{\pm i\phi}$, points
outwards the unit circle on the right arc joining the critical points and points inwards on the
left arc. This implies our inequalities for $\Re S(z)$ on the contours. (We assume that the
contours are fairly close to the unit circle so that the gradient argument works.)

Now it follows that after the deformation of the contours the integral vanishes as
$\theta\to\infty$. Does this mean that the correlation functions also vanish? Actually, no. The
reason is that the integrand in \eqref{eq_x5} has a singularity at $v=w$. Therefore, when we
deform the contours from the contour configuration with $|w|<|v|$, as we had in Corollary
\ref{Cor_Planch_cor_function}, to the contours of Figure \ref{Figure_contours} we get a residue of
the integrand in \eqref{eq_x5} at $z=w$ along the arc of the unit circle joining $e^{\pm i\phi}$.
This residue is
$$
 \frac{1}{2\pi \i}\int_{e^{-i\phi}}^{e^{i\phi}} \frac{dz}{z} =\frac{\phi}{\pi}.
$$

We conclude that if $u=2\cos(\phi)$ with $0<\phi<\pi$, then
$$
 \lim_{\theta\to\infty} K_\theta(u\theta,u\theta)=\frac{\phi}{\pi}.
$$

Turning to the original picture we see that the asymptotic density of particles at point $i$
changes from $0$ when $i\approx 2\theta$ to $1$ when $i\approx -2\theta$. This means that after
rescaling by the factor $\theta^{-1}$ times the Plancherel--random Young diagram asymptotically
looks like in Figure \ref{Figure_limit_shape}. This is a manifestation of the
Vershik--Kerov--Logan--Shepp limit shape theorem, see \cite{VK}, \cite{LS}.

\begin{figure}[h]
\begin{center}
{\scalebox{0.8}{\includegraphics{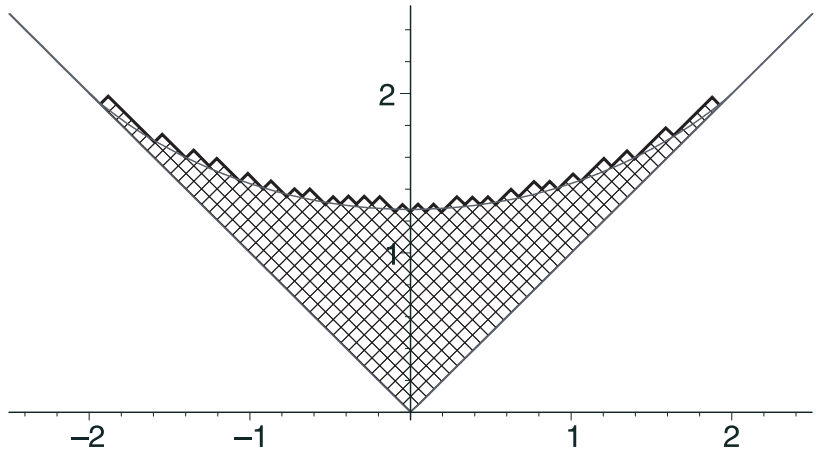}}} \caption{The celebrated
Vershik--Kerov--Logan-Shepp curve as a limit shape for the Plancherel random Young diagrams.
 \label{Figure_limit_shape}}
\end{center}
\end{figure}
\medskip

More generally, what happens with $K_\theta(i,j)$ when $i=u\theta+x$, $j=u\theta +y$ and
$\theta\to\infty$? In other words, we want to study how the point configuration (or the boundary
of the random Young diagram $\lambda$) behaves in the limit locally near a ``bulk point''. One
proves the following theorem.
\begin{theorem}[\cite{BOO}]
\label{theorem_bulk}
 For any $-2<u<2$ and any two integers $x$, $y$ we have
 \begin{equation}
 \label{eq_discrete_sine}
  \lim_{\theta\to\infty} K_\theta(\lfloor u\theta\rfloor+x, \lfloor u \theta\rfloor +y) =\begin{cases}
  \dfrac{\sin(\phi(x-y))}{\pi(x-y)},&\text{ if } x\ne y,\\ \dfrac{\phi}{\pi},&\text{ otherwise,}
  \end{cases}
 \end{equation}
 where $\phi=\arccos(u/2)$.
\end{theorem}
\noindent{\bf Remark.} The right--hand side of \eqref{eq_discrete_sine} is known as the
\emph{discrete sine kernel} and it is similar to the \emph{continuous sine kernel} which arises as
a \emph{universal} local limit of correlation functions for eigenvalues of random Hermitian
(Wigner) matrices, see e.g.\ \cite{EY}, \cite{TV} and references therein.

\begin{proof}[Proof of Theorem \ref{theorem_bulk}]
 The whole argument remains the same as in the case $x=y=0$, except for the computation of the residue which is now
\begin{equation}
\label{eq_x6}
 \frac{1}{2\pi \i}\int_{e^{-i\phi}}^{e^{i\phi}} \frac{dz}{z^{x-y+1}} =\frac{\sin(\phi(x-y))}{\pi(x-y)}.
\end{equation} We conclude that if $u=2\cos(\phi)$ with $0<\phi<\pi$, then
$$
 \lim_{\theta\to\infty} K_\theta(\lfloor u\theta\rfloor+x,\lfloor u\theta\rfloor+y)=\frac{\sin(\phi(x-y))}{\pi(x-y)}.
$$
\end{proof}

So far we got some understanding on what's happening in the \emph{bulk}, while we started with the
Last Passage Percolation which is related to the so-called \emph{edge} asymptotic behavior, i.e.\
limit fluctuations of $\lambda_1$. This corresponds to having $u=2$, at which point the above
arguments no longer work. With some additional efforts one can prove the following theorem:

\begin{theorem}[\cite{BOO},\cite{J-Annals}]
For any two reals $x$, $y$ we have \label{Theorem_Edge}
\begin{equation}
\label{eq_Airy_limit}
\lim_{\theta\to\infty} \theta^{1/3} K_\theta(2\theta+x\theta^{1/3},2\theta+y\theta^{1/3})=K_{Airy}(x,y)
\end{equation}
where
\begin{equation}
\label{eq_Airy_def} K_{Airy}(x,y)=\frac{1}{(2\pi \i)^2} \int \int e^{\widetilde v^3/3-\widetilde
w^3/3 + \widetilde v x -\widetilde w y} \frac{d\widetilde v d\widetilde w}{\tilde v-\tilde w},
\end{equation}
 with contours shown at the right panel of Figure \ref{Figure_contours2}.
\end{theorem}
\noindent {\bf Remark 1.} Theorem \ref{Theorem_Edge} means that the random point process
$X(\lambda)$ ``at the edge'', after shifting by $2\theta$ and rescaling by $\theta^{1/3}$,
converges to a certain non-degenerate determinantal random process with state space $\mathbb R$
and correlation kernel $K_{Airy}$.

\noindent {\bf Remark 2.} As we will see, Theorem \ref{Theorem_Edge} implies the following limit
theorem for the Last Passage Percolation Time $\lambda_1$: For any $s\in\mathbb R$
$$
 \lim_{\theta\to\infty}\P(\lambda_1\le 2\theta + s \theta^{1/3})=
 \det({\mathbf 1} -K_{Airy}(x,y))_{L_2(s,+\infty)}.
$$
One shows that the above \emph{Fredholm determinant} is the Tracy--Widom distribution $F_2(s)$
from Section \ref{Section_Intro}, see \cite{TW-U}.

\begin{proof}[Proof of Theorem \ref{Theorem_Edge}]
We start as in the proof of Theorem \ref{theorem_bulk}. When $u=2$ the two critical points of
$S(z)$ merge, so that the contours now look as in Figure \ref{Figure_contours2} (left panel) and
the integral in \eqref{eq_x6} vanishes. Therefore, the correlation functions near the edge tend to
$0$. This is caused by the fact that points of our process near the edge rarify, distances between
them become large, and the probability of finding a point in any given location tends to $0$.

\begin{figure}[h]
\begin{center}
{\scalebox{0.7}{\includegraphics{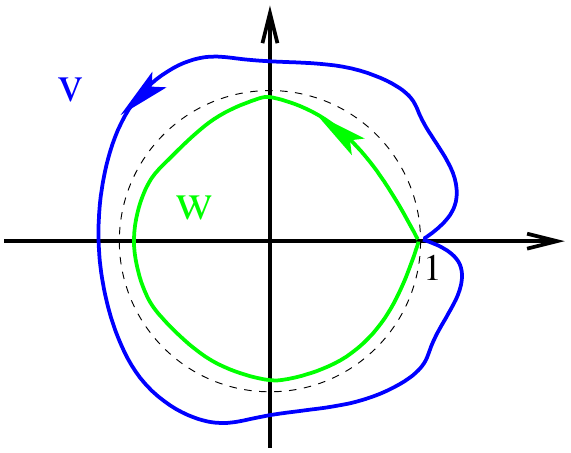}}} \hfill
{\scalebox{0.7}{\includegraphics{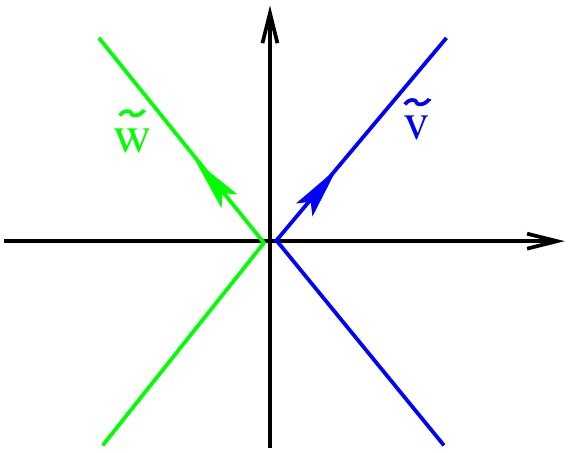}}} \caption{Contours for the edge--scaling limit
($u=2$). Left panel: $v$--contour in blue and $w$--contour in green. The dashed contour is the
unit circle. Right panel: limiting contours.
 \label{Figure_contours2}}
\end{center}
\end{figure}

In order to see some nontrivial behavior we need \emph{rescaling}. Set
$$v=1+\theta^{-1/3}  \widetilde v,\quad w =1 +\theta^{-1/3} \widetilde w$$
in the contour integral. Note that $z=1$ is a double critical point of $S(z)=z-z^{-1}-2\ln(z)$, so
that in the neighborhood of $1$ we have
$$
 S(z)=\frac{1}{3} (z-1)^3+O((z-1)^4)
$$
Now as $\theta\to\infty$ we have
\begin{multline*}
\exp\bigg(\theta(S(v)-S(w))\bigg)=\exp\left(\theta\left(\frac{1}{3}(\theta^{-1/3}\widetilde v)^3-
\frac{1}{3}(\theta^{-1/3}\tilde w)^3\right) +o(1)\right)\\=\exp\left(\frac{1}{3}\widetilde
v^3-\frac{1}{3}\tilde w^3\right).
\end{multline*}
We conclude that as $\theta\to\infty$
\begin{equation}
\label{eq_Airy_limit} K_\theta(2\theta+x\theta^{1/3},2\theta+y\theta^{1/3})\approx
\frac{\theta^{-1/3}}{(2\pi \i)^2}\int \int e^{\widetilde v^3/3-\widetilde w^3/3 + \widetilde v x
-\widetilde w y} \frac{d\widetilde v d\widetilde w}{\tilde v-\tilde w},
\end{equation}
and the contours here are contours of Figure \ref{Figure_contours2} (left panel) in the
neighborhood of 1; they are shown at the right panel of Figure \ref{Figure_contours2}.
\end{proof}
Using Theorem \ref{Theorem_Edge} we can now compute the asymptotic behavior of the Last Passage
Percolation time, i.e.\ of $\lambda_1$. Using the inclusion--exclusion principle, for any
$A\in\mathbb R$ we have
\begin{equation}
\label{eq_inc_exc}
 \P(\lambda_1\le A)=1-\sum_{x>A} \rho_1(x)
 +\frac{1}{2!}\sum_{x,y>A} \rho_2(x,y) - \frac{1}{3!} \sum_{x,y,z>A} \rho_3(x,y,z)+\dots
\end{equation}
Recall that correlation functions $\rho_k$ are $k\times k$ determinants involving kernel
$K_\theta$, substitute $A=2\theta + s \theta^{1/3}$ and send $\theta\to\infty$. The sums in
\eqref{eq_inc_exc} turn into the integrals and we get (ignoring convergence issues, which can,
however, be handled)
\begin{multline*}
 \lim_{\theta\to\infty}\P(\lambda_1\le 2\theta + s \theta^{1/3})=1-\int_{x>s}
 K_{Airy}(x,x)dx\\+\frac{1}{2!} \int_{x,y>s} \det\begin{bmatrix}
K_{Airy}(x,x)&K_{Airy}(x,y)\\ K_{Airy}(y,x) &K_{Airy}(y,y)\end{bmatrix}dxdy
-\dots\,.
\end{multline*} In the last expression one recognizes the \emph{Fredholm
determinant} expansion (see e.g.\ \cite{Lax} or \cite{Simon})
for
$$
 \det({\mathbf 1} -K_{Airy}(x,y))_{L_2(s,+\infty)}. \qedhere
$$

\bigskip

The conceptual conclusion from all the above is that as soon as we have an integral representation
for the correlation kernel of a point process, many limiting questions can be answered by
analyzing these integrals. The method for the analysis that we presented is, actually, quite
standard and is well-known (at least since the $XIX$ century) under the \emph{steepest descent
method} name. In the context of determinantal point processes and Plancherel measures it was
pioneered by Okounkov and we recommend \cite{Ok-lectures} for additional details.

\section{The Schur processes and Markov chains}
\label{Section_Schur_process} While in the previous sections we gave a few of tools for solving
the problems of probabilistic origin, in this section we present a general framework, which
produces ``analyzable'' models.

\subsection{The Schur process}

The following definition is due to \cite{OR-S}.

\begin{definition}
\label{Def_Schur_proc} The Schur process (of rank $N$) is a probability measure on sequences of
Young diagrams $\lambda^{(1)}, \mu^{(1)}, \lambda^{(2)}, \mu^{(2)},\dots,\mu^{(N-1)},
\lambda^{(N)}$, parameterized by $2N$ Schur--positive specializations
$\rho_0^+,\dots,\rho^+_{N-1}$, $\rho_1^-,\dots,\rho_N^-$ and given by
\begin{multline}
\label{eq_Schur_proc}
 \P\bigg(\lambda^{(1)},\mu^{(1)}, \lambda^{(2)}, \mu^{(2)},\dots,\mu^{(N-1)},
\lambda^{(N)}\bigg)
\\=\frac{1}{Z} s_{\lambda^{(1)}}(\rho^+_0)
 s_{\lambda^{(1)}/\mu^{(1)}}(\rho^-_1) s_{\lambda^{(2)}/\mu^{(1)}}(\rho^+_1)
 \cdots s_{\lambda^{(N)}/\mu^{(N-1)}}(\rho^+_{N-1})
 s_{\lambda^{(N)}}(\rho^-_N),
\end{multline}
where $Z$ is a normalization constant.
\end{definition}
Proposition \ref{Prop_J-T_skew} implies that in Definition \ref{Def_Schur_proc} almost surely
$$
\lambda^{(1)}\supset\mu^{(1)}\subset \lambda^{(2)}\supset\dots\supset\mu^{(N-1)}\subset
\lambda^{(N)}.
$$
It is convenient to use the graphical illustration for the  Schur process as shown in Figure \ref{Figure_Schur_saw}.

\begin{figure}[ht]
\begin{center}
\begin{picture} (330,60)
 \put(10,5){$\varnothing$}
 \put(20,10){\color{blue}\vector(1,1){30}}
 \put(25,30){\color{blue}$\rho^+_0$}
 \put(52,38){$\lambda^{(1)}$}
 \put(70,40){\color{blue}\vector(1,-1){30}}
 \put(85,30){\color{blue}$\rho^-_1$}
 \put(100,5){$\mu^{(1)}$}
 \put(120,10){\color{blue}\vector(1,1){30}}
 \put(125,30){\color{blue}$\rho^+_1$}
 \put(152,38){$\lambda^{(2)}$}
 \put(170,40){\color{blue}\vector(1,-1){30}}
 \put(185,30){\color{blue}$\rho^-_2$}
 \put(200,5){$\mu^{(2)}$}
 \put(220,10){\color{blue}\vector(1,1){30}}
 \put(225,30){\color{blue}$\rho^+_2$}
 \put(252,38){$\lambda^{(3)}$}
 \put(270,40){\color{blue}\vector(1,-1){30}}
 \put(285,30){\color{blue}$\rho^-_3$}
 \put(300,5){$\varnothing$}
\end{picture}
\end{center}
\caption{Graphical illustration for the Schur process with $N=3$. \label{Figure_Schur_saw}}
\end{figure}
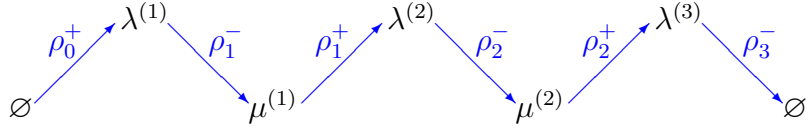

Note that if we set $N=1$ in the definition of the Schur process then we get back the Schur
measure of Definition \ref{Def_Schur_meas}.

Let us introduce some further notations. For two Schur--positive specializations $\rho$, $\rho'$
we set
$$
 H(\rho;\rho')=\exp\left(\sum_{k=1}^{\infty} \frac{p_k(\rho) p_k(\rho')}{k}\right).
$$
Given two specializations $\rho_1,\rho_2$, their \emph{union} $(\rho_1,\rho_2)$ is defined through
its values on power sums $p_k$
$$
 p_k(\rho_1,\rho_2)=p_k(\rho_1)+p_k(\rho_2).
$$
Theorem \ref{Theorem_Thoma} implies that if $\rho_1$ is a Schur--positive specialization with
parameters $(\alpha^{(1)},\beta^{(1)},\gamma^{(1)})$, and $\rho_2$ is a Schur--positive
specialization with parameters $(\alpha^{(2)},\beta^{(2)},\gamma^{(2)})$, then $(\rho_1,\rho_2)$
is a Schur--positive specialization with parameters
$(\alpha^{(1)}\bigcup\alpha^{(2)},\beta^{(1)}\bigcup\beta^{(2)},\gamma^{(1)}+\gamma^{(2)})$, where
$\alpha^{(1)}\bigcup\alpha^{(2)}$ stands for the sequence obtained by rearranging the union of
sequences $\alpha^{(1)}$ and $\alpha^{(2)}$ in decreasing order (and similarly for $\beta$). In
particular, if $\rho_1$ and $\rho_2$ specialize symmetric functions by substituting sets of
variables, say $(x_1,x_2,\dots)$ and $(y_1,y_2,\dots)$ (which corresponds to zero $\beta_i$ and
$\gamma$), then $(\rho_1,\rho_2)$ substitutes all the variables $(x_1,x_2,\dots,y_1,y_2,\dots)$.

The definition implies that for specializations $\rho_1,\dots,\rho_k$, $\rho'_1,\dots,\rho'_m$ we
have
$$
 H(\rho_1,\dots,\rho_k;\rho'_1,\dots,\rho'_m)=\prod_{i=1}^k \prod_{j=1}^m H(\rho_i;\rho_j).
$$

\begin{proposition} \label{Prop_Schur_proc_normalizations} Suppose that for every $i<j$ we have $H(\rho_i^+;\rho_j^-)<\infty$.
Then the Schur process is well--defined and the normalization constant $Z$ in its definition is
$$
 Z=\prod_{i<j} H(\rho_i^+;\rho_j^-).
$$
\end{proposition}
\begin{proof}
The proof is based on the iterated applications of identities
$$
 \sum_{\mu\in\Y} s_{\mu/\lambda}(\rho) s_{\mu/\nu}(\rho') \\= H(\rho;\rho')
  \sum_{\kappa\in\Y} s_{\lambda/\kappa}(\rho') s_{\nu/\kappa}(\rho)
$$
and
$$
 s_{\lambda/\mu}(\rho,\rho')=\sum_{\nu\in\Y} s_{\lambda/\nu}(\rho) s_{\nu/\mu}(\rho').
$$
The above identities are valid for any specializations $\rho$, $\rho'$ such that all the sums are
convergent and are just the results of the application of these specializations to the statements
of Propositions \ref{proposition_skew_Cauchy} and \ref{proposition_fundamental_skew_Schur}.

We have:
\begin{multline}
\label{eq_x8} \sum s_{\lambda^{(1)}}(\rho^+_0)
 s_{\lambda^{(1)}/\mu^{(1)}}(\rho^-_1) s_{\lambda^{(2)}/\mu^{(1)}}(\rho^+_1)
 \cdots s_{\lambda^{(N)}/\mu^{(N-1)}}(\rho^+_N)
 s_{\lambda^{(N)}}(\rho^-_N)
\\= H(\rho^+_0;\rho^-_1) \sum s_{\varnothing/\nu}(\rho^-_1)
 s_{\mu^{(1)}/\nu}(\rho^+_0) s_{\lambda^{(2)}/\mu^{(1)}}(\rho^+_1)
 \cdots s_{\lambda^{(N)}/\mu^{(N-1)}}(\rho^+_N)
 s_{\lambda^{(N)}}(\rho^-_N)
 \\= H(\rho^+_0;\rho^-_1) \sum s_{\varnothing/\nu}(\rho^-_1)
 s_{\lambda^{(2)}/\nu}(\rho^+_0,\rho^+_1)
 \cdots s_{\lambda^{(N)}/\mu^{(N-1)}}(\rho^+_N)
 s_{\lambda^{(N)}}(\rho^-_N)
\\= H(\rho^+_0;\rho^-_1) \sum s_{\lambda^{(2)}}(\rho^+_0,\rho^+_1)
 \cdots s_{\lambda^{(N)}/\mu^{(N-1)}}(\rho^+_N)
 s_{\lambda^{(N)}}(\rho^-_N),
\end{multline}
where we used the fact that $s_{\varnothing/\nu}=0$ unless $\lambda^{(1)}=\varnothing$, and
$s_{\varnothing/\varnothing}=1$. Note, that the summation in the last line of \eqref{eq_x8} runs
over $\lambda^{(2)}, \mu^{(2)},\dots,\mu^{(N-1)}, \lambda^{(N)}$, i.e.\ there is no summation over
$\lambda^{(1)}, \mu^{(1)}$ anymore. Iterating this procedure, we get the value of the
normalization constant $Z$.
\end{proof}

It turns out that ``one--dimensional'' marginals of the Schur processes are the Schur measures:

\begin{proposition}
\label{Prop_projection_of_process_is_measure} The projection of the Schur process to the Young
diagram $\lambda^{(k)}$ is the Schur measure $\mathbb S_{\rho_1;\rho_2}$ with specializations
$$
 \rho_1=(\rho^+_0,\rho^+_1,\dots,\rho^+_{k-1}),\quad
 \rho_2=(\rho^-_k,\rho^-_{k+1},\dots,\rho^-_N).
$$
\end{proposition}
\begin{proof}
 The proof is analogous to that of Proposition \ref{Prop_Schur_proc_normalizations}.
\end{proof}

Proposition \ref{Prop_projection_of_process_is_measure} means that the projection of a Schur
process to the Young diagram $\lambda^{(k)}$ can be identified with a determinantal point process
and, thus, can be analyzed with methods of Section \ref{Section_Schur_measures}. In fact, a more
general statement is true: The joint distribution of \emph{all} Young diagrams of a Schur process
is also a determinantal point process with correlation kernel similar to that of Theorem
\ref{theorem_corr_Schur}, see \cite{OR-S}, \cite{BR05}.

Note that if one of the specializations, say $\rho^+_k$ is \emph{trivial}, i.e.\ this is the Schur
positive specialization $(0;0;0)$, then (since $s_{\lambda/\mu}(0;0;0)=0$ unless $\lambda=\mu$)
two of the Young diagrams should coincide, namely $\mu^{(k)}=\lambda^{(k+1)}$. In this case we can
safely forget about $\mu^{(k)}$ and omit it from our notations. This also shows that in the
definition of the Schur process we could replace the saw--like diagram of Figure
\ref{Figure_Schur_saw} by any staircase--like scenario.

Let us give two examples of the Schur processes.

\subsection{Example 1. Plane partitions}

\label{Section_plane_partitions}

A \emph{plane partition} $Y$ is a $2d$ array of non-negative integers $Y_{ij}$, $i,j=1,2,\dots$,
such that $\sum_{i,j} Y_{ij}<\infty$ and the numbers weakly decrease along the rows and columns,
i.e.\ if $i'\ge i$ and $j'\ge j$, then $Y_{ij}\ge Y_{i'j'}$. The sum $\sum_{i,j} Y_{ij}$ is called
the \emph{volume} of the plane partition $Y$. In the same way as ordinary partitions were
identified with Young diagrams in Section \ref{Section_symmetric}, plane partitions can be
identified with \emph{$3d$ Young diagrams}. To see that, view a plane partition as a collection of
numbers written in the vertices of the regular square grid on the plane and put $k$ unit cubes on
each number $k$. The resulting $3d$ body is the desired $3d$ Young diagram; an example is shown in
Figure \ref{Figure_plane_partition}.

\begin{figure}[h]
\begin{center}
{\scalebox{1.0}{\includegraphics{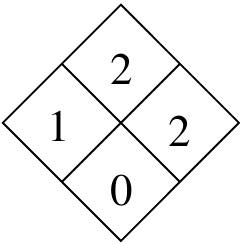}}} \hskip 2cm
{\scalebox{1.0}{\includegraphics{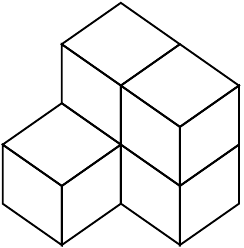}}} \caption{ Left panel: Plane partition of volume $5$.
Right panel: The corresponding $3d$ Young diagram.
 \label{Figure_plane_partition}}
\end{center}
\end{figure}

Fix  $q$,  $0<q<1$ and consider the following probability measure on the set of \emph{all} plane
partitions
$$
 \P(Y)=\frac{1}{M} q^{{\rm volume}(Y)},
$$
which is one of the simplest (speaking of definition, not properties) possible probability
measures on this set. The normalization constant $M$ is given by the celebrated MacMahon formula
(see \cite{MacMahon}, \cite[Section 7.20]{St_book}, \cite[Chapter I, Section 5, Example 13]{M})
$$
 M=\prod_{n=1}^{\infty} (1-q^n)^{-n}.
$$

We claim that the above measure can be described via a Schur process. In fact this is a particular
case of a more general statement that we now present.

\begin{definition}\label{Definition_skew_plane_part}
Fix two natural numbers $A$ and $B$. For a Young diagram $\pi\subset
B^A=(\underbrace{B,\dots,B}_{A\text{ times}})$, set $\bar\pi=B^A/ \pi$.
A skew plane
partition $\Pi$ with support $\bar\pi$ is a filling of all boxes of $\bar\pi$ by nonnegative
integers $\Pi_{i,j}$ (we assume that $\Pi_{i,j}$ is located in the $i$th row and $j$th column of
$B^A$) such that $ \Pi_{i,j}\ge \Pi_{i,j+1}$ and $\Pi_{i,j}\ge \Pi_{i+1,j}$ for all values of
$i,j$. The volume of the skew plane partition $\Pi$ is defined as
$$
{\rm volume}(\Pi)=\sum_{i,j}\Pi_{i,j}.
$$
\end{definition}

For an example of a skew plane partition see Figure \ref{Fig_3dskew}.

Our goal is to explain that the measure on plane partitions with given support $\bar\pi$  and
weights proportional to $q^{{\rm volume}(\,\cdot\,)}$, $0<q<1$, is a Schur process. This fact has
been observed and used in \cite{OR-S}, \cite{OR-skew}, \cite{OR-birth}, \cite{BMRT},
\cite{B-Schur}.

The Schur process will be such that for any two neighboring specializations $\rho_k^-,\rho_k^+$ at
least one is trivial. This implies that each $\mu^{(j)}$ coincides either with $\lambda^{(j)}$ or
with $\lambda^{(j+1)}$. Thus, we can restrict our attention to $\lambda^{(j)}$'s only.

 For a plane partition $\Pi$, we define the Young diagrams $\lambda^{(k)}$ ($1\le k\le A+B+1$) via
$$
\lambda^{(k)}(\Pi)=\bigl\{\Pi_{i,i+k-A-1}\mid (i,i+k-A-1)\in\bar\pi\bigr\}.
$$
Note that $\lambda^{(1)}=\lambda^{(A+B+1)}=\varnothing$. Figure \ref{Fig_3dskew} shows a skew
plane partition $\Pi$ and corresponding sequence of Young diagrams.

\begin{figure}[h]
\begin{center}
{\scalebox{0.7}{\includegraphics{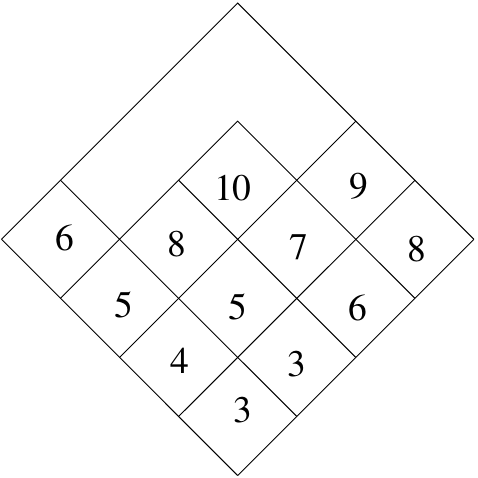}}} \caption{A skew plane partition $\Pi$ with support
 $B^A/\pi$. Here $B=A=4$ and  $\pi=(2,1,1)$. The corresponding sequence of Young diagrams  is $\varnothing \subset (6)
\supset (5) \subset (8,4) \subset (10,5,3) \supset (7,3) \subset
 (9,6) \supset (8) \supset \varnothing.
$ \label{Fig_3dskew}}
\end{center}
\end{figure}

We need one more piece of notation. Define
$$
{\mathcal L}(\pi)=\{A+\pi_i-i+1\mid i=1,\dots,A\}.
$$
This is an $A$-point subset in $\{1,2,\dots,A+B\}$, and all such subsets are in bijection with the
partitions $\pi$ contained in the box $B^A$; this is similar to the identification of Young
diagrams and point configurations used in Theorem \ref{theorem_corr_Schur}. The elements of
${\mathcal L}(\pi)$ mark the up-right steps in the boundary of $\pi$ (=back wall of $\Pi$), as in
Figure \ref{Fig_3dskew}.




\begin{theorem} \label{Theorem_qvol_is_Schur}
Let $\pi$ be a partition contained in the box $B^A$. The measure on the plane partitions $\Pi$
with support $\bar\pi$ and weights proportional to $q^{vol(\Pi)}$, is the Schur process with
$N=A+B+1$ and Schur--positive specializations $\{\rho_i^+\}$, $\{\rho_j^-\}$ defined by

\begin{gather} \rho_0^+=\rho_N^-=(0;0;0)\\
\rho_j^+=\begin{cases} (q^{-j};0;0),&j\in{\mathcal L}(\pi),
\\ (0;0;0),&j\notin{\mathcal L}(\pi);
\end{cases}
\qquad \rho_j^-=\begin{cases} (0;0;0),&j\in{\mathcal L}(\pi),\\(q^{j};0;0),&j\notin{\mathcal
L}(\pi),
\end{cases}
\end{gather}
where $(a;0;0)$ is the Schur--positive specialization with single non-zero parameter $\alpha_1=a$.
\end{theorem}
\noindent {\bf Remark.} One can send $A,B$ to infinity in Theorem \eqref{Theorem_qvol_is_Schur}.
If $\pi=\varnothing$, then we get plane partitions ($3d$ Young diagrams) we started with.

\begin{proof}[Proof of Theorem \ref{Theorem_qvol_is_Schur}]
Definition \ref{Definition_skew_plane_part} and Proposition \ref{Proposition_spec_alpha} imply
that the set of all skew plane partitions supported by $\bar\pi$, as well as the support of the
Schur process from the statement of the theorem, consists of sequences
$(\lambda^{(1)},\lambda^{(2)},\dots,\lambda^{(N)})$ with
$$
\gathered \lambda^{(1)}=\lambda^{(N)}=\varnothing, \\
\lambda^{(j)}\prec\lambda^{(j+1)} \text{  if  } j\in {\mathcal L}(\lambda),\qquad
\lambda^{(j)}\succ \lambda^{(j+1)} \text{  if  } j\notin {\mathcal L}(\lambda),
\endgathered
$$
where we write $\mu\prec\nu$ or $\nu\succ\mu$  if $\nu_1\ge\mu_1\ge\nu_2\ge\mu_2\ge \dots$\,.

On the other hand, Proposition \ref{Proposition_spec_alpha} implies that the weight of
$(\lambda^{(1)}, \lambda^{(2)}, \dots, \lambda^{(N)})$ with respect to the Schur process from the
hypothesis is equal to $q$ raised to the power
$$
\sum_{j=2}^{A+B}\left|\lambda^{(j)}\right|\Bigl( -(j-1){\bold 1}_{j-1\in {\mathcal
L}(\pi)}-(j-1){\bold 1}_{j-1\notin {\mathcal L}(\pi)} +j{\bold 1}_{j\in {\mathcal L}(\pi)}+j{\bold
1}_{j\notin {\mathcal L}(\pi)} \Bigr),
$$
where the four terms are the contributions of $\rho_{j-1}^+,\rho_{j-1}^-,\rho_j^+,\rho_j^-$,
respectively.

Clearly, the sum is equal to $\sum_{j=2}^{A+B}|\lambda^{(j)}|={\rm volume}(\Pi)$.
\end{proof}

Theorem \ref{Theorem_qvol_is_Schur} gives a way for analyzing random (skew) plane partitions via
the approach of Section \ref{Section_Schur_measures}. Using this machinery one can prove various
interesting limit theorems describing the asymptotic behavior of the model as $q\to 1$, see
\cite{OR-S}, \cite{OR-skew}, \cite{OR-birth}, \cite{BMRT}.

\subsection{Example 2. RSK and random words}
\label{Section_RSK_dyn}

Our next example is based on the Robinson--Schensted--Knuth correspondence and is a generalization
of constructions of Section \ref{Section_RSK}.

Take the alphabet of $N$ letters $\{1,\dots,N\}$ and a collection of positive parameters
$a_1,\dots,a_N$. Consider a growth process of the random word $\omega_N(t)$ with each letter $j$
appearing (independently) at the end of the word according to a Poisson process of rate $a_j$. In
particular, the length $|\omega_N(t)|$ of the word at time $t$ is a Poisson random variable with
parameter $(a_1+\dots+a_N)t$:
$$
 \P(|\omega_N(t)|=k) =e^{-(a_1+\dots+a_N)t} \frac{ ( (a_1+\dots+a_N)t)^k} {k!}.
$$
The growth of $\omega_N(t)$ can be illustrated by the random point process with point $(t,j)$
appearing if the letter $j$ is added to the word at time $t$, see Figure \ref{Fig_RSK_discrete}.
For any $T>0$, one can produce a \emph{Young diagram} $\lambda^{(N)}(T)$ from all the points
$(t,j)$ with $t\le T$ using the Robinson--Schensted-Knuth (RSK) algorithm, whose geometric version
was given in Section \ref{Section_RSK}. (We again address the reader to \cite{Sagan},
\cite{Fulton}, \cite{Kn} for the details on RSK.) In particular, the length of the first row of
$\lambda^{(N)}(T)$ equals the maximal number of points one can collect along a monotonous path
joining $(0,0)$ and $(T,N)$, as in Figure \ref{Fig_RSK_discrete}. More generally, let $w_{N-k}(t)$
be the word obtained from $w_N(t)$ by removing all the instances of letters $N,N-1,\dots,N-k+1$,
and let $\lambda^{(N-k)}(T)$ denote the Young diagram corresponding to $w_{N-k}(t)$, i.e.\ this is
the Young diagram obtained from all the points $(t,j)$ with $t\le T$, $j\le N-k$.

\begin{figure}[h]
\begin{center}
{\scalebox{1.0}{\includegraphics{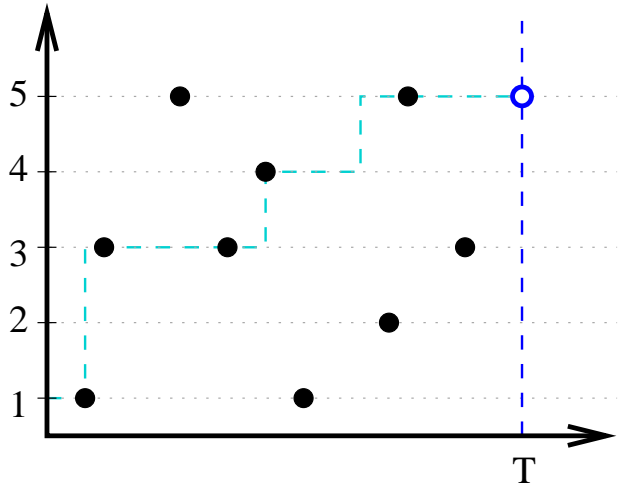}}} \caption{ Collection of points with
coordinates $(t,j)$ corresponding to the growth of random word and the path collecting maximal
number of points $\ell_N(T)=\lambda^{(N)}(T)=5$. Here $N=5$. \label{Fig_RSK_discrete}}
\end{center}
\end{figure}

\begin{proposition}
\label{Proposition_RSK_Schur} For any $t$ and $N$ the collection of (random) Young diagrams
$\lambda^{(1)}(t),\lambda^{(2)}(t),\dots,\lambda^{(N)}(t)$ forms a Schur process with probability
distribution
$$
 \frac{1}{Z} s_{\lambda^{(1)}}(a_1)
 s_{\lambda^{(2)}/\lambda^{(1)}}(a_2)\cdots s_{\lambda^{(N)}/\lambda^{(N-1)}}(a_N)
 s_{\lambda(N)}(\rho_t),
$$
where we identify $a_i$ with the Schur--positive specialization with parameter $\alpha_1=a_i$ and
all other parameters $0$, and $\rho_t$ is the specialization with single non-zero parameter
$\gamma=t$.
\end{proposition}
\begin{proof}
 This statement follows from  properties of RSK correspondence, cf.\ \cite{J-lectures}.
\end{proof}

Now let us concentrate on the random vector
$(\ell_1(t),\dots,\ell_N(t))=(\lambda^{(1)}_1(t),\dots,\lambda_1^{(N)}(t))$ and try to describe
its time evolution as $t$ grows. Recall that in Figure \ref{Fig_RSK_discrete} the value of
$\ell_k(T)$ is the maximal number of points collected along monotonous paths joining $(0,0)$ and
$(T,k)$. Suppose that at time $t$ a new letter $k$ appears, so that there is a point $(t,k)$ in
the picture. It means that $\ell_k$ grows by $1$ ($\ell_k(t)=\ell_k(t-)+1$), because we can add
this point to any path coming to any $(t',k)$ with $t'<k$. Clearly, $\ell_j$ does not change for
$j<k$. Note that if $\ell_{k+1}(t-)>\ell_{k}(t-)$, then $\ell_{k+1}$ also does not change, since
the maximal number of points collected along a path passing through $(t,k)$ is at most
$\ell_{k}(t)$. Finally, if $\ell_{k}(t-)=\ell_{k+1}(t-)=\dots=\ell_{k+m}(t-)$, then all the
numbers $\ell_k,\dots,\ell_{k+m}$ should increase by $1$, since the optimal path will now go
through $(t,k)$ and collect $\ell_k$ points.

The above discussion shows that the evolution of $(\ell_1(t),\dots,\ell_N(t))$ is a Markov
process, and it admits the following interpretation. Take $N$ (distinct) particles on $\mathbb Z$
with coordinates $\ell_1+1 <\ell_2+2<\dots<\ell_N+N$. Each particle has an independent exponential
clock of rate $a_i$. When a clock rings, the corresponding particle attempts to jump to the right
by one. If that spot is empty then the particle jumps and nothing else happens. Otherwise, in
addition to the jump, the particle \emph{pushes} all immediately right adjacent particles by one,
see Figure \ref{Fig_PushASEP} for an illustration of these rules. The dynamics we have just
described is known as the Long Range Totally Asymmetric Exclusion Process, and it is also a
special case of PushASEP, see \cite{BF-Push}, \cite{Spitzer}.
\begin{figure}[h]
\begin{center}
{\scalebox{1.0}{\includegraphics{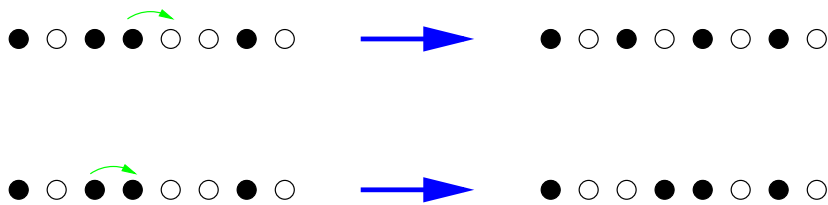}}} \caption{PushASEP. Top panel: jump of a
particle. Bottom panel: jump of a particle which results in pushing.\label{Fig_PushASEP}}
\end{center}
\end{figure}

We also note (without  proof) that if instead of
$(\lambda^{(1)}_1(t),\dots,\lambda_1^{(N)}(t))$ one considers the random vector
$(\lambda^{(1)}_1(t), \lambda^{(2)}_2(t)\dots,\lambda_N^{(N)}(t))$, then the fixed
time distribution of the particles
$\lambda_N^{(N)}-N<\lambda_{N-1}^{(N-1)}-(N-1)<\dots<\lambda_1^{(1)}-1$ coincides
with that of the well-known TASEP process\footnote{Note, however, that the
\emph{evolution} of the particles in TASEP is different from that of
$(\lambda^{(1)}_1(t), \lambda^{(2)}_2(t)\dots,\lambda_N^{(N)}(t))$.} that was
presented in Section \ref{Section_Intro}.

The conclusion now is that Proposition \ref{Proposition_RSK_Schur} together with methods of
Section \ref{Section_Schur_measures} gives a way of the asymptotic analysis of TASEP and PushASEP
stochastic dynamics at large times. In particular, one can prove the analogues of Theorems
\ref{Theorem_first_row_plancherel}, \ref{Theorem_longest_increasing}, \ref{Theorem_polymer_intro},
see e.g.\ \cite{J-TASEP}, \cite{BF-Push}.

\subsection{Markov chains on Schur processes}

\label{Section_Markov_chains}

In the last section we linked Schur processes to simple particle dynamics like TASEP using the RSK
correspondence. In this section we produce \emph{another} family of Markov dynamics connecting
such objects whose description is arguably more straightforward and independent of complicated
combinatorial algorithms, such as RSK.

We start by introducing a  general framework. Let $\rho$ and $\rho'$ be two Schur--positive
specializations such that $H(\rho;\rho')<\infty$. Define matrices $p^{\uparrow}_{\lambda\to\mu}$
and $p^{\downarrow}_{\lambda\to\mu}$ with rows and columns indexed by Young diagrams $\lambda$ and
$\mu$ as follows:
\begin{equation}
\label{eq_p_up}
 p^{\uparrow}_{\lambda\to\mu}(\rho;\rho')= \frac{1}{H(\rho;\rho')}
 \frac{s_{\mu}(\rho)}{s_\lambda(\rho)} s_{\mu/\lambda}(\rho'),
\end{equation}
\begin{equation}
\label{eq_p_down}
 p^{\downarrow}_{\lambda\to\mu}(\rho;\rho')=  \frac{s_{\mu}(\rho)}{s_\lambda(\rho,\rho')}
 s_{\lambda/\mu}(\rho').
\end{equation}

\begin{proposition}
\label{Prop_p_is_stochastic} The matrices $p^{\uparrow}_{\lambda\to\mu}$ and
$p^{\downarrow}_{\lambda\to\mu}$ are stochastic, i.e.\ all matrix elements are non-negative, and
for every $\lambda\in\Y$ we have
\begin{equation}
\label{eq_x9}
 \sum_{\mu\in\Y} p^{\uparrow}_{\lambda\to\mu}(\rho,\rho')=1,
\end{equation}
\begin{equation}
\label{eq_x10}
 \sum_{\mu\in\Y} p^{\downarrow}_{\lambda\to\mu}(\rho,\rho')=1.
\end{equation}
\end{proposition}
\begin{proof}
 Non-negativity of matrix elements follows from the definition of Schur--positive specializations.
 Equality \eqref{eq_x9} is a specialized version of the skew Cauchy identity of Proposition
 \ref{proposition_skew_Cauchy}, and equality \eqref{eq_x10} is a specialized version of Proposition
 \ref{proposition_fundamental_skew_Schur}.
\end{proof}
Proposition \ref{Prop_p_is_stochastic} means that matrices $p^{\uparrow}_{\lambda\to\mu}$ and
$p^{\downarrow}_{\lambda\to\mu}$ can be viewed as transitional probabilities of Markov chains.
Definitions \eqref{eq_p_up}, \eqref{eq_p_down} and Proposition \ref{Prop_J-T_skew} imply that
$p^{\uparrow}_{\lambda\to\mu}=0$ unless $\mu\supset\lambda$, i.e.\ the Young diagram
\emph{increases}, and $p^{\downarrow}_{\lambda\to\mu}=0$ unless $\mu\subset\lambda$, i.e.\ the
Young diagram \emph{decreases} (hence up and down arrows in the notations).

One of the important properties of the above stochastic matrices is that they agree with Schur
measures, i.e.\ Markov chains defined using them preserve the class of Schur measures. Formally,
we have the following statement:

\begin{proposition} \label{Prop_p_agrees_with_Schur} For any $\mu\in\Y$ we have
$$
 \sum_{\lambda\in\Y} {\mathbb S}_{\rho_1;\rho_2}(\lambda) p^\uparrow_{\lambda\to\mu}(\rho_2;\rho_3)=
 {\mathbb S}_{\rho_1,\rho_3;\rho_2} (\mu)
$$
and
$$
 \sum_{\lambda\in\Y} {\mathbb S}_{\rho_1;\rho_2,\rho_3}(\lambda) p^\downarrow_{\lambda\to\mu}(\rho_2;\rho_3)=
 {\mathbb S}_{\rho_1;\rho_2} (\mu).
$$
\end{proposition}
\noindent {\bf Remark.} Informally, $p^\uparrow_{\lambda\to\mu}$ increases the specialization by
adding $\rho_3$, while $p^\downarrow_{\lambda\to\mu}$ decreases the specialization by removing
$\rho_3$. Note also that ${\mathbb S}_{\rho_1;\rho_2}={\mathbb S}_{\rho_2;\rho_1}$ for any
$\rho_1,\rho_2$.
\begin{proof}[Proof of Proposition \ref{Prop_p_agrees_with_Schur}]
 This is an application of the specialized versions of Propositions
 \ref{proposition_skew_Cauchy} and \ref{proposition_fundamental_skew_Schur}.
\end{proof}

Next, note that the distribution of the Schur process of the form which appeared in Proposition
\ref{Proposition_RSK_Schur}
$$
 \frac{1}{Z}s_{\lambda^{(1)}}(\rho^+_0)
 s_{\lambda^{(2)}/\lambda^{(1)}}(\rho^+_1)\cdots s_{\lambda^{(N)}/\lambda^{(N-1)}}(\rho^+_{N-1})
 s_{\lambda(N)}(\rho_-)
$$
can be rewritten as
$$
 {\mathbb S}_{\rho^+_0,\dots,\rho^+_{N-1};\rho^-}(\lambda^{(N)})
p^{\downarrow}_{\lambda^{(N)}\to\lambda^{(N-1)}}(\rho^+_0,\dots,\rho^+_{N-2};\rho^+_{N-1}) \cdots
p^{\downarrow}_{\lambda^{(2)}\to\lambda^{(1)}}(\rho^+_0;\rho^+_{1}).
$$
More generally, \emph{any} Schur process can be viewed as a trajectory of a Markov chain with
transitional probabilities given by matrices $p^{\uparrow}$ and $p^{\downarrow}$ (with suitable
specializations) and initial distribution being a Schur measure.

Another property that we need is the following commutativity.

\begin{proposition}
\label{prop_Schur_commutativity} The following commutation relation on matrices
$p^\uparrow_{\lambda\to\mu}$ and $p^\downarrow_{\lambda\to\mu}$ holds:
\begin{equation}
\label{eq_Schur_commutativity} p^{\uparrow}(\rho_1,\rho_2;\rho_3)  p^{\downarrow} (\rho_1;\rho_2)
=  p^{\downarrow}(\rho_1;\rho_2) p^{\uparrow}(\rho_1;\rho_3)
\end{equation}
\end{proposition}
\noindent{\bf Remark.} In terms of acting on Schur measures, as in Proposition
\ref{Prop_p_agrees_with_Schur}, \eqref{eq_Schur_commutativity} says that adding
$\rho_3$ and
then removing $\rho_2$ is the same as first removing $\rho_2$ and then adding
$\rho_3$:
$$
\begin{array}{ll}
  {\mathbb S}_{\rho_4;\rho_1,\rho_2}  p^{\uparrow}(\rho_1,\rho_2;\rho_3)=   {\mathbb
  S}_{\rho_3,\rho_4;\rho_1,\rho_2},&  {\mathbb
  S}_{\rho_3,\rho_4;\rho_1,\rho_2} p^{\downarrow} (\rho_1;\rho_2)= {\mathbb
  S}_{\rho_3,\rho_4;\rho_1},\\
  {\mathbb S}_{\rho_4;\rho_1,\rho_2}p^{\downarrow}(\rho_1;\rho_2)  = {\mathbb S}_{\rho_4;\rho_1},& {\mathbb S}_{\rho_4;\rho_1}
  p^{\uparrow}(\rho_1;\rho_3)  = {\mathbb
  S}_{\rho_3,\rho_4;\rho_1}.
\end{array}
$$
\begin{proof}[Proof of Proposition \ref{prop_Schur_commutativity}]
We should prove that for any $\lambda,\nu\in\Y$
$$
 \sum_{\mu\in\Y}  p^{\uparrow}_{\lambda\to\mu}(\rho_1,\rho_2;\rho_3) p^{\downarrow}_{\mu\to\nu} (\rho_1;\rho_2)
=\sum_{\mu\in\Y}  p^{\downarrow}_{\lambda\to\mu} (\rho_1;\rho_2)
p^{\uparrow}_{\mu\to\nu}(\rho_1;\rho_3).
$$
Using definitions \eqref{eq_p_up}, \eqref{eq_p_down} this boils down to the specialized version of
the skew Cauchy Identity, which is Proposition \ref{proposition_skew_Cauchy}, cf.\ \cite{B-Schur}.
\end{proof}

Commutativity relation \eqref{eq_Schur_commutativity} paves the way to introducing a family of new
Markov chains through a construction that we now present. This construction first appeared in
\cite{DF} and was heavily used recently for  probabilistic models related to Young diagrams, see
\cite{BF}, \cite{BG}, \cite{BGR}, \cite{BD}, \cite{B-Schur}, \cite{BO-GT},  \cite{Be}, \cite{BC}.

Take two Schur--positive specializations $\rho_1$, $\rho_2$ and a state space $\Y^{(2)}$ of
\emph{pairs} of Young diagrams $\lambda^{(2)}\choose \lambda^{(1)}$ such that
$p^{\downarrow}_{\lambda^{(2)}\to\lambda^{(1)}}(\rho_1;\rho_2)>0$. Define a Markov chain on
$\Y^{(2)}$ with the following transition probabilities:
\begin{equation}
\label{eq_two_levels_transition}
 \P\left({\lambda^{(2)}\choose \lambda^{(1)}} \to
{\mu^{(2)}\choose\mu^{(1)}}\right)=p^{\uparrow}_{\lambda^{(1)}\to\mu^{(1)}} (\rho_1;\rho')
\dfrac{p^{\uparrow}_{\lambda^{(2)}\to\mu^{(2)}}(\rho_1,\rho_2;\rho')
p^{\downarrow}_{\mu^{(2)}\to\mu^{(1)}}(\rho_1;\rho_2)}{\sum_{\mu}
p^{\uparrow}_{\lambda^{(2)}\to\mu}(\rho_1,\rho_2;\rho')
p^{\downarrow}_{\mu\to\mu^{(1)}}(\rho_1;\rho_2)}
\end{equation}
In words \eqref{eq_two_levels_transition} means that the first Young diagram
$\lambda^{(1)}\to\mu^{(1)}$ evolves according to the transition probabilities
$p^{\uparrow}_{\lambda^{(1)}\to\mu^{(1)}} (\rho_1;\rho')$, and given $\lambda^{(2)}$ and
$\mu^{(1)}$ the distribution of $\mu^{(2)}$ is the distribution of the middle point in the
two-step Markov chain with transitions $p^{\uparrow}_{\lambda^{(2)}\to\mu}(\rho_1,\rho_2;\rho')$
and $p^{\downarrow}_{\mu\to\mu^{(1)}}(\rho_1;\rho_2)$.

\begin{proposition}
\label{Proposition_Schur_preserved}
 The above transitional probabilities on $\Y^{(2)}$ map the Schur process with distribution
$$
 {\mathbb S}_{\rho_1,\rho_2;\rho^-}(\lambda^{(2)})
p^{\downarrow}_{\lambda^{(2)}\to\lambda^{(1)}}(\rho_1;\rho_2)
$$
to the Schur process with distribution
$$
 {\mathbb S}_{\rho_1,\rho_2;\rho^-,\rho'}(\lambda^{(2)})
p^{\downarrow}_{\lambda^{(2)}\to\lambda^{(1)}}(\rho_1;\rho_2).
$$
 Informally, the specialization $\rho'$ was added to $\rho^-$ and nothing
else changed.
\end{proposition}
\begin{proof}
 Direct computation based on Proposition \ref{prop_Schur_commutativity}, cf.\ \cite[Section 2.2]{BF}.
\end{proof}

More generally, we can iterate the above constructions and produce a Markov chain on sequences of
Young diagrams $\lambda^{(N)},\dots,\lambda^{(1)}$ as follows. The first Young diagram
$\lambda^{(1)}$ evolves according to transition probabilities
$p^{\uparrow}_{\lambda^{(1)}\to\mu^{(1)}} (\rho_1;\rho')$. Then, for any $k\ge 2$, as soon as
$\mu^{(k-1)}$ is defined and given $\lambda^{(k)}$ the distribution of $\mu^{(k)}$ is the
distribution of the middle point in the two-step Markov chain with transitions
$p^{\uparrow}_{\lambda^{(k)}\to\mu}(\rho_1,\dots,\rho_k;\rho')$ and
$p^{\downarrow}_{\mu\to\mu^{(k-1)}}(\rho_1,\dots,\rho_{k-1};\rho_k)$.

Similarly to Proposition \ref{Proposition_Schur_preserved} one proves that one step of thus
constructed Markov chain adds $\rho'$ to the specialization $\rho^-$ of the Schur process with
distribution
\begin{equation}
\label{eq_x11}
 {\mathbb S}_{\rho_1,\dots,\rho_N;\rho^-}(\lambda^{(N)})
p^{\downarrow}_{\lambda^{(N)}\to\lambda^{(N-1)}}(\rho_1,\dots,\rho_{N-1};\rho_N)\cdots
p^{\downarrow}_{\lambda^{(2)}\to\lambda^{(1)}}(\rho_1;\rho_2).
\end{equation}

The above constructions might look quite messy, so let us consider several examples, where they
lead to relatively simple Markov chains.

Take each $\rho_k$ to be the Schur--positive specialization with single non--zero parameter
$\alpha_1=1$, and let $\rho'$ be the Schur--positive specialization with single non-zero parameter
$\beta_1=b$. Consider a discrete time homogeneous Markov chain
$(\lambda^{(1)}(t),\dots,\lambda^{(N)}(t))$ with defined above transitional probabilities and
started from the Schur process as in \eqref{eq_x11} with $\rho^-$ being trivial specialization
(with all zero parameters). This implies that
$(\lambda^{(1)}(0),\dots,\lambda^{(N)}(0))=(\varnothing,\dots,\varnothing)$. Note that at any time
$t$ the Young diagram $\lambda^{(k)}(t)$ has at most $k$ non-empty rows and their coordinates
satisfy the following \emph{interlacing conditions}:
\begin{equation}
\label{eq_interlacing}
 \lambda^{(k)}_1\ge\lambda^{(k-1)}_1\ge\lambda^{(k)}_2\ge\dots\ge\lambda^{(k-1)}_{k-1}\ge\lambda^{(k)}_k.
\end{equation}
 In particular, $\lambda^{(1)}$ has a single
row, i.e.\ it is a number. The definitions imply that the transitional probabilities for
$\lambda^{(1)}$ are
$$
 p^{\uparrow}_{\lambda^{(1)}\to\mu^{(1)}} (\rho_1;\rho')=\begin{cases} \frac{b}{1+b},&\text{ if
 }\mu^{(1)}=\lambda^{(1)}+1,\\
 \frac{1}{1+b},&\text{ if
 }\mu^{(1)}=\lambda^{(1)},\\
 0,&\text{ otherwise.}
 \end{cases}
$$
In other words, the evolution of $\lambda^{(1)}$ is a simple Bernoulli random walk with
probability of move $p=b/(1+b)$. More generally, given that $\lambda^{(k)}(t)=\lambda$ and
$\lambda^{(k-1)}(t+1)=\mu$ the distribution of $\lambda^{(k)}(t+1)$ is given by

\begin{multline*}
 {\rm Prob} (\lambda^{(k)}(t+1)=\nu \mid \lambda^{(k)}(t)=\lambda,\lambda^{(k-1)}(t+1)=\mu)
 = \\
 = \dfrac{s_{\nu/\lambda}(0;b;0) s_{\nu/\mu}(1;0;0)}{\sum_\eta s_{\eta/\lambda}(0;b;0) s_{\eta/\mu}(1;0;0)}
 \sim b^{|\nu|-|\lambda|},
\end{multline*}
given that $0\le \nu_i -\lambda_i \le 1$, and $\mu$ and $\nu$ satisfy the interlacing conditions
\eqref{eq_interlacing}. In other words, the length of each row of $\lambda^{(k)}$ independently
increases by $1$ with probability $b/(1+b)$ unless this contradicts interlacing conditions (in
which case that length either stays the same or increases by $1$ with probability $1$).

In order to visualize the above transitional probabilities consider $N(N+1)/2$ interlacing
particles with integer coordinates $x_i^j=\lambda^{(j)}_{j+1-i}-N+i$, $j=1,\dots, N$,
$i=1,\dots,j$. In other words, for each $j$ we reorder the coordinates $\lambda^{(j)}_i$ and make
them strictly increasing. The coordinates of particles thus satisfy the inequalities
\begin{equation}\label{int_ineq}
 x_{i-1}^j < x_{i-1}^{j-1} \le x_{i}^{j}.
\end{equation}
Such arrays are often called Gelfand--Tsetlin patterns (or schemes) of size $N$ and under this
name they are widely used in representation-theoretic context. We typically use the notation
$x_i^j$ both for the location of a particle and the particle itself. It is convenient to put
particles $x_i^j$ on adjacent horizontal lines, as shown in Figure \ref{Figure_Interlacing}.
\begin{figure}[h]
\begin{center}
\noindent{\scalebox{1.3}{\includegraphics{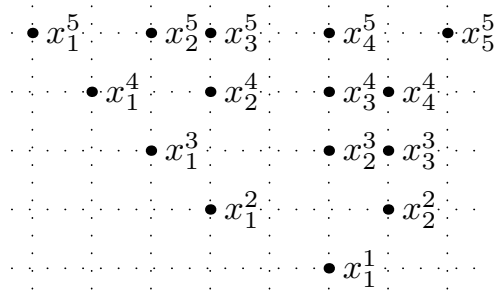}}} \caption{Interlacing particles.
\label{Figure_Interlacing} }
\end{center}
\end{figure}

Now the dynamics has the following description. At each time $t$ each of the $N(N+1)/2$ particles
flips a biased coin with probability of ``heads'' $p=b/(1+b)$ (all flips are independent). After
that positions $x_i^j(t+1)$ of particles at time $t+1$ are obtained by the following
\emph{sequential update}. First, we define $x_1^1(t+1)$, then $x_i^2(t+1)$, $i=1,2$, then
$x_i^3(t+1)$, $i=1,2,3$, etc. To start with, we set $x_1^1(t+1)=x_1^1(t)+1$, if the coin of the
particle $x_1^1$ came ``heads'', otherwise, $x_1^1(t+1)=x_1^1(t)$. Subsequently, once the values
of $x^j_i(t+1)$, $j=1,\dots,k-1$, $i=1,\dots,j$ are determined, we define $x_i^k(t+1)$ for each
$i=1,\dots,k$ independently by the following procedure, in which each rule 3 is performed only if
the conditions for the previous two are \textit{not} satisfied.
\begin{enumerate}
\item If $i>1$ and $x_i^k(t)=x_{i-1}^{k-1}(t+1)-1$,
 then we say that particle $x_i^k$ is \emph{pushed} by $x_{i-1}^{k-1}$ and
 set $x_{i}^{k}(t+1)=x_i^k(t)+1$.
\item If $x_{i}^{k-1}(t+1)=x_i^k(t)+1$, then we say that particle
 $x_i^k$ is \emph{blocked} by $x_i^{k-1}$ and set $x_i^k(t+1)=x_i^k(t)$.
\item If the coin of the particle $x_i^k$ came ``heads'', then we set $x_i^k(t+1)=x_i^k(t)+1$, otherwise,
we set $x_i^k(t+1)=x_i^k(t)$.
\end{enumerate}

Informally, one can think of each particle having  a weight depending on its vertical coordinate,
with higher particles being lighter. The dynamics is defined in such a way that heavier particles
push lighter ones and lighter ones are blocked by heavier ones in order to preserve the
interlacement conditions.

One interesting fact about this dynamics is that the joint distribution of $x_i^j(t+j)$ is the
projection of the uniform distribution on the domino tilings of the so-called Aztec diamond, see
\cite{N}, \cite{BF}. A computer simulation showing the connection with tilings can be found on
Ferrari's website \cite{Fe}.

\medskip

Sending $b\to 0$ and rescaling the time we get a continuous version of the above dynamics.
Formally, the process $Y(t)=\{Y_i^j(t)\}$, $t\geq0$, $j=1,\dots, N$, $i=1,\dots,j$, is a
continuous-time Markov chain defined as follows. Each of the $N(N+1)/2$ particles has an
independent exponential clock of rate $1$ (in other words, the times when clocks of particles ring
are independent standard Poisson processes). If the clock of the particle $Y_i^j$ rings at time
$t$, we check whether $Y_{i}^{j-1}(t-)=Y_{i}^j(t-)+1$. If so, then nothing happens; otherwise we
let $Y_i^j(t)=Y_i^j(t-)+1$. If $Y_i^j(t-)=Y_{i+1}^{j+1}(t-)=\dots=Y_{i+k}^{j+k}(t-)$, then we also
set $Y_{i+1}^{j+1}(t)=Y_{i+1}^{j+1}(t-)+1,\dots,Y_{i+k}^{j+k}(t)=Y_{i+k}^{j+k}(t-)+1$.

The Markov chain $Y$ was introduced in \cite{BF} as an example of a 2d growth model relating
classical interacting particle systems and random surfaces arising from dimers. The computer
simulation of $Y(t)$ can be found at Ferrari's website \cite{Fe}. The restriction of $Y$ to the
$N$ leftmost particles $Y_1^1,\dots,Y_1^N$ is the familiar totally asymmetric simple exclusion
process (TASEP), the restriction $Y$ to the $N$ rightmost particles $Y_1^1,\dots,Y_N^N$ is long
range TASEP (or PushASEP), while the particle configuration $Y(t)$ at a fixed time $t$ can be
identified with a lozenge tiling of a sector in the plane and with a stepped surface (see the
introduction in \cite{BF} for the details).

One proves that the fixed time distribution of $Y(t)$ is the Schur process of Section
\ref{Section_RSK_dyn} with $a_i=1$ (under the above identification of particle configurations and
sequences of Young diagrams). Moreover, the restriction of $Y(t)$ on $N$ leftmost particles
$Y_1^1,Y_1^2\dots,Y_1^N$ is the same as the restriction of the dynamics of Section
\ref{Section_RSK_dyn} and similar statement is true for restrictions on $Y_1^1, Y_2^2\dots,Y_N^N$
and on $Y_1^N, Y_2^N,\dots,Y_N^N$. Nevertheless, the dynamics $Y(t)$ is \emph{different} from that
of Section \ref{Section_RSK_dyn}.

By appropriate limit transition we can also make the state space of our dynamics continuous.

\begin{theorem}[\cite{GS}]
\label{theorem_convergence_to_Warren} The law of $N(N+1)/2$ dimensional stochastic process
$(Y(Lt)-tL)/\sqrt{L}$, $t\geq0$ converges in the limit $L\to\infty$ to the law of a continuous
time-space process $W(t)$.
\end{theorem}
Note that for $N=1$ Theorem \ref{theorem_convergence_to_Warren} is a classical statement (known as
Donsker invariance principle) on the convergence of Poisson process towards the Brownian motion
(see e.g.\ \cite[Section 37]{Bi} or \cite[Section 12]{Kal}).

The process $W(t)$ was introduced in \cite{W} and has an independent probabilistic
description: $W_1^1$ is the standard Brownian motion; given $W_i^j$ with $j<k$ the
coordinate $W_i^k$ is defined as the Brownian motion \emph{reflected} by the
trajectories $W_{i-1}^{k-1}$ (if $i>1$) and $W_{i}^{k-1}$ (if $i<k$) (see \cite{W}
for a more detailed definition). The process $W$ has many interesting properties:
For instance, its projection $W_i^N$, $i=1,\dots,N$, is the $N$-dimensional Dyson's
Brownian motion, namely the process of $N$ independent Brownian motions conditioned
to never collide with each other. The fixed time distribution of vector $W_i^N$,
$i=1,\dots,N$, can be identified with the distribution of eigenvalues of the random
Hermitian matrix from the Gaussian Unitary Ensemble (it was introduced after
Definition \ref{Definition_biorth} above), and the fixed--time distribution of the
whole process $W(t)$ is the so--called GUE--corners process\footnote{The name
GUE--minors process is also used.}, cf.\ \cite{Bar}, \cite{JN}.

\bigskip

Our construction of Markov dynamics can be adapted to the measures on (skew) plane partitions of
Section \ref{Section_plane_partitions}. In particular, this leads to an efficient \emph{perfect
sampling algorithm}, which allows to visualize how a typical random (skew) plane partition looks
like, see \cite{B-Schur} for details and Figure \ref{Figure_skew_sample} for a
result of a computer
simulation.
\begin{figure}[h]
\begin{center}
\noindent{\scalebox{0.7}{\includegraphics{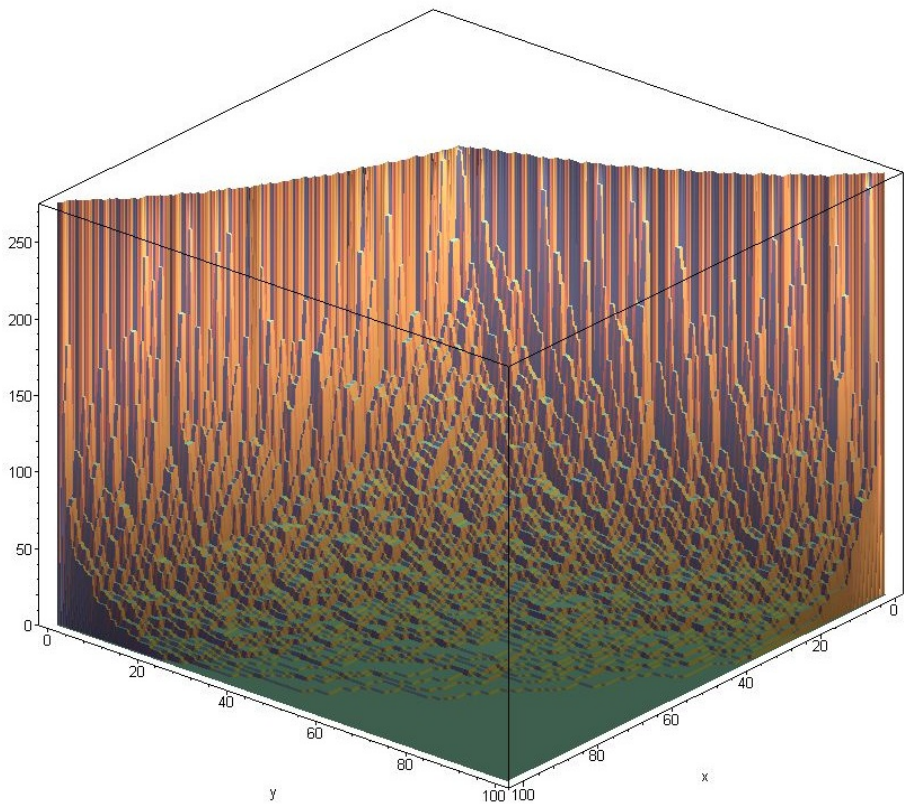}}}
\caption{Sample from the measure $q^{\rm volume}$ on skew plane partitions with one particular
choice of support. \label{Figure_skew_sample}}
\end{center}
\end{figure}
 Further generalizations include, in particular, various
probability distributions on \emph{boxed} plane partitions (see \cite{BG}, \cite{BGR}, \cite{Be}])
and directed polymers in random media which we discuss in the next section.

\bigskip

One feature that all of the above Markov chains share is that the transitional probabilities
decompose into relatively simple $1d$ distributions. In other words, each step of the dynamics
boils down to sampling from several explicit distributions and applying not very complicated
combinatorial rules. But in the same time, after we do several steps, we arrive at highly
non-trivial probabilistic objects of independent interest, such as random plane partitions or
eigenvalues of random matrices.

\section{Macdonald polynomials and directed polymers}

\label{Section_Macdonald}

In this section we generalize the constructions of Sections \ref{Section_Schur_measures},
\ref{Section_Schur_process} based on the Schur functions to their $(q,t)$ deformation known as
\emph{Macdonald symmetric functions} and discuss new applications.

\subsection{Preliminaries on Macdonald polynomials}
In what follows we assume that $q$ and $t$ are real numbers satisfying $0<q<1$, $0<t<1$.

One way to define Macdonald symmetric functions $P_\lambda(\,\cdot\,;q,t)$ indexed by Young
diagrams $\lambda$ is through the Gram--Schmidt orthogonalization procedure. Define the following
scalar product on $\Lambda$ via its values on the basis of power sums as
$$
 \left\langle p_\lambda, p_\mu \right\rangle_{q,t} = \delta_{\lambda=\mu} z_\lambda
 \prod_{i=1}^{\ell(\lambda)}\frac{1-q^{\lambda_i}}{1-t^{\lambda_i}},
$$
where $z_\lambda$, $p_\lambda$ are as in Theorem \ref{theorem_Cauchy}.
\begin{definition}
 Macdonald symmetric functions $P_\lambda(\,\cdot\,;q,t)$ form a unique linear basis in $\Lambda$ such that
 \begin{enumerate}
  \item The leading coefficient of $P_\lambda(\,\cdot\,;q,t)$ (with respect to lexicographic ordering on
  monomials) is $x_1^{\lambda_1} x_2^{\lambda_2}\cdots$.
  \item $\left\langle P_\lambda(\,\cdot\,;q,t), P_\mu(\,\cdot\,;q,t)\right\rangle_{q,t} =0$, unless
  $\lambda=\mu$.
 \end{enumerate}
\end{definition}
We omit the dependence on $(q,t)$ and write simply $P_\lambda(\,\cdot\,)$ when it leads to no
confusion.

Since $\Lambda$ is a projective limit of algebras $\Lambda_N$ of symmetric polynomials in $N$
variables $x_1,\dots,x_N$, the functions $P_\lambda$ automatically define Macdonald symmetric
polynomials in $N$ variables $P_\lambda(x_1,\dots,x_N;q,t)$. The latter can be characterized (for
generic $q$, $t$)  by being eigenvectors of certain difference operators.

\begin{definition} The $r$th Macdonald difference operator $\mathcal D_r^N$ is defined as
$$
 \mathcal D_r^N = \sum_{I\subset \{1,\dots,N\},\, |I|=r} A_I(x;t) \prod_{i\in I} T_i,
$$
where
$$
 A_I(x;t)=t^{r(r-1)/2} \prod_{i\in I,\, j\notin I}\frac{tx_i-x_j}{x_i-x_j},
$$
and $T_i$ is the $q$--shift operator in variable $x_i$
$$
 (T_i f)(x_1,\dots,x_N)=f(x_1,\dots,x_{i-1},qx_i,x_{i+1},\dots,x_N).
$$
In particular,
$$
 (\mathcal D_1^N f)(x_1,\dots,x_N)=\sum_{i=1}^N \prod_{j\ne i} \frac{tx_i-x_j}{x_i-x_j}
 f(x_1,\dots,x_{i-1},qx_i,x_{i+1},\dots,x_N).
$$
\end{definition}

\begin{proposition}
\label{Proposition_Macdonald_difference}
 For any $N\ge 1$, polynomials $P_\lambda(x_1,\dots,x_N;q,t)$ with $\ell(\lambda)\le N$ form a common eigenbasis
 of operators $\mathcal D_r^N$, $r=1,\dots,N$. More precisely,
 $$
  \mathcal D_r^N P_\lambda(x_1,\dots,x_N;q,t) =
  e_r(q^{\lambda_1}t^{N-1},q^{\lambda_2}t^{N-2},\dots,q^{\lambda_N}) P_\lambda(x_1,\dots,x_N;q,t),
 $$
 where $e_r$ are the elementary symmetric polynomials.
 In particular,
  $$
  \mathcal D_1^N P_\lambda(x_1,\dots,x_N;q,t) =
  (q^{\lambda_1}t^{N-1}+q^{\lambda_2}t^{N-2}+\dots+q^{\lambda_N}) P_\lambda(x_1,\dots,x_N;q,t).
 $$
\end{proposition}
Proposition \ref{Proposition_Macdonald_difference} can be taken as an alternative definition of
Macdonald polynomials. Both these difference operators and polynomials themselves were introduced
by I.~Macdonald in the late 80s \cite{M-new}. Macdonald polynomials generalized  various
previously known classes of symmetric polynomials. When $q=t$ we get Schur polynomials, when $q=0$
we get \emph{Hall--Littlewood} polynomials, and if we set $q=t^{a}$ and send $t$ to $1$ then we
get \emph{Jack} polynomials. We follow the notations of Macdonald's book \cite{M} where an
interested reader can find proofs of various properties of Macdonald polynomials that we use.

Macdonald polynomials inherit many of their properties from Schur polynomials. In particular,
there exist Macdonald versions of skew functions and Cauchy identity which we used for the
definition of the Schur measure (process) and for the construction of the dynamics in Section
\ref{Section_Schur_process}.

Set
$$
Q_\lambda(\,\cdot\,;q,t)=\frac{P_\lambda(\,\cdot\,;q,t)}{ \left\langle P_\lambda(\,\cdot\,;q,t),
P_\lambda(\,\cdot\,;q,t)\right\rangle_{q,t}}.
$$
The normalizing constant $\left\langle P_\lambda(\,\cdot\,;q,t),
P_\lambda(\,\cdot\,;q,t)\right\rangle_{q,t}$ can be computed explicitly, see \cite[Chapter VI]{M}.

The following statement is a $(q,t)$--analogue of the Cauchy identity:
\begin{proposition} \label{Proposition_Macdonald_Cauchy} Let $x$ and $y$ be two sets of variables. Then
$$
 \sum_{\lambda\in\Y} P_\lambda(x;q,t) Q_\lambda(y;q,t) =\prod_{i,j} \frac{(tx_iy_j;q)_{\infty}}{(x_i
 y_k;q)_{\infty}},
 $$
 where
 $$
  (a;q)_\infty =\prod_{i=0}^{\infty}(1-a q^i).
 $$
\end{proposition}

Similarly to Definition \ref{Definition_skew_Schur} we have
\begin{definition}
Let $x$ and $y$ be two sets of variables. The skew functions $P_{\lambda/\mu}$ and
$Q_{\lambda/\mu}$ are defined via
$$
P_{\lambda}(x,y)=\sum_{\mu\in\Y} P_{\lambda/\mu}(x) P_\mu(y),
$$
and
$$
Q_{\lambda}(x,y)=\sum_{\mu\in\Y} Q_{\lambda/\mu}(x) Q_\mu(y).
$$
\end{definition}
There is also a $(q,t)$ analogue of the skew Cauchy identity, see \cite[Chapter VI, Section 7]{M}.

Another ingredient of our constructions related to Schur functions was the classification of
positive specializations of Theorem \ref{Theorem_Thoma}. For Macdonald polynomials the following
conjecture is due to Kerov:
\begin{conjecture}
\label{Conjecture_Kerov}
 The Macdonald--positive specializations\footnote{A specialization $\rho$ of $\Lambda$ is called Macdonald--positive
 if $P_\lambda(\rho)\ge 0$ for any $\lambda\in\Y$ (we assume ${0<q,t<1}$).} are parameterized by pairs of sequences
of non-negative reals $\alpha=(\alpha_1\ge\alpha_2\ge\dots\ge 0)$ and
$\beta=(\beta_1\ge\beta_2\ge\dots\ge 0)$ satisfying $\sum_i(\alpha_i+\beta_i)<\infty$ and an
additional parameter $\gamma\ge 0$. The specialization with parameters $(\alpha;\beta;\gamma)$ can
be described by its values on power sums
$$
p_1\mapsto p_1(\alpha;\beta;\gamma)=\frac{1-q}{1-t}\gamma+\sum_i
\left(\alpha_i+\frac{1-q}{1-t}\beta_i\right),
$$
$$
 p_k\mapsto
p_k(\alpha;\beta;\gamma)=\sum\limits_{i}
 \left(\alpha_i^k + (-1)^{k-1}\frac{1-q^k}{1-t^k}\beta_i^{k}\right), \quad k\ge 2,
$$
or, equivalently, via generating functions
$$
\sum\limits_{k=0}^\infty g_{k}(\alpha;\beta;\gamma) z^k = e^{\gamma z} \prod\limits_{i\ge 1}
(1+\beta_i z)\frac{(t\alpha_i z;q)_\infty}{(\alpha_i z;q)_\infty},
$$
where $g_k$ is the one-row $Q$--function $$ g_k(\,\cdot\,)=Q_{(k)}(\,\cdot\,;q,t).$$
\end{conjecture}
\noindent {\bf Remark 1.} The fact that all of the above specializations are non-negative on
Macdonald symmetric functions $P_\lambda$ is relatively simple, see e.g.\ \cite[Section
2.2.1]{BC}. However, the completeness of the list given in Conjecture \ref{Conjecture_Kerov} is a
hard open problem known as Kerov's conjecture and stated in \cite[Section II.9]{Kerov_book}.

\noindent{\bf Remark 2.} One can show that all Macdonald--positive specializations given in
Conjecture \ref{Conjecture_Kerov} also take non-negative values on skew Macdonald symmetric
functions.

\medskip

Continuing the analogy with Schur functions we define \emph{Macdonald measures}.

\begin{definition}
\label{Def_Macdonald_meas}
 Given two Macdonald--nonnegative specializations $\rho_1$, $\rho_2$,
 \emph{the Macdonald measure} $\mathbb M_{\rho_1;\rho_2}$ is a probability measure on the set of all Young diagrams defined
 through
 $$
 \P_{\rho_1,\rho_2}(\lambda)=\dfrac{P_\lambda(\rho_1) Q_\lambda(\rho_2)}{H_{q,t}(\rho_1;\rho_2)},
 $$
 where the normalizing constant $H_{q,t}(\rho_1;\rho_2)$ is given by
$$
 H_{q,t}(\rho_1;\rho_2)=\exp\left(\sum_{k=1}^{\infty} \frac{1-t^k}{1-q^k}\cdot \frac{p_k(\rho_1)p_k(\rho_2)}{k}\right).
$$
\end{definition}
\noindent{\bf Remark 1.}
 The above definition makes sense only if $\rho_1$, $\rho_2$ are such that
 \begin{equation}
 \label{eq_sum_Macdonald}
  \sum_\lambda P_\lambda(\rho_1) Q_\lambda(\rho_2)<\infty,
 \end{equation}
in which case this sum equals $H_{q,t}(\rho_1;\rho_2)$, as follows from Proposition
\ref{Proposition_Macdonald_Cauchy}. The convergence of \eqref{eq_sum_Macdonald} is guaranteed, for
instance, if $|p_k(\rho_1)|<C r^k$ and $|p_k(\rho_2)|< C r^k$ with some constants $C>0$ and
$0<r<1$.

\noindent{\bf Remark 2.} The definition of Macdonald measure was first given almost ten years ago
by Forrester and Rains, see \cite{FR}. In addition to the Schur ($q=t$) case, Macdonald measures
(and processes) were also studied by Vuletic \cite{Vu} for the Hall--Littlewood symmetric
functions, which correspond to $q=0$. Recently, new applications of these measures and new tools
to work with them were developed starting from \cite{BC}.

\medskip

Definition \ref{Def_Schur_proc} of the Schur process also has a straightforward $(q,t)$--analogue
involving Macdonald polynomials. In our applications we will use only one particular case of this
definitions which is a $(q,t)$ generalization of measures of Section \ref{Section_RSK_dyn}.
\begin{definition}
\label{Definition_Macdonald_Ascending} Given Macdonald--nonnegative specializations
$\rho_1^+,\dots,\rho_N^+$ and $\rho^-$, the \emph{ascending Macdonald process} is the probability
distribution on sequences of Young diagrams $\lambda^{(1)},\dots,\lambda^{(N)}$ defined via
$$
\P(\lambda^{(1)},\dots,\lambda^{(N)})= \dfrac{ P_{\lambda^{(1)}}(\rho_1^+)
 P_{\lambda^{(2)}/\lambda^{(1)}}(\rho_2^+)\cdots P_{\lambda^{(N)}/\lambda^{(N-1)}}(\rho_N^+)
 Q_{\lambda(N)}(\rho^-)}{H_{q,t}(\rho_1^+;\rho^-)\cdots H_{q,t}(\rho_N^+;\rho^-)}.
$$
\end{definition}

\subsection{Probabilistic models related to Macdonald polynomials}

\label{Section_Macdonald_probability}

The construction of Section \ref{Section_Markov_chains} of dynamics preserving the class of Schur
processes can be literally repeated for Macdonald processes by replacing all instances of skew
Schur polynomials by skew Macdonald polynomials, see \cite[Section 2.3]{BC} for the details.

Let us focus on one example. We set $t=0$ till the end of this section and, thus, Macdonald
functions are replaced with their degeneration known as $q$--Whittaker functions.

We continue the analogy with Section \ref{Section_Markov_chains} and study the example which led
to the process $Y(\tau)$ (we replaced the time parameter by $\tau$, since $t$ now has a different
meaning). Namely, we set $\rho_i^+$ to be the specialization with single non-zero parameter
$\alpha_1=1$ in Definition \ref{Definition_Macdonald_Ascending} and $\rho^-$ to be the
specialization with single non-zero $\beta_1=b$. At each step of our (discrete time) dynamics
another $\beta$ equal to $b$ is added to the specialization $\rho^-$. Sending $b\to 0$ and
rescaling time we arrive at the continuous time dynamics $Z(\tau)$ which should be viewed as a
$q$--analogue of $Y(\tau)$. Let us give a probabilistic description of $Z(\tau)$.

The process $Z(\tau)=\{Z_i^j(\tau)\}$, $t\geq0$, $j=1,\dots, N$, $i=1,\dots,j$, is a continuous
time Markov evolution of $N(N+1)/2$ particles with its coordinates being integers satisfying the
interlacing conditions
$$Z_{i-1}^j(\tau) < Z_{i-1}^{j-1}(\tau) \le Z_{i}^{j}(\tau)$$
for all meaningful $i,j$, and defined as follows. Each of the $N(N+1)/2$ particles has an
independent exponential clock.  The clock rate of particle $Z_i^j$ at time $\tau$ is
\begin{equation}
\label{eq_q_TASEP_GT}
 \dfrac{(1-q^{Z^{j-1}_i(\tau)-Z^j_i(\tau)-1})(1-q^{Z^j_i(\tau)-Z^j_{i-1}(\tau)})}{1-q^{Z^j_i(\tau)-Z^{j-1}_{i-1}(\tau)+1}},
\end{equation}
in other words, this rate depends on the positions of three neighbors, as shown in Figure
\ref{Figure_neighbors}. If one of the neighbors does not exist, then the corresponding factor
disappears from \eqref{eq_q_TASEP_GT}. When the clock of particle $Z_i^j$ rings at time $\tau$, we
let $Z_i^j(\tau)=Z_i^j(\tau-)+1$. If
$Z_i^j(\tau-)=Z_{i+1}^{j+1}(\tau-)=\dots=Z_{i+k}^{j+k}(\tau-)$, then we also set
$Z_{i+1}^{j+1}(\tau)=Z_{i+1}^{j+1}(\tau-)+1,\dots,Z_{i+k}^{j+k}(\tau)=Z_{i+k}^{j+k}(\tau-)+1$.
Equivalently, one can think that when \eqref{eq_q_TASEP_GT} becomes infinite because of the
denominator vanishing, the corresponding particle immediately moves by one to the right. Note that
if $Z^{j-1}_i=Z^j_i+1$, then the rate \eqref{eq_q_TASEP_GT} vanishes, therefore the blocking which
was present in the definition of $Y(\tau)$ is also implied by the definition of $Z(\tau)$.

\begin{figure}[h]
\begin{center}
\noindent{\scalebox{1.0}{\includegraphics{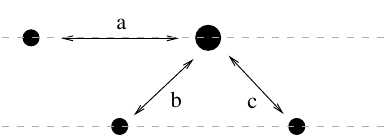}}} \hskip 1cm
{\scalebox{1.3}{\includegraphics{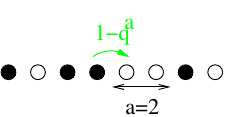}}} \caption{Left panel: The
distances to three
neighbors which the rate of the particle depends on. Right panel: jump rate for $q$--TASEP.
\label{Figure_neighbors} }
\end{center}
\end{figure}

The restriction of $Z(\tau)$ to the leftmost particles $Z_1^1(\tau),\dots,Z_1^N(\tau)$ is a
Markovian dynamics known as $q$--TASEP (see \cite{BC}). Namely, the rate of particle $Z_1^i$ at
time $\tau$ is $1-q^{Z^{j-1}_1(\tau)-Z^j_1(\tau)-1}$, as shown in Figure \ref{Figure_neighbors}.
When $q\to 0$ we recover the familiar TASEP.

There is one distinction from $Y(\tau)$, namely, the restriction of $Z(\tau)$ on the rightmost
particles is no longer a simple Markov process.

There are various interesting limit transitions as $\tau\to\infty$ and $q\to 1$. Let us
concentrate on one where a phenomenon known as \emph{crystallization} appears. Let
$$
 q=\exp(-\varepsilon),\quad \tau=\frac{\mathbf t}{\varepsilon^2},
$$
and send $\varepsilon$ to $0$. In this limit, particles $Z_i^j(\tau)$ approximate a perfect
lattice, namely,
$$
 Z_i^j(\tau)\approx \frac{\mathbf t}{\varepsilon^2} - \frac{\ln \varepsilon}{\varepsilon}
 (2i-j-1).
$$
The fluctuations around the points of this lattice are of order $\varepsilon^{-1}$. More
precisely, the following theorem holds.
\begin{theorem}
\label{Theorem_convergence_to_polymer} Let
 $$
 \widehat Z_i^j(\mathbf t,\varepsilon)=\varepsilon\left(Z_i^j(\tau)- \frac{\mathbf t}{\varepsilon^2} + \frac{\ln \varepsilon}{\varepsilon}
 (2i-j-1)\right),\quad q=\exp(-\varepsilon),\quad \tau=\frac{\mathbf t}{\varepsilon^2}.
 $$
 Then as $\varepsilon\to 0$, the stochastic process $\widehat Z(\tau,\varepsilon)$ weakly converges to a certain random vector $T(\mathbf t)$.
\end{theorem}
 The proof of Theorem \ref{Theorem_convergence_to_polymer} can be found in \cite{BC}. The
fixed time distribution of $T_i^j(\mathbf t)$ is known as the Whittaker process (first introduced
in \cite{OC}) and it is related to yet another class of symmetric functions (which are no longer
polynomials) that can be obtained from Macdonald symmetric functions through a limit transition.
These are class one $\mathfrak {gl}_N$--Whittaker functions, see e.g.\ \cite{Ko}, \cite{Et},
\cite{GLO}.

The stochastic dynamics $T_i^j(\mathbf t)$ can be also independently defined via certain
relatively simple stochastic differential equations, see \cite{OC}.

The interest in $T_i^j(\mathbf t)$ is caused by its remarkable connections with directed polymers
found by O'Connell and Yor \cite{OCY}, \cite{OC}.
\begin{theorem}
 Consider $N$ independent Brownian motions $B_1(s)$,\dots,$B_N(s)$. The distribution of $\exp(T_1^N(\mathbf t))$ coincides with that of the $N-1$--dimensional integral
 \begin{equation}
 \label{eq_OCon_Yor_polymer}
  \int\limits_{0<s_1<s_2<\dots<s_{N-1}<t} \exp\bigg(B_1(0,s_1) + B_2(s_1,s_2)+\dots+ B_N(s_{N-1},\mathbf
  t) \bigg) ds_1\cdots ds_{N-1},
 \end{equation}
 where
 $$
  B_i(a,b)=B_i(b)-B_i(a).
 $$
\end{theorem}
The integral \eqref{eq_OCon_Yor_polymer} is known as the partition function of the O'Connell--Yor
semidiscrete polymer. This has the following explanation. Identify each sequence
$0<s_1<s_2<\dots<s_{N-1}<\mathbf t$ with a  piecewise--constant monotonous function with unit
increments at $s_i$ or, equivalently, with staircase--like path (``directed polymer'') joining
$(0,1)$ and $(\mathbf t, N)$, as shown in Figure \ref{Figure_polymer}. Further, we view $B_i(a,b)$
as an integral of the $1d$ white noise along the interval $(a,b)$. Then the sum
$$
B_1(0,s_1) + B_2(s_1,s_2)+\dots+ B_N(s_{N-1},\mathbf
  t)
$$
turns into the integral of space--time white noise along the path defined by
$0<s_1<s_2<\dots<s_{N-1}<t$. Integrating over all possible paths we arrive at the partition
function, see \cite{BC}, \cite{BCF} and references therein for more details.

\begin{figure}[h]
\begin{center}
\noindent{\scalebox{1.0}{\includegraphics{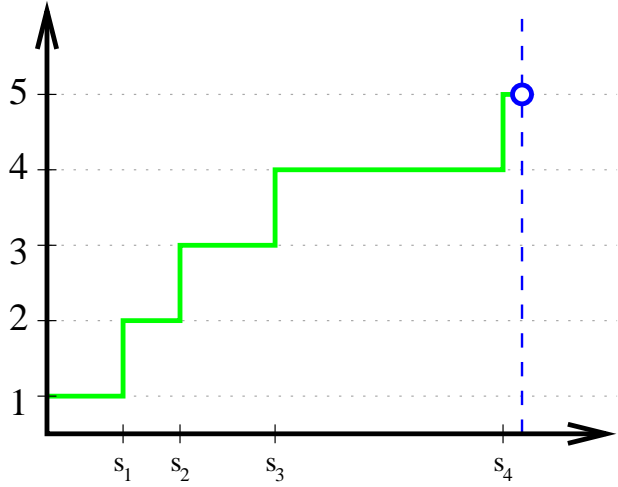}}}  \caption{Staircase--like path
corresponding to the sequence $0<s_1<s_2<\dots<s_{N-1}<\mathbf t$. Here $N=5$.
\label{Figure_polymer} }
\end{center}
\end{figure}

\subsection{Analysis of Macdonald measures}

In the previous section we explained that the formalism of Macdonald processes leads to quite
interesting probability models, but we do not know yet how to analyze them. The methods we used
for the Schur processes were based on the determinantal point processes. Similar determinantal
structure is not known to exist for the Macdonald processes and, probably, there is no such
structure at all. However, there is another approach based on the Macdonald difference operators.

We start from the $(q,t)$--Cauchy identity of Proposition \ref{Proposition_Macdonald_Cauchy} in
the form
\begin{equation}
\label{eq_qt_Cauchy}
 \sum_{\lambda\in\Y} P_\lambda(x_1,\dots,x_N;q,t) Q_\lambda(Y;q,t) =\prod_{i=1}^N \Pi(x_i;Y),
\end{equation}
where
$$
 \Pi(x;Y)= \prod_j \frac{(txy_j;q)_{\infty}}{(x
 y_k;q)_{\infty}}=\exp\left(\sum_{k=1}^{\infty} \frac{1-t^k}{1-q^k} \frac{p_k(Y) x^k }{k} \right).
$$
Let us apply the Macdonald difference operator ${\mathcal D}^r_N$ to both sides of
\eqref{eq_qt_Cauchy} with respect to the variables $x_1,\dots,x_N$. Proposition
\ref{Proposition_Macdonald_difference} yields
$$
 \sum_{\lambda\in\Y} e_r(q^{\lambda_1} t^{N-1},\dots,q^{\lambda_N})
 P_\lambda(x_1,\dots,x_N;q,t) Q_\lambda(Y;q,t) = {\mathcal D}^r_N\left(\prod_{i=1}^N \Pi(x_i;Y)\right).
$$
We immediately arrive at the following statement.
\begin{proposition}
\label{Proposition_observable}
 Let $\lambda\in\Y$ be distributed according to the Macdonald measure $\mathbb M_{\rho_1,\rho_2}$,
 where $\rho_1$ is the specialization with finitely many $\alpha$--parameters $a_1,\dots,a_N$.
 Then for $r=1,\dots,N$
 \begin{equation}
 \label{eq_observable}
  \E \bigg(e_r(q^{\lambda_1} t^{N-1},\dots,q^{\lambda_N})\bigg)=
    \frac{{\mathcal D}^r_N\left(\prod_{i=1}^N \Pi(x_i;\rho_2)\right)}{\prod_{i=1}^N
   \Pi(x_i;\rho_2)}\rule[-5.3mm]{.7pt}{13mm}_{\, x_i=a_i,\,\, 1\le i \le N,}
 \end{equation}
 where
 $$
  \Pi(x;\rho)= \exp\left(\sum_{k=1}^{\infty} \frac{1-t^k}{1-q^k} \frac{p_k(\rho)x^k}{k}
  \right).
 $$
\end{proposition}
{\bf Remark.} Clearly, we can replace ${\mathcal D}^r_N$ with any product of such operators,
similarly replacing $e_r(q^{\lambda_1} t^{N-1},\dots,q^{\lambda_N})$ with the corresponding
product, and the statement will still be valid.

\medskip

The remarkable property of Proposition \ref{Proposition_observable} is that while the Macdonald
polynomials themselves are fairly mysterious objects with complicated definition, the right--hand
side of \eqref{eq_observable} is explicit. Therefore, we have the formulas for averages of certain
\emph{observables} of Macdonald measures. These formulas can be compactly written via contour
integrals. Let us present one particular case of such formulas.

\begin{theorem}
\label{theorem_Observable_contour}
 Assume that, in the notations of Proposition \ref{Proposition_observable}, $t=0$,
all parameters $a_i$ are equal to $1$, and $\rho_2$ is the Macdonald--positive specialization with
single non-zero parameter $\gamma=\tau$. In other words, we deal with probability measure
$$
 e^{-N\tau} P_\lambda(1,\dots,1;q,0) Q_\lambda((0;0;\tau);q,0).
$$
Then
$$
 \E \Bigl(q^{k\lambda_N}\Bigr)= \frac{(-1)^k
q^{\frac{k(k-1)}{2}}}{(2\pi \i)^k} \oint\dots\oint \prod_{A<B}
 \frac{z_A-z_B}{z_A-q z_B} \prod_{j=1}^k \frac{e^{(q-1)\tau z_j}}{(1-z_j)^N} \frac{d z_j}{z_j},
$$
where the integration is over the \emph{nested} contours such that $z_j$--contour contains $1$ and
also $qz_{j+1},\dots,qz_k$, and no other singularities of the integrand, see Figure
\ref{Fig_nested}.
\end{theorem}
{\bf Remark 1.} The measure of Theorem \ref{theorem_Observable_contour} via the identification of
Young diagrams with $N$--rows and $N$--point particle configurations coincides with the
distribution of the vector $Z_1^N(\tau),\dots, Z_N^N(\tau)$ of the stochastic dynamics of Section
\ref{Section_Macdonald_probability}.

{\bf Remark 2.} Theorem \ref{theorem_Observable_contour} admits various generalizations: $t$ can
be non-zero, both specializations can be arbitrary, we can compute expectations related to the
higher order Macdonald operators and also apply it to the joint distributions of various Young
diagrams of the Macdonald process. See \cite{BC}, and \cite{BCGS} for details.

\medskip

\begin{figure}[h]
\begin{center}
\noindent{\scalebox{1.0}{\includegraphics{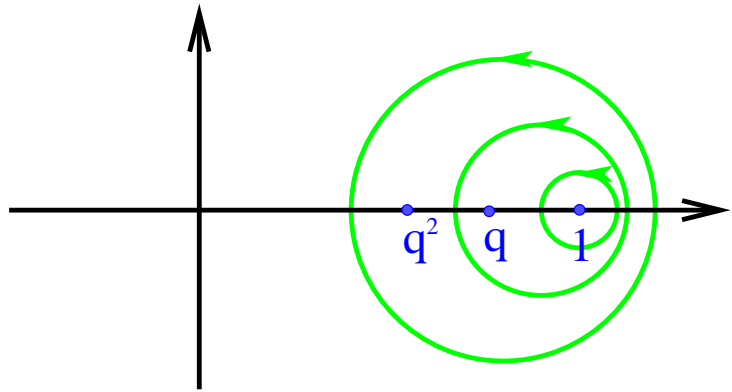}}}  \caption{Nested
contours of
integration. Here $N=3$, $z_1$ is integrated over the largest contour and $z_3$ is integrated over
the smallest one. \label{Fig_nested}}
\end{center}
\end{figure}

The moments of $q^{\lambda_N}$ can be combined into a $q$--Laplace transform for which we can also
get a neat expression in the form of a \emph{Fredholm determinant}.
\begin{theorem} \label{Theorem_Fredholm}In the notations of Theorem \ref{theorem_Observable_contour}, for all $\zeta\in \mathbb C\setminus \mathbb R_+$
we have
$$
 \E\left(\frac{1}{(\zeta q^{\lambda_N}; q)_\infty}\right) =\det(\mathbf I+
 K_\zeta)_{L^2(C_\omega)},
$$
where $C_\omega$ is a positively oriented small circle around $1$, and the operator $K_\zeta$ is
an integral operator defined in terms of its kernel
$$
 K_\zeta(w,w')=\frac{1}{2\pi \i} \int_{-\i\infty+1/2}^{\i\infty+1/2} \Gamma(-s)
\Gamma(1+s)
 (-\zeta)^s g_{w,w'}(q^s)ds,
$$
where
$$
 g_{w,w'}(q^s)=\frac{1}{q^s w-w'} \left(\frac{(q^sw;q)_\infty}{(w;q)_\infty}\right)^N \exp\bigg(\tau w
 (q^s-1)\bigg).
$$
The operator $K_\zeta$ is trace-class for all $\zeta\in \mathbb C\setminus \mathbb R_+$.
\end{theorem}

The appearance of the Fredholm determinant in Theorem \ref{Theorem_Fredholm} is unexpected. As we
already mentioned, there is no known structure of determinantal point process for the Macdonald
measure. The computation which leads to this determinant (see \cite{BC} for the proof) turns out
to be parallel to the ``shift contour argument'' in the harmonic analysis on Riemannian symmetric
spaces, which goes back to Helgason  and Heckman--Opdam, see \cite{HO} and references therein. The
main step of the proof is the modification of contours of integration in Theorem
\ref{theorem_Observable_contour} and careful residue book-keeping. In the end all the residues
miraculously combine into a single Fredholm determinant, which should be viewed as a manifestation
of the fact that we are working in the representation--theoretic framework.

\bigskip

Through the limit transition described in Section \ref{Section_Macdonald_probability} we can now
get some information about the partition function of the semi--discrete O'Connell--Yor polymer.
Namely, the limit of Theorem \ref{theorem_Observable_contour} gives moments of this partition
function and the limit of Theorem \ref{Theorem_Fredholm} gives the expectation of its Laplace
transform, see \cite{BC} for an exact statement.

Interestingly enough, the argument relating the Laplace transform to the moments no longer works
after the limit transition. The desired relation between the moments and the Laplace transform of
a random variable $\xi$ is
\begin{equation}
\label{eq_x12} \E \exp(u \xi) =\E \left(\sum_{k=0}^{\infty} \frac{ u^k
\xi^k}{k!}\right)=\sum_{k=0}^{\infty} \frac{u^k}{k!} \E \xi^k.
\end{equation}

However, the moments of the partition function of the O'Connell--Yor directed polymer grow rapidly
(as $e^{k^2}$, cf.\ \cite{BC-a}) and the series in the right side of \eqref{eq_x12} does not
converge for any $u\ne 0$. This is caused by the \emph{intermittency} which we discussed in
Section \ref{Section_Intro}. Similar things happen when one considers fully continuous polymer
which we briefly mentioned in Section \ref{Section_Intro}, i.e.\ when one integrated $2d$ white
noise over the paths of Brownian bridges. Nevertheless, physicists tried to overcome these
difficulties and find the Laplace transform (and the distribution itself after that) using the
moments (the latter could be computed in this model using the so--called Bethe ansatz for the
delta quantum Bose--gas). A non-rigorous argument used here is known as \emph{replica trick} and
it has a long history (first applied to directed polymers in random media in \cite{Kar}); this is
some sort of an analytic continuation argument for the function with specified values at integer
points. However, the first answers to this problem obtained via the replica trick were
\emph{incorrect}. (They were later corrected though, cf.\ \cite{Do}, \cite{CDR}.) A correct
mathematical computation of the Laplace transform for the continuous directed polymer appeared
concurrently in the work of Amir--Corwin--Quastel \cite{ACQ} (which is fully rigorous) and,
independently, of Sasamoto--Spohn \cite{SS}; it was based on previous work of Tracy--Widom on
ASEP, see \cite{TW-ASEP} and references therein.

Therefore, the manipulations with contours which deduce Theorem \ref{Theorem_Fredholm} from
Theorem \ref{theorem_Observable_contour} can be viewed as a fully rigorous $q$--incarnation (put
it otherwise, as a mathematical justification) of the non-rigorous replica trick. See \cite{BCS}
for further developments in this direction.

The asymptotic analysis of the operator $K$ in the Fredholm determinant of Theorem
\ref{Theorem_Fredholm} based on the steepest descent ideas that we discussed in Section
\ref{Section_Schur_measures}, leads to the Tracy--Widom distribution $F_2(s)$ (which is also given
by Fredholm determinant, as we remarked after Theorem \ref{Theorem_Edge}). Along these lines one
proves the KPZ--universality for the partition function of the O'Connell--Yor directed polymer,
and also of certain integrable discrete polymers (Theorem \ref{Theorem_polymer_intro}), see
\cite{BC}, \cite{BCF}, \cite{BCR}, \cite{BCS}.

\end{document}